\documentclass[12pt, reqno]{amsart}
\usepackage{amsmath, amssymb, amscd, amsthm, comment, amsxtra}
\usepackage{graphicx}
\usepackage{subfigure}
\usepackage{enumerate}
\usepackage[pdftex, linktocpage=true]{hyperref}
\usepackage{euscript}
\usepackage[format=plain,labelfont=bf,up,width=.99\textwidth]{caption}
\usepackage[toc,page]{appendix}
\usepackage{xcolor}

\theoremstyle{plain}
\newtheorem{thm}{Theorem}[section]
\newtheorem{lem}[thm]{Lemma}
\newtheorem{prop}[thm]{Proposition}
\newtheorem{cor}[thm]{Corollary}

\newtheorem*{thma}{Theorem A}
\newtheorem*{thmb}{Theorem B}
\newtheorem*{thmc}{Theorem C}
\newtheorem*{thmd}{Theorem D}
\newtheorem*{thme}{Theorem E}

\newtheorem*{apriori}{\emph{A Priori} Bounds}
\newtheorem*{finite}{Finite-time Checkability}
\newtheorem*{ren conv}{Renormalization Convergence}
\newtheorem*{reg unicrit}{Regular Unicriticality}

\theoremstyle{definition}
\newtheorem{defn}[thm]{Definition}
\newtheorem{ex}[thm]{Example}
\newtheorem{rem}[thm]{Remark}
\newtheorem{conj}{Conjecture}

\numberwithin{equation}{section}
\newcommand{\thmref}[1]{Theorem~\ref{#1}}
\newcommand{\propref}[1]{Proposition~\ref{#1}}
\newcommand{\lemref}[1]{Lemma~\ref{#1}}
\newcommand{\corref}[1]{Corollary~\ref{#1}}
\newcommand{\figref}[1]{Figure~\ref{#1}}
\newcommand{\secref}[1]{Section~\ref{#1}}
\newcommand{\subsecref}[1]{Subsection~\ref{#1}}
\newcommand{\appref}[1]{Appendix~\ref{#1}}
\newcommand{\remref}[1]{Remark~\ref{#1}}

\newcommand{\defnref}[1]{Definition~\ref{#1}}
\newcommand{\exref}[1]{Example~\ref{#1}}

\newcommand{\md}[1]{\;(\operatorname{mod}\; #1)}

\renewcommand{\epsilon}{\varepsilon}

\newcommand{\bbB}{\mathbb B}

\newcommand{\bbD}{\mathbb D}

\newcommand{\bbL}{\mathbb L}

\newcommand{\bbN}{\mathbb N}

\newcommand{\bbP}{\mathbb P}

\newcommand{\bbR}{\mathbb R}

\newcommand{\bbZ}{\mathbb Z}

\newcommand{\cA}{\mathcal A}
\newcommand{\cB}{\mathcal B}
\newcommand{\cC}{\mathcal C}
\newcommand{\cD}{\mathcal D}
\newcommand{\cE}{\mathcal E}

\newcommand{\cG}{\mathcal G}
\newcommand{\cH}{\mathcal H}
\newcommand{\cI}{\mathcal I}
\newcommand{\cJ}{\mathcal J}
\newcommand{\cK}{\mathcal K}
\newcommand{\cL}{\mathcal L}

\newcommand{\cO}{\mathcal O}
\newcommand{\cP}{\mathcal P}
\newcommand{\cQ}{\mathcal Q}
\newcommand{\cR}{\mathcal R}
\newcommand{\cS}{\mathcal S}
\newcommand{\cT}{\mathcal T}
\newcommand{\cU}{\mathcal U}
\newcommand{\cV}{\mathcal V}
\newcommand{\cW}{\mathcal W}

\newcommand{\cY}{\mathcal Y}
\newcommand{\cZ}{\mathcal Z}

\newcommand{\bfC}{\mathbf C}

\newcommand{\bfK}{\mathbf K}
\newcommand{\bfL}{\mathbf L}

\newcommand{\bfZ}{\mathbf Z}

\newcommand{\bfb}{\mathbf b}

\newcommand{\bfq}{\mathbf q}

\newcommand{\hB}{\hat B}

\newcommand{\hE}{\hat E}

\newcommand{\hH}{\hat H}

\newcommand{\hJ}{\hat J}

\newcommand{\hR}{\hat R}
\newcommand{\hS}{\hat S}

\newcommand{\ha}{\hat a}

\newcommand{\hf}{\hat f}

\newcommand{\hk}{\hat k}

\newcommand{\hm}{\hat m}

\newcommand{\hr}{\hat r}

\newcommand{\hx}{\hat x}
\newcommand{\hy}{\hat y}
\newcommand{\hz}{\hat z}

\newcommand{\tiB}{\tilde B}
\newcommand{\tiC}{\tilde C}

\newcommand{\tiE}{\tilde E}
\newcommand{\tiF}{\tilde F}

\newcommand{\tiH}{\tilde H}
\newcommand{\tiI}{\tilde I}
\newcommand{\tiJ}{\tilde J}

\newcommand{\tiS}{\tilde S}

\newcommand{\tic}{\tilde c}

\newcommand{\tif}{\tilde f}
\newcommand{\tig}{\tilde g}

\newcommand{\bepsilon}{{\bar \epsilon}}

\newcommand{\bL}{{\bar L}}

\newcommand{\chcB}{{\check{\cB}}}

\newcommand{\chF}{{\check{F}}}
\newcommand{\chf}{{\check{f}}}

\newcommand{\tixi}{\tilde \xi}

\newcommand{\tirho}{\tilde \rho}

\newcommand{\chPhi}{\check \Phi}

\newcommand{\chJ}{\check{J}}

\newcommand{\tiGamma}{\tilde \Gamma}

\newcommand{\tiPhi}{\tilde \Phi}

\newcommand{\chB}{\check B}

\newcommand{\bdelta}{{\bar \delta}}

\newcommand{\bK}{{\bar K}}

\newcommand{\udelta}{\underline{\delta}}
\newcommand{\uepsilon}{\underline{\epsilon}}

\newcommand{\bchi}{\bar{\chi}}

\newcommand{\ueta}{\underline{\eta}}
\newcommand{\baeta}{\bar\eta}

\newcommand{\chgamma}{\check{\gamma}}
\newcommand{\chcI}{\check{\cI}}
\newcommand{\chI}{\check{I}}
\newcommand{\ticI}{\tilde{\cI}}

\newcommand{\hcJ}{\hat{\cJ}}
\newcommand{\hcD}{\hat{\cD}}

\newcommand{\chl}{\check{l}}
\newcommand{\chH}{\check{H}}
\newcommand{\chcD}{\check{\cD}}

\newcommand{\hcB}{\hat{\cB}}

\newcommand{\ticB}{\tilde\cB}

\newcommand{\frU}{{\mathfrak{U}}}
\newcommand{\frQ}{{\mathfrak{Q}}}
\newcommand{\frH}{{\mathfrak{H}}}
\newcommand{\frF}{{\mathfrak{F}}}

\newcommand{\frN}{{\mathfrak{N}}}
\newcommand{\frHL}{{\mathfrak{H}\mathfrak{L}}}

\newcommand{\frd}{{\frd}}
\newcommand{\chh}{{\check h}}

\newcommand{\ticE}{{\tilde\cE}}

\newcommand{\hcH}{{\hat\cH}}
\newcommand{\chd}{{\check d}}
\newcommand{\frA}{{\mathfrak{A}}}

\newcommand{\frI}{{\mathfrak{I}}}
\newcommand{\frf}{{\mathfrak{f}}}

\newcommand{\bC}{{\bar C}}

\newcommand{\In}{\operatorname{in}}
\newcommand{\out}{\operatorname{out}}

\newcommand{\Jac}{\operatorname{Jac}}
\newcommand{\diam}{\operatorname{diam}}

\newcommand{\dist}{\operatorname{dist}}

\newcommand{\Id}{\operatorname{Id}}
\newcommand{\loc}{\operatorname{loc}}

\newcommand{\crit}{\operatorname{crit}}

\newcommand{\sign}{\operatorname{sign}}

\newcommand{\depth}{\operatorname{depth}}

\newcommand{\logl}{\log_\lambda}

\newcommand{\vd}{\operatorname{vd}}

\newcommand{\Dis}{\operatorname{Dis}}
\newcommand{\vm}{\operatorname{vm}}

\newcommand{\cRod}{\cR_{\operatorname{1D}}}

\newcommand{\Piod}{\Pi_{\operatorname{1D}}}
\newcommand{\Index}{\operatorname{Index}}

\newcommand{\matsp}[1]{\hspace{5mm} \text{#1} \hspace{5mm}}

\newcommand{\comma}{, \hspace{5mm}}

\title[On Regular H\'enon-like Renormalization]{On Regular H\'enon-like Renormalization}
\author{Jonguk Yang}

\begin{document}

\maketitle

\begin{abstract}
\vspace{-0.75cm}
We develop a renormalization theory of non-perturbative dissipative H\'enon-like maps with  combinatorics of bounded type. The main novelty of our approach is the incorporation of Pesin theoretic ideas to the renormalization method, which enables us to control the small-scale geometry of dynamics in the higher-dimensional setting. In \cite{CLPY3}, it is shown that, under certain regularity conditions on the return maps, renormalizations of H\'enon-like maps have {\it a priori} bounds. The current paper is devoted to the applications of this critical estimate. First, we prove that H\'enon-like maps converge under renormalization to the same renormalization attractor as for 1D unimodal maps. Second, we show that the necessary and sufficient conditions for renormalization convergence are finite-time checkable. Lastly, we show that every infinitely renormalizable H\'enon-like map is {\it regularly unicritical}: there exists a unique orbit of tangencies between strong-stable and center manifolds, and outside a slow-exponentially shrinking neighborhood of this orbit, the dynamics behaves as a uniformly partially hyperbolic system.\vspace{-0.75cm}
\end{abstract}

\tableofcontents


\section{Introduction}\label{sec.intro}

The set of real quadratic polynomials, after normalization, can be represented as the following one-parameter family:
\begin{equation}\label{eq.quad fam}
\frQ := \{f_a(x) := x^2+a \; | \; a\in\bbR\}.
\end{equation}
We refer to $\frQ$ as the {\it (real) quadratic family}. Despite its elementary form, the dynamics in $\frQ$ turns out to be incredibly rich and fascinatingly complicated. In fact, the study of this family has been a focal point in the field of one-dimensional dynamics for nearly three decades (see e.g. \cite{L2}).

At the heart of this topic lies the renormalization theory of unimodal maps, which analyzes the appearance of small-scale self-similarity in these systems. It was first introduced to the subject independently by Feigenbaum \cite{Fe} and Coullet--Tresser \cite{CoTr} in the mid 1970's. They observed that under successive renormalizations (or ``zoom-ins''), the small-scale dynamics of a unimodal map asymptotically approach a {\it universal} limit sequence that only depends on the combinatorial type of the original system. As a conjectural explanation of this phenomenon, they proposed that renormalization can be viewed as an operator acting on the space of unimodal maps, and that the set of renormalization limits form a hyperbolic attractor $\frA$ for this operator. A rigorous mathematical proof of this conjecture was completed in 1999, through the combined efforts of Sullivan \cite{Su}, McMullen \cite{Mc} and Lyubich \cite{L1}.

In dimension two, the role of the quadratic family is assumed by the {\it H\'enon family}:
\begin{equation}\label{eq.henon fam}
\frH := \{F_{a,b}(x,y) := (x^2+a-by, x) \; | \; a,b \in \bbR\}.
\end{equation}
The elements in $\frH$ are referred to as {\it H\'enon maps}. We identify $\frQ$ with the line $b = 0$ consisting of degenerate H\'enon maps. This way, $\frH$ can be viewed as a two-dimensional extension of $\frQ$. A H\'enon map $F_{a,b}$ is said to be {\it perturbative} if $|b|\ll 1$, as it can be obtained by making a small 2D perturbation to the 1D system $F_{a,0} \sim f_a$.

The H\'enon maps were introduced by H\'enon in 1969 as particularly simple systems that, based on numerical experiments, appeared to feature {\it strange attractors}: attractors in the phase space that exhibit chaotic behaviors \cite{He}. Since then, these maps have been some of the most widely studied examples in two-dimensional dynamics. For some particularly notable results about H\'enon maps, see \cite{BeCa}, \cite{WaYo}, \cite{PaYo}, \cite{Ber}, \cite{BeMaPa} and \cite{BoSt}. However, despite these remarkable developments, the dynamics in the H\'enon family remains deeply enigmatic, and unravelling its dynamics continues to be a wide open area of research.

Given the enormous success of renormalization theory in the study of the quadratic family, it is natural to try and extend this technique to the two-dimensional setting. The main conjecture concerning renormalization of H\'enon maps was proposed by Gambaudo--Tresser in the late 1980's \cite{GaTr}. We state it below in greater generality as two separate conjectures.

\begin{conj}
Let $f_{a_*} \in \cQ$ be an infinitely renormalizable quadratic polynomial with bounded type combinatorics. Then there exists a real analytic curve $\gamma(b) = (a(b), b)$ for $b \in [0, 1)$ extending from $\gamma(0) := (a_*, 0)$ such that the H\'enon map $F_{\gamma(b)} \in \cH$ is infinitely renormalizable, and converges to the same universal renormalization limit as $f_{a_*}$ under renormalization.
\end{conj}

\begin{conj}
The {\it boundary of chaos} $\partial_{\operatorname{chaos}}$ (the boundary of the set of maps with zero entropy) in $\cH$ is a piece-wise real analytic curve consisting of infinitely renormalizable H\'enon maps.
\end{conj}

In the perturbative regime, these conjectures can be verified for bounded type combinatorics using the following argument. Since the 1D renormalization attractor $\frA$ is hyperbolic, infinitely renormalizable quadratic polynomials converge to $\frA$ under renormalization in a robust way. Using this fact, one can show that the 1D renormalization convergence in the quadratic family extends to nearby 2D infinitely renormalizable H\'enon maps. This argument has been applied to the period-doubling case by Coullet-Eckmann-Koch \cite{CoEcKo}, Gambaudo--Treser--van Strien \cite{GavSTr} and De Carvalho--Lyubich--Martens \cite{DCLMa}; and to the stationary case by Hazard \cite{Ha}.

The goal of this paper is to extend the 1D renormalization theory of unimodal maps to a {\it non-perturbative} 2D setting. A natural 2D analogue of a unimodal map is given by a {\it H\'enon-like map}: a diffeomorphism $F : D \to F(D) \Subset D$ of the form $F(x, y) = (f(x,y), x)$ defined on a rectangle $0 \in D \subset \bbR^2$, such that for any fixed $y$, the 1D map $f(\cdot, y)$ is unimodal. One may visualize the action of $F$ as bending $D$ into a U-shape, and then turning it on its side. See \figref{fig.henontrans}. We refer to $\Piod(F)(x) := f(x, 0)$ as the {\it 1D profile of $F$}.

\begin{figure}[h]
\centering
\includegraphics[scale=0.45]{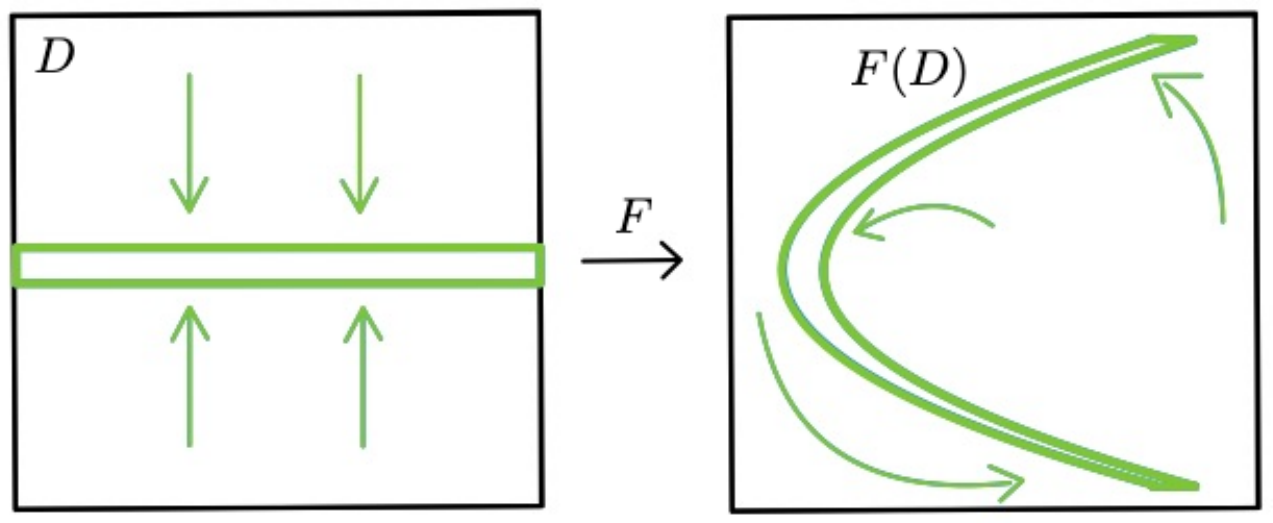}
\caption{H\'enon-like mapping.}
\label{fig.henontrans}
\end{figure}

The H\'enon-like map $F$ is {\it (H\'enon-like) renormalizable} if for some integer $R \geq 2$, there is an $R$-periodic subdomain $\cB^1 \Subset D$, and the return map $F^R|_{\cB^1}$ is again H\'enon-like after a smooth change-of-coordinates $\Phi : \cB^1 \to D^1$. In this case, the map $\Phi$ and the pair $(F^R, \Phi)$ are referred to as a {\it straightening chart} and a {\it H\'enon-like return} respectively. We define the {\it (H\'enon-like) renormalization} $\cR(F)$ of $F$ as the H\'enon-like map obtained via a suitable affine rescaling of $\Phi\circ F^R \circ \Phi^{-1}$ that normalizes the width of the domain $D^1$. See \figref{fig.henonren}. Lastly, a {\it centered straightening chart} $\Psi : \cB^1 \to B^1$ is a chart that induces the same vertical and horizontal foliations over $\cB^1$, and preserves the arclengths along the vertical and horizontal lines through the origin in $B^1$.

\begin{figure}[h]
\centering
\includegraphics[scale=0.25]{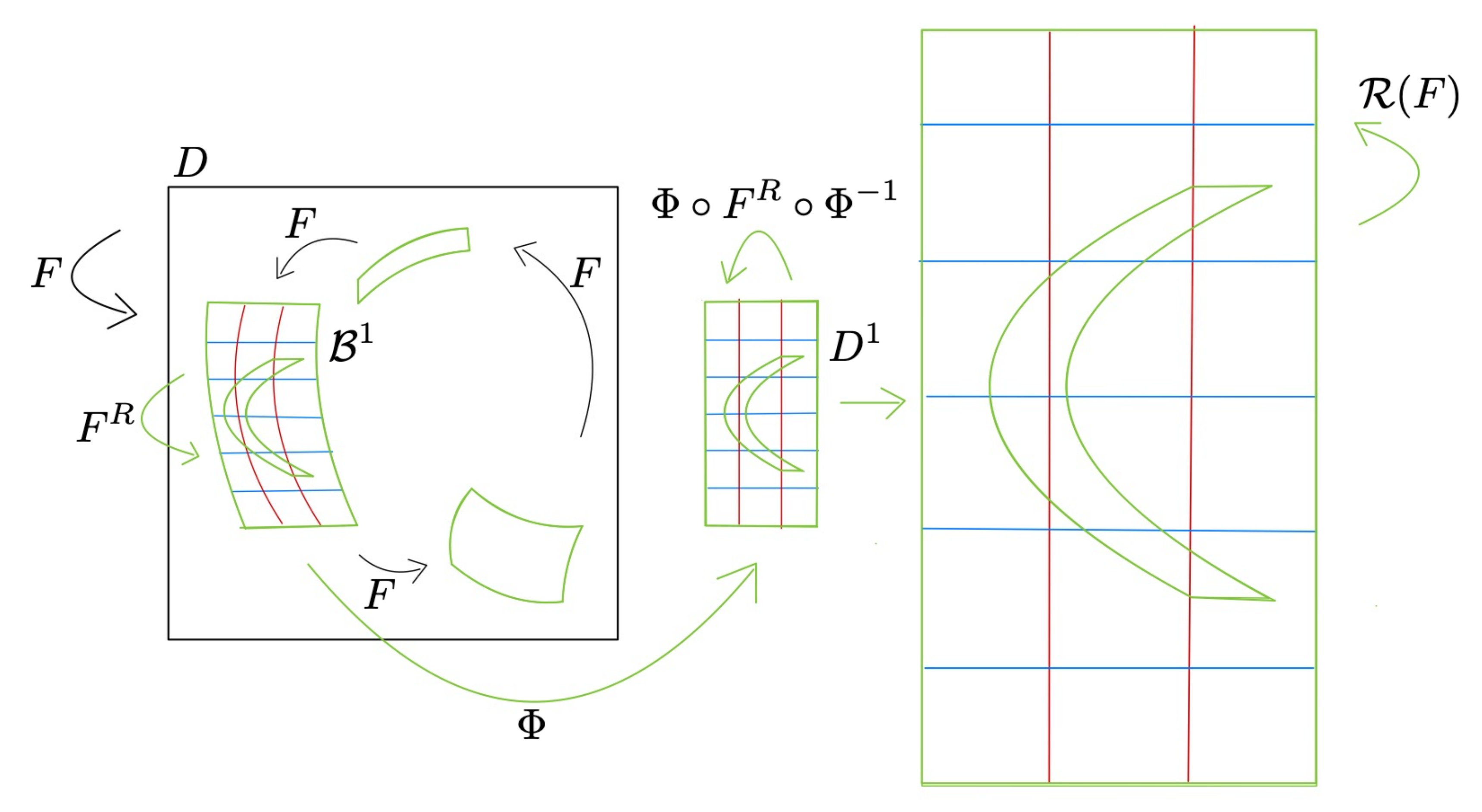}
\caption{H\'enon-like renormalization.}
\label{fig.henonren}
\end{figure}
 
A key novelty of our approach is the incorporation of Pesin theoretic ideas to the renormalization method. This involves keeping track of the {\it regularity} of points, which can then be used to control the geometry of dynamics in the higher-dimensional setting (see \appref{sec.pesin}). We give loose definitions of these notions below. For the precise definitions, see \subsecref{subsec.regular}.

Let $p$ be a point in $D$, and let $E_p$ be a tangent direction at $p$.  For $M \in \bbN \cup \{\infty\}$, we say that $p$ is {\it $M$-times forward regular along $E_p$} if there is sufficiently dominant exponential contraction along $E_p$ under $DF^m$ for $1 \leq m \leq M$. Similarly, $p$ is {\it $M$-times backward regular along $E_p$} if there is sufficiently dominant exponential expansion along $E_p$ under $DF^{-m}$ for $1 \leq m \leq M$.

Consider the H\'enon-like return $(F^R, \Phi)$. For $p \in \cB^1$, let $E_p^v$ and $E_p^h$ be the tangent directions at $p$ that are mapped by $D\Phi$ to the genuine vertical and horizontal directions respectively. We say that $(F^R, \Phi)$ is {\it regular}, and that $F$ is {\it regularly H\'enon-like renormalizable} if 
\begin{enumerate}[i)]
\item every $p \in \cB^1$ is $R$-times forward regular along $E^v_p$;
\item every $q \in F^{R}(\cB^1)$ is $R$-times backward regular along $E^h_q$; and
\item at every $p \in \cB^1$, the tangent directions $E_p^v$ and $E_p^h$ are uniformly transversal.
\end{enumerate}

Under this regularity assumption, it is established in \cite{CLPY3} the following uniform control on the small-scale geometry of the dynamics of a H\'enon-like map that holds at all renormalization depths. An estimate of this kind is commonly referred to as {\it a priori} bounds, and is typically the key ingredient needed to develop a functioning renormalization theory. The precise version is stated in \thmref{a priori}.

\begin{apriori}\cite{CLPY3}
Suppose for some $N \in \bbN \cup \{\infty\}$, a H\'enon-like map $F$ is $N$-times regularly H\'enon-like renormalizable with bounded type combinatorics. Then for all $1 \leq n \leq N$, the distortion along the horizontal direction of the $n$th return map is uniformly bounded.
\end{apriori}

A priori bounds has far-reaching consequences for renormalization of H\'enon-like maps, which we summarize below as three main results. They are given informally here in order to better convey their conceptual meaning to the readers. For their precise statements, see \secref{sec.main thm}.

The first main result describes the asymptotics of H\'enon-like maps under regular H\'enon-like renormalization. The precise version is stated as Theorem A in \secref{sec.main thm}.

\begin{ren conv}
Suppose a H\'enon-like map $F$ is infinitely regularly H\'enon-like renormalizable with bounded type combinatorics. Then the following statements hold. 
\begin{enumerate}[i)]
\item The centered straightening charts for the renormalizations of $F$ converge super-exponentially fast.
\item The renormalizations of $F$ converge to the space of 1D systems (i.e. their dependence on the second coordinate goes to zero) super-exponentially fast.
\item The 1D profiles of the renormalizations of $F$ converge to the 1D renormalization attractor for unimodal maps exponentially fast.
\end{enumerate}
\end{ren conv}

The second main result addresses the problem of guaranteeing the existence of infinitely regularly H\'enon-like renormalizable maps. It is actually a combination of two theorems: Theorem B and Theorem C in \secref{sec.main thm}. See also \remref{rem.finite}.

\begin{finite}
For bounded type combinatorics, infinite regular H\'enon-like renormalizability is a finite-time checkable condition.
\end{finite}

Applying the previous two main results to the H\'enon family $\frH$, it becomes theoretically possible to verify numerically if the curves of infinitely renormalizable H\'enon maps extend  arbitrarily close to $b=1$ (although the computations involved would become infinitely complex as the Jacobian gets closer and closer to $1$). See Examples \ref{ex.henon 0} and \ref{ex.henon 1} for more details.

Renormalization convergence gives us extremely precise information about what the dynamics of a H\'enon map looks like when it is ``zoomed-in'' at a certain point (which we later identify as the {\it critical value} of the map). The last main result concerns the global geometry of the dynamics over the entire renormalization limit set. A more detailed version of this result is stated in Theorem E in \secref{sec.main thm}.

\begin{reg unicrit}
An infinitely regularly H\'enon-like renormalizable map with bounded type combinatorics is \emph{regularly unicritical} on its renormalization limit set.
\end{reg unicrit}

The notion of regular unicriticality is introduced and studied in \cite{CLPY1}. It is a new type of axiomatic dynamics defined on uniquely ergodic sets. Loosely speaking, it means that the system features a unique {\it critical orbit}: an orbit of tangencies between strong-stable and center manifolds. Moreover, outside a slow-exponentially shrinking neighborhood of this orbit, every point is uniformly regular. See \defnref{def unicrit}. In \cite{CLPY1}, it is shown that, despite the presence of the critical orbit adding highly non-linear ``bends'' in the picture, the global geometry of the dynamics for a regularly unicritical system can be understood almost as explicitly as for a uniformly partially hyperbolic system.

\subsection{Notations and conventions}

Unless otherwise specified, we adopt the following notations and conventions. 

Any diffeomorphism on a domain in $\bbR^2$ is assumed to be orientation-preserving. The projective tangent space at a point $p \in \bbR^2$ is denoted by $\bbP^2_p$.

Given a number $\eta > 0$, we use $\baeta$ to denote any number that satisfy
$$
\eta < \baeta < C\eta^D
$$
for some uniform constants $C > 1$ and $D > 1$ (if $\eta >1$) or $D \in (0, 1)$ (if $\eta < 1$) that are independent of the map being considered. Additionally, we allow $\baeta$ to absorb any uniformly bounded coefficient or power. So for example, if $\baeta >1$, then we may write
$$
\text{``}\;\; 10\baeta^5 = \baeta\;\;\text{''}.
$$
Similarly, we use $\ueta$ to denote any number that satisfy
$$
c\eta^d < \ueta < \eta
$$
for some uniform constants $c \in (0, 1)$ and $d \in (0, 1)$ (if $\eta >1$) or $d >1$ (if $\eta <1$) that are independent of the map being considered. As before, we allow $\ueta$ to absorb any uniformly bounded coefficient or power. So for example, if $\ueta >1$, then we may write
$$
\text{``}\;\;\tfrac{1}{3}\ueta^{1/4} = \ueta\;\;\text{''}.
$$
These notations apply to any positive real number: e.g. $\bepsilon > \epsilon$, $\udelta < \delta$, $\bL > L$, etc.

Note that $\baeta$ can be much larger than $\eta$ (similarly, $\ueta$ can be much smaller than $\eta$). Sometimes, we may avoid the $\beta$ or $\ueta$ notation when indicating numbers that are somewhat or very close to the original value of $\eta$. For example, if $\eta \in (0,1)$ is a small number, then we may denote $\eta':=(1-\baeta)\eta$. Then $\ueta\ll \eta' < \eta$.

We use $n, m, i, j$ to denote integers (and less frequently $l, k$). The letter $i$ is never the imaginary number. Typically (but not always), $n \in\bbN$ and $m \in \bbZ$. We typically use $N, M$ to indicate fixed integers (often related to variables $n, m$).

We typically denote constants used for estimate bounds by $C, K \geq 1$ (less frequently $c > 0$).

We use calligraphic font $\cU, \cT, \cI,$ etc, for objects in the phase space and regular fonts $U, T, I,$ etc, for corresponding objects in the linearized/uniformized coordinates. A notable exception are for the invariant manifolds $W^{ss}, W^c$.

We use $p, q$ to indicate points in the phase space, and $z, w$ for points in linearized/uniformized coordinates.

For any set $X_m \subset \Omega$ with a numerical index $m \in \bbZ$, we denote
$$
X_{m+l} := F^l(X_m)
$$
for all $l \in \bbZ$ for which the right-hand side is well-defined. Similarly, for any direction $E_{p_m} \in \bbP^2_{p_m}$ at a point $p_m \in \Omega$, we denote
$$
E_{p_{m+l}} := DF^l(E_{p_m}).
$$

Define
$$
\pi_h(x,y) := x
\comma
\pi_v(x,y) := y
\comma
\Pi_h(x,y) := (x,0)
\matsp{and}
\Pi_v(x,y) := (0,y).
$$

\section{Preliminaries}\label{sec.prelim}

\subsection{Renormalization of unimodal maps}\label{subsec.unimodal}

Let $I \subset \bbR$ be an interval. A $C^2$-map $f : I \to I$ is {\it unimodal} if it has a unique critical point $c \in I$, which of quadratic type: i.e. $f'(c) = 0$ and $f''(c) \neq 0$. Denote the critical value of $f$ by $v := f(c)$. We say that $f$ is {\it normalized} if $c = 0$ and $f''(c) = 2$. Let $\gamma \in \{r, \omega\}$, where $r \geq 2$ is an integer. The space of normalized $C^\gamma$-unimodal maps is denoted $\frU^\gamma$.

A unimodal map $f : I \to I$ is {\it topologically renormalizable} if there exists an {\it $R$-periodic} subinterval $I^1 \subset I$:
$$
f^i(I^1) \cap I^1 = \varnothing
\matsp{for}
1 \leq i < R
\matsp{and}
f^R(I^1) \Subset I^1.
$$
We say that $f$ is {\it (valuably) renormalizable} if $f^R(I^1)$ contains the critical value $v$.

If $f$ is valuably renormalizable, then the return map $f^R|_{I^1}$ is also unimodal. We define the renormalization of $f$ to be
$$
\cRod(f) := S\circ f^R|_{I^1}\circ S^{-1},
$$
where $S : \bbR \to \bbR$ is the unique affine map such that $\cRod(f) \in \frU^\gamma$.

\subsection{H\'enon-like maps}

Let $D := I \times J \subset \bbR^2$ be a rectangle, where $0 \in I \Subset J \subset \bbR$ are intervals. A $C^2$-diffemorphism $F : D \to F(D) \Subset D$ is {\it H\'enon-like} if $F$ is of the form
\begin{equation}\label{eq.henonlike}
F(x,y) = (f(x,y), x)
\matsp{for}
(x,y) \in D,
\end{equation}
such that for any $y \in J$, the map $f(\cdot,y) : I \to I$ is a unimodal map. We say that $F$ is {\it normalized} if $f(\cdot, 0)$ is normalized. The set of normalized $C^\gamma$-H\'enon-like maps is denoted $\frHL^\gamma$.

For $\beta \in (0, 1]$, we say that $F$ is {\it $\beta$-thin (in $C^\gamma$)} if
$$
\|\partial_y f\|_{C^{\gamma-1}} < \beta.
$$
The space of $\beta$-thin H\'enon-like maps in $\frHL^\gamma$ is denoted $\frHL^\gamma_\beta$.

For any 1D map $g : I \to I$, define its {\it 2D embedding} $\iota(g) : I\times \bbR \to I\times \bbR$ by
\begin{equation}\label{eq.embed}
\iota(g)(x,y) := (g(x), x).
\end{equation}
For any 2D map $G : D \to D$, define its {\it 1D profile} $\Piod(G) : I \to I$ by
\begin{equation}\label{eq.1d profile}
\Piod(G)(x) := \pi_h\circ G(x, 0).
\end{equation}
Note that we have $\Piod \circ \iota(g) = g$.

The space of degenerate $C^\gamma$-H\'enon-like maps is given by $\frHL^\gamma_0 := \iota(\frU^\gamma)$. A map $F \in \frHL^\gamma_\beta$ is said to be {\it perturbative} if $\beta \ll 1$, as it can be obtained by making a small 2D perturbation to a 1D system in $\frHL^\gamma_0$.

\subsection{Charts}

For $z \in \bbR^2$, let $E^{gv}_z, E^{gh}_z\in \bbP^2_z$ denote the genuine vertical and horizontal directions at $z$ respectively.

A {\it $C^r$-chart} is a $C^r$-diffeomorphism $\Phi : \cB \to B$ from a quadrilateral $\cB \subset \bbR^2$ to a rectangle $B = I \times J \subset \bbR^2$, where $I, J \subset \bbR$ are intervals. The {\it vertical/horizontal direction $E^{v/h}_p \in \bbP^2_p$ at $p\in\cB$ (associated to $\Phi$)} are given by
$$
E^{v/h}_p := D\Phi^{-1}\left(E^{gv/gh}_{\Phi(p)}\right).
$$
The chart $\Phi$ is said to be {\it genuinely vertical/horizontal} if $E^{v/h}_p = E^{gv/gh}_p$ for all $p \in\cB$.

A {\it vertical leaf in $\cB$} is a curve $l^v$ such that
$$
l^v \subseteq \Phi^{-1}(\{a\} \times \pi_v(B))
\matsp{for some}
a \in \pi_h(B).
$$
If the above containment is an equality, then $l^v$ is said to be {\it full}. A {\it (full) horizontal leaf $l^h$ in $\cB$} is defined analogously.

Let $p \in \cB$ and $E_p \in \bbP^2_p$. Denote
$
z := \Phi(p)
$ and $
E_z := D\Phi(E_p).
$
For $t > 0$, the direction $E_p$ is said to be {\it $t$-vertical in $\cB$} if 
$$
\frac{\measuredangle(E_z, E_z^{gv})}{\measuredangle(E_z, E_z^{gh})} < t.
$$
A {\it $t$-horizontal direction in $\cB$} is analogously defined.

A $C^0$-curve $\Gamma^v \subset \cB$ is said to be {\it vertical in $\cB$} if $\Phi(\Gamma^v)$ is a vertical graph in $B$ in the usual sense. That is, there exists an interval $I^v \subseteq \pi_v(B)$ and a map $g_v : I^v \to \pi_h(B)$ such that
$$
\Phi(\Gamma^v) = \cG^v(g_v) := \{(g_v(y), y) \; | \; y \in I^v\}.
$$
If $I^v = \pi_v(B)$, then $\Gamma^v$ is said to be {\it vertically proper in $\cB$}. A {\it horizontal} or a {\it horizontally proper curve $\Gamma^h$ in $\cB$} is analogously defined. If $\Gamma^v$ is $C^r$, and $\|g_v'\|_{C^{r-1}} \leq t$ for some $t \geq 0$, then we say that $\Gamma^v$ is {\it $t$-vertical (in $C^r$) in $\cB$}. Note that $\Gamma^v$ is a (vertically proper) $0$-vertical curve if and only if it is a (full) vertical leaf.

If $\Gamma^v$ is $C^2$, and $g_v$ has a unique critical point $c \in I^v$ of quadratic type: $g_v'(c) =0$ and
\begin{equation}\label{eq.val curv}
\kappa_\Phi(\Gamma^v) := g_v''(c) \neq 0,
\end{equation}
then $\Gamma^v$ is a {\it vertical quadratic curve in $\cB$}. We refer to $\kappa_\Phi(\Gamma^v)$ as the {\it valuable curvature of $\Gamma^v$ in $\cB$}.

Let $\cE^v : \cB \to T^1(\cB)$ be the $C^{r-1}$-unit vector field given by
$$
\cE^v(p) := D\Phi^{-1}(E^{gv}_{\Phi(p)}).
$$
A $C^{r-1}$-unit vector field $\ticE^v : \cU \to T^1(\cU)$ defined on a domain $\cU \subset \cB$ is said to be {\it $t$-vertical in $C^{r-1}$ in $\cB$} for some $t \geq 0$ if $\|\ticE^v - \cE^v\|_{C^{r-1}} \leq t$.

Let $\tiPhi : \ticB \to \tiB$ be another chart with $\ticB \subset \cB$. We define the following relations between $\Phi$ and $\tiPhi$.
Let $\tiPhi : \ticB \to \tiB$ be another chart with $\ticB \subset \cB$. We define the following relations between $\Phi$ and $\tiPhi$.
\begin{itemize}
\item We say that $\ticB$ is {\it vertically proper in $\cB$} if every full vertical leaf in $\ticB$ is vertically proper in $\cB$.
\item We say that $\Phi$ and $\tiPhi$ are {\it horizontally equivalent on $\ticB$} if every horizontal leaf in $\ticB$ is a horizontal leaf in $\cB$.
\item For $t \geq 0$, we say that {\it $\ticB$ is $t$-vertical in $\cB$} if $\Phi$ and $\tiPhi$ are horizontally equivalent, and the unit vector field given by
$$
\ticE^v(p) := D\tiPhi^{-1}(E^{gv}_{\tiPhi(p)})
\matsp{for}
p \in \ticB
$$
is $t$-vertical in $C^{r-1}$ in $\cB$.
\item We say that $\Phi$ and $\tiPhi$ are {\it equivalent on $\ticB$} if $\ticB$ is $0$-vertical in $\cB$.
\end{itemize}

Let $\Psi : \cB \to B$ be a chart satisfying the following properties.
\begin{itemize}
\item There exists $q \in \cB$ such that $\Psi(q) = 0 \in B$.
\item Let
$$
\cI^h(t) := \Psi^{-1}(t, 0)
\matsp{for}
t \in \pi_h(B),
$$
and
$$
\cI^v(s) := \Psi^{-1}(0, s)
\matsp{for}
s \in \pi_v(B).
$$
Then $\|(\cI^{h/v})'\| \equiv 1$.
\end{itemize}
In this case, we say that $\Psi$ is {\it centered (at $q$)}. Clearly, for any chart $\Phi : \cB \to D$ and any point $q \in \cB$, there exists a unique chart $\Psi :\cB \to B$ equivalent to $\Phi$ that is centered at $q$.

Suppose that $\Psi : (\cB, q) \to (B, 0)$ is centered at some point $q \in \cB$. Let $\Gamma^h \subset \cB$ be a horizontal $C^r$-curve, so that $\Psi(\Gamma^h)$ is the horizontal graph in $B$ of a $C^r$-map $g_h : I^h \to \pi_v(B)$ defined on an interval $I^h \subset \pi_h(B)$. We say that $\Gamma^h$ is {\it $t$-horizontal in $C^r$ in $\cB$} if $\|g_h\|_{C^r} \leq t$. In particular, $\Gamma^h$ is $0$-horizontal in $\cB$ if and only if $\Gamma^h$ is a subarc of the full horizontal leaf containing $q$.

\subsection{H\'enon-like renormalization}

Consider a $C^{r+1}$-H\'enon-like map $F : D \to D$. We say that $F$ is {\it topologically renormalizable} if there exists an $R$-periodic Jordan domain $\cB^1 \Subset D$ for some integer $R \geq 2$:
$$
F^i(\cB^1) \cap \cB^1 = \varnothing
\matsp{for}
1 \leq i < R
\matsp{and}
F^R(\cB^1) \Subset \cB^1.
$$
If, additionally, $\cB^1$ contains $(v,0)$, where $v$ is the critical value of the unimodal map $\Piod(F)$, and there exists a genuinely horizontal $C^r$-chart $\Phi$ from $\cB^1$ to a rectangle $D^1 \subset \bbR^2$ such that the map $\tiF_1 := \Phi\circ F\circ \Phi^{-1}$
is again H\'enon-like, then $F$ is said to be {\it (H\'enon-like) renormalizable}. In this case, any chart $\Psi : \cB^1 \to B^1$ equivalent to $\Phi$ is referred to as a {\it straightening chart}, and the pair $(F^R, \Psi)$ is referred to as a {\it H\'enon-like return}.

Denote $\tif_1 := \Piod (\tiF_1)$, and let $S : \bbR \to \bbR$ be the unique affine map such that $S\circ \tif_1 \circ S^{-1} \in \frU^\gamma$. Define $\cS$ as the affine map on $\bbR^2$ given by $\cS(x,y) := (S(x), S(y)).$ The {\it (H\'enon-like) renormalization of $F$} is
$$
\cR(F) := \cS \circ \Phi\circ F^R \circ (\cS \circ \Phi)^{-1}.
$$
Observe that $\cR(F) \in \frHL^r$.

\begin{rem}
Note the loss of one degree of smoothness from $F$ (which is $C^{r+1}$) to $\Phi$ and $\cR(F)$ (which are $C^r$). This is to account for the loss of smoothness in the construction of regular charts given in \thmref{reg chart}. This is not a critical issue, since it forces the loss of only one degree of smoothness no matter how many times $F$ is renormalized.
\end{rem}

\subsection{Definition of regularity}\label{subsec.regular}

Consider a $C^r$-diffeomorphism $F : D \to F(D) \Subset D$ defined on a domain $D\subset \bbR^2$. Let $L \geq 1$; $\epsilon, \lambda \in (0,1)$  and $M \in \bbN \cup \{\infty\}$. A point $p \in D$ is {\it $M$-times forward $(L, \epsilon, \lambda)$-regular along $E_p^+ \in \bbP^2_p$} if for $s \in \{0, 1\}$, we have
\begin{equation}\label{eq.for reg}
L^{-1}\lambda^{(1+\epsilon)m}\leq \frac{\|DF^m|_{E^+_p}\|^{s+1}}{(\Jac_pF^m)^s} \leq L \lambda^{(1-\epsilon)m}
\matsp{for all}
1 \leq m \leq M.
\end{equation}
Similarly, $p$ is {\it $M$-times backward $(L, \epsilon, \lambda)$-regular along $E^-_p \in \bbP^2_p$} if for $s \in \{0, 1\}$, we have
\begin{equation}\label{eq.back reg}
L^{-1}\lambda^{(1+\epsilon)m}\leq \frac{(\Jac_pF^{-m})^s}{\|DF^{-m}|_{E^-_p}\|^{s+1}} \leq L \lambda^{(1-\epsilon)m}
\matsp{for all}
1 \leq m \leq M.
\end{equation}
The constants $L$, $\epsilon$ and $\lambda$ are referred to as an {\it irregularity factor}, a {\it marginal exponent} and a {\it contraction base} respectively.

There exists a uniform constant $\epsilon_1 \in (0,1)$ independent of $F$ such that if \eqref{eq.for reg} (or \eqref{eq.back reg} resp.) holds with $\epsilon \leq \epsilon_1$, then the local dynamics of $F$ near the forward (or backward resp.) orbit of $p$ can be quasi-linearized up to the $M$th iterate (see \thmref{reg chart}). If $M = \infty$, this implies in particular that $p$ has a well-defined $C^r$-smooth strong-stable manifold $W^{ss}(p)$ (or center manifold $W^c(p)$ resp.). It should be noted that the center manifold at an infinitely backward regular point $p$ is not uniquely defined. However, its $C^r$-jet at $p$ is unique (see \thmref{center jet}). Henceforth, any marginal exponent will be assumed to be less than $\epsilon_1$.

\subsection{Regular H\'enon-like returns}\label{subsec.reg ret}

A H\'enon-like return $(F^R, \Psi : \cB^1 \to B^1)$ is said to be {\it $(L, \epsilon, \lambda)$-regular} if the following conditions hold. Let
$$
E_p^{v/h} := D\Psi^{-1}\left(E^{gv/gh}_{\Psi(p)}\right).
$$
\begin{itemize}
\item Every $p \in \cB^1$ is $R$-times forward $(L, \epsilon, \lambda)$-regular along $E^v_p$.
\item Every $q \in F^R(\cB^1) \Subset \cB^1$ is $R$-times backward $(L, \epsilon, \lambda)$-regular along $E^h_q$.
\item For any $p \in \cB^1$, we have $\measuredangle(E^v_p, E^h_p) > 1/L$.
\end{itemize}
In this case, we say that $F$ is {\it $(L, \epsilon, \lambda)$-regularly H\'enon-like renormalizable}.

\begin{ex}\label{ex.perturb}
Let $f : I \to I$ be a $C^{r+1}$-unimodal map. Suppose $f$ is valuably renormalizable: there exists an $R$-periodic subinterval $I^1 \subset I$ such that $f^R(I^1)$ contains the critical value $v$ of $f$. Then for $\epsilon > 0$, there exists $\beta = \beta(f, R, \epsilon) > 0$ such that the following holds. Let $F : D \to D$ be a $\beta$-thin $C^{r+1}$-H\'enon-like map defined on a rectangle $D := I \times J$ with $\Piod(F) = f$. Then there exists an $R$-periodic quadrilateral $\cB^1 \subset D$ containing $(v,0)$ that is $\beta^{1-\epsilon}$-close to $I^1 \times J$ in the Hausdorff topology, and a $C^r$-chart $\Psi : \cB^1 \to B^1$ centered at $(v,0)$ that is $\beta^{1-\epsilon}$-close to the identity in the $C^r$-topology such that $(F^R, \Psi)$ is a $(1, \epsilon, \beta)$-regular H\'enon-like return. See \propref{2d ren from 1d ren}.
\end{ex}

\subsection{Nested H\'enon-like returns}\label{subsec.nested}

A $C^{r+1}$-H\'enon-like map $F : D \to D$ is $N$-times topologically renormalizable for some $N \in \bbN \cup \{\infty\}$ if there exist sequences
$$
D =: \cB^0 \Supset \cB^1\Supset \ldots
\matsp{and}
1 =: R_0 < R_1 < \ldots
$$
such that for $1 \leq n \leq N$, the set $\cB^n$ is an $R_n$-periodic Jordan domain. If there exists $\bfb \geq 2$ such that
$$
r_{n-1} := R_n/R_{n-1} \leq \bfb
\matsp{for all}
1\leq n \leq N,
$$
then we say that the combinatorics of renormalization for $F$ is {\it of ($\bfb$-)bounded type}. If $N = \infty$, then the {\it renormalization limit set} for $F$ is
\begin{equation}\label{eq.limit set}
\Lambda_F := \bigcap_{n=1}^\infty \bigcup_{i=0}^{R_n-1} F^{R_n+i}(\cB^n).
\end{equation}

Suppose for $1 \leq n \leq N$, there exist a $C^r$-straightening chart $\Psi^n:\cB^n \to B^n$ such that $(F^{R_n}, \Psi^n)$ is a H\'enon-like return. Then the sequence
\begin{equation}\label{eq.returns}
\{(F^{R_n}, \Psi^n:\cB^n \to B^n)\}_{n=1}^N
\end{equation}
said to be {\it nested}. Without loss of generality, we may assume that $\Psi^n$ is a {\it centered straightening chart}: that is, $\Psi^n$ is centered at some common point
$$
v_0 \in \bigcap_{n=1}^N F^{R_n}(\cB^n).
$$

Let $\Phi^n : \cB^n \to D^n$ be a chart equivalent to $\Psi^n$ such that $\tiF_n := \Phi^n \circ F^{R_n}\circ (\Phi^n)^{-1}$ is H\'enon-like. Denote $\tif_n := \Piod(\tiF_n)$, and let $S^n : \bbR \to \bbR$ be the unique affine map such that $S^n\circ \tif_n \circ (S^n)^{-1} \in \frU^r$. Define $\cS^n$ as the affine map on $\bbR^2$ given by $\cS^n(x,y) := (S^n(x), S^n(y))$. The $n$th (H\'enon-like) renormalization of $F$ is given by
\begin{equation}\label{eq.renorm}
F_n = \cR^n(F) := \cS^n \circ \Phi^n \circ F^{R_n} \circ (\cS^n \circ \Phi^n)^{-1}.
\end{equation}
Note that $\cR^n(F) \in \frHL^r$. Lastly, we say that $F$ is $N$-times $(L, \epsilon, \lambda)$-regularly H\'enon-like renormalizable for some $L \geq 1$ and $\epsilon, \lambda \in (0,1)$ if $(F^{R_n}, \Psi^n)$ is $(L, \epsilon, \lambda)$-regular for all $1 \leq n \leq N$.

Suppose that the combinatorics of renormalization for $F$ are of $\bfb$-bounded type for some $\bfb \geq 2$. For many of our results, the specific values of the constants of regularity are not important, as long as $\epsilon$ is sufficiently small to compensate for the size of $\bfb$. That is, we have $\bfb\epsilon^d <1$ for some uniform constant $d \in (0,1)$ independent of $F$. In this case, we will sometimes say that $F$ is ``$N$-times regularly H\'enon-like renormalizable,'' without specifying the constants of regularity.

\begin{defn}\label{val curve}
For $1 \leq n\leq N$, denote
$$
I^n_0 := \pi_h(B^n_0)
\matsp{and}
\cI^n_m := F^m\circ (\Psi^n)^{-1}(I^n_0 \times \{0\})
\matsp{for}
m \in \bbZ.
$$
The {\it $n$th valuable curvature} of the H\'enon-like returns given in \eqref{eq.returns} is defined as
\begin{equation}\label{eq.n val curv}
\kappa_n := \kappa_{\Psi^n}(\cI^n_{R_n})
\end{equation}
(see \eqref{eq.val curv}).
\end{defn}

\subsection{Definition of regular unicriticality}\label{subsec.reg unicrit}

Consider a $C^2$-H\'enon-like map $F : D \to D$. Suppose that $F$ is infinitely renormalizable, and the renormalization limit set $\Lambda_F$ supports a unique invariant probability measure $\mu$. Then with respect $\mu$, the Lyapunov exponents of $F$ are $0$ and $\log \lambda_\mu <0$ for some $\lambda_\mu \in (0,1)$ (see \propref{lya exp}). By Oseledets theorem, $\mu$-a.e. point $p \in \Lambda_F$ has strong-stable and center directions $E^{ss}_p, E^c_p \in \bbP^2_p$ such that
\begin{equation}\label{eq.for ly reg}
\lim_{n \to +\infty} \frac{1}{n}\log \|DF^n|_{E_p^{ss}}\| = \log\lambda_\mu
\end{equation}
and
\begin{equation}\label{eq.back ly reg}
\lim_{n \to +\infty} \frac{1}{n}\log \|DF^{-n}|_{E_p^c}\| = 0.
\end{equation}
Let $\epsilon >0$. Since $F|_{\Lambda_F}$ is uniquely ergodic, \eqref{eq.for ly reg} (\eqref{eq.back ly reg} resp.) implies that $p$ is infinitely forward (backward resp.) $(L, \epsilon, \lambda_\mu)$-regular for some $L = L(p, \epsilon) \geq 1$ (see \cite[Proposition 4.7]{CLPY1}).

If $p \in \Lambda_F$ satisfies \eqref{eq.for ly reg} and \eqref{eq.back ly reg} with
$$
E^*_p := E^{ss}_p = E^c_p,
$$
then $\{F^m(p)\}_{m\in\bbZ}$ is referred to as a {\it regular critical orbit}. Note that in this case, the local strong-stable manifold $W^{ss}_{\loc}(p)$ and the center manifold $W^c(p)$ form a tangency at $p$. If this tangency is quadratic, then $\{F^m(p)\}_{m\in\bbZ}$ is referred to as a {\it regular quadratic critical orbit}.

For $t > 0$ and $p \in \bbR^2$, denote
$$
\bbD_p(t) := \{q \in \bbR^2 \; | \; \dist(q, p) < t\}.
$$

\begin{defn}\label{def unicrit}
For $0<\epsilon < \delta < 1$, we say that $F$ is {\it $(\delta, \epsilon)$-regularly unicritical on $\Lambda_F$} if the following conditions hold.
\begin{enumerate}[i)]
\item There is a regular quadratic critical orbit point $v_0 \in \Lambda_F$ (referred to as the {\it critical value} of $F$).
\item For all $t >0$, there exists $L(t) \geq 1$ such that for any $N \in \bbN$, if
\begin{equation}\label{eq.away from crit}
p \in \Lambda_F \setminus \bigcup_{n = 0}^{N-1} \bbD_{v_{-n}}(t\lambda_\mu^{\epsilon n}),
\end{equation}
then $p$ is $N$-times forward $(L(t), \delta, \lambda_\mu)$-regular.
\end{enumerate}
When $\delta$ and $\epsilon$ are implicit, we simply say that $F$ is {\it regularly unicritical on $\Lambda_F$}.
\end{defn}

\section{Statements of the Main Theorems}\label{sec.main thm}

Let $r \geq 2$ be an integer, and consider a $C^{r+1}$-H\'enon-like map $F \in \frHL^{r+1}$ of the form \eqref{eq.henonlike}. A quick computation shows that
$$
\Jac F(x,y) = -\partial_y f(x,y).
$$
If $\|\Jac F\| = 0$, then $F$ does not depend on the second coordinate $y$. This means that $F$ has the same dynamics as the unimodal map $\Piod(F)(\cdot) := f(\cdot, 0) \in \frU^{r+1}$. Hence, one may view the size of $\|\Jac F\|$ as a measure of how far $F$ is from being a 1D system.

In this paper, we focus on the case when $F$ is {\it dissipative}: $\|\Jac F\| \leq \lambda < 1$ for some $\lambda \in (0,1)$. Our goal is to understand what happens when such a map is renormalized many times. The following heuristics imply that the renormalizations of $F$ rapidly become more and more one-dimensional.

Suppose that $F$ is $N \in \bbN\cup\{\infty\}$ times renormalizable. For $1\leq n\leq N$, the $n$th renormalization $\cR^n(F)$ of $F$ is given by \eqref{eq.renorm}. Once we establish that that the nonlinear component $\Phi^n$ of the $n$th rescaling map and its inverse remain uniformly bounded, we obtain
$$
\|\Jac \cR^n(F)\| = \|\Jac \Phi^n\|\cdot\|\Jac F^{R_n}\|\cdot\|\Jac (\Phi^n)^{-1}\| \asymp \|\Jac F^{R_n}\| \leq \lambda^{R_n}.
$$
Since the return times $R_n$ grow exponentially with $n$, we see that the renormalizations $\cR^n(F)$ of $F$ converge to the space of 1D systems super-exponentially fast. Thus, to complete our understanding of the behavior of the 2D renormalization sequence $\{\cR^n(F)\}_{n=1}^N \subset \frHL^r$, it suffices to study the following sequence of 1D profiles:
\begin{equation}\label{eq.1d ren seq}
\{f_n := \Piod\circ \cR^n(F)\}_{n=1}^N \subset \frU^r.
\end{equation}

\subsection{Renormalization convergence}

The following theorem summarizes the asymptotic behavior of an infinite regular H\'enon-like renormalization sequence with combinatorics of bounded type.

\begin{thma}
Let $r \geq 2$ be an integer, and consider a $C^{r+4}$-H\'enon-like map $F : D \to D$. Suppose  that $F$ has infinite nested regular H\'enon-like returns given by \eqref{eq.returns} with combinatorics of bounded type. Let $\epsilon, \lambda \in (0,1)$ be the marginal exponent and the contraction base of regularity respectively. Then
\begin{equation}\label{eq.crit value}
\bigcap_{n=1}^\infty F^{R_n}(\cB^n) = \{v_0\}.
\end{equation}
and $\cO_{\crit} := \{v_m := F^m(v_0)\}$ is a regular critical orbit. Moreover, for all $n \in \bbN$ sufficiently large, the following properties hold.
\begin{enumerate}[i)]
\item There exists a $C^{r+3}$-chart $\Phi : \cU \to U$ centered at $v_0$ such that $\cB^n \Subset \cU$, and
$$
\|\Phi \circ(\Psi^n)^{-1} - \Id \|_{C^{r+3}} < \lambda^{(1-\bepsilon)R_n}.
$$
\item The renormalization domain $\cB^n$ can be extended vertically so that $|\pi_v \circ \Phi(\cB^n)|$ is uniformly bounded below. Moreover, there exist a uniform constants $0 < \sigma_1 < \sigma_2<1$ such that
$$
\sigma_1^n < |\pi_h\circ \Phi(\cB^n)| < \sigma_2^n
\matsp{and}
|\pi_v \circ \Phi(F^{R_n}(\cB^n))|^2 \asymp |\pi_h\circ \Phi(\cB^n)|.
$$
\item The H\'enon-like map $\cR^n(F)$ is $\lambda^{(1-\bepsilon)R_n}$-thin in $C^{r+3}$.
\item We have $\|\cR^n(F)\|_{C^r} = O(1)$.
\item If $r \geq 4$, then there exists a universal constant $\rho \in (0,1)$ such that for any unimodal map $f_* \in \frU^r$ with the same asymptotic combinatorial type as $F$, we have
$$
\|\Piod\circ\cR^n(F) - \cRod^n(f_*)\|_{C^{r-1}} <\rho^n.
$$
\end{enumerate}
\end{thma}

The point $v_0$ given in \eqref{eq.crit value} is referred to as a {\it regular critical value}. As the name suggests, it is the 2D analog of the critical value for unimodal maps.

Let us briefly comment on how part v) in Theorem A relates to the existing renormalization theory of unimodal maps in literature. For $\gamma \in \{r, \omega\}$, the 1D renormalization $\cRod$ defined in \subsecref{subsec.unimodal} can be viewed as an operator acting on the Banach space $\frU^\gamma$ of unimodal maps. In \cite{L1}, Lyubich showed that $\cRod$ restricted to $\frU^\omega$ is an analytic operator that has a hyperbolic attractor $\frA \subset \frU^\omega$ with exactly one unstable dimension. This attractor is referred to as the {\it full renormalization horseshoe}. 

Given an integer $\bfb \geq 2$, let $\frA_{\bfb}$ be the compact invariant subset of $\frA$ that consists of maps with combinatorics of $\bfb$-bounded type. In \cite{dFdMPi}, de Faria-de Melo-Pinto showed that for the action of the renormalization operator $\cRod$ on the much larger space $\frU^3 \supset \frU^\omega$, the set $\frA_{\bfb}$ is still a hyperbolic attractor with one unstable dimension.

Once we establish parts iii) and iv) in Theorem A, it follows by a general lemma (\lemref{vary compose}) that the sequence of 1D profiles of the renormalizations of $F$ is shadowed by actual orbits of 1D renormalization. Asymptotic convergence then follows from the hyperbolicity of 1D renormalization discussed above.

\subsection{Finite-time checkability}

If a quadratic polynomial $f : I \to I$ has a periodic subinterval $I^1 \subset I$, then the cycle of $I^1$ must go through the critical point of $f$ exactly once. This ensures that the first return map on $I^1$ under $f$ is again a unimodal map. In contrast, if a H\'enon map $F : D \to D$ has a periodic subdomain $\cB^1 \subset B$, it is not necessarily true that the first return map on $\cB^1$ under $F$ must also be H\'enon-like. Intuitively, this contrast is due to the following important conceptual difference between the 1D case and the 2D case: for the dynamics of 2D diffeomorphisms, there is no definite single pinpoint location at which the action of  the critical point takes place. Nevertheless, the following result states that if $F$ is regular H\'enon-like renormalizable up to a sufficiently deep depth, then all further topological renormalizations of $F$ are necessarily also regular H\'enon-like.

\begin{thmb}
Suppose we are given the following data:
\begin{enumerate}[i)]
\item a H\'enon-like map $F \in \frHL^6$;
\item constants $\bfb \geq 2$; $L \geq 1$ and $\lambda \in (0,1)$; and
\item an increasing sequence $\{R_n\}_{n=0}^\infty$ such that $R_0 = 1$ and $R_n/R_{n-1} \leq \bfb$ for all $n \in \bbN$.
\end{enumerate}
Let $\epsilon_0 \in (0,1)$ and $n_0 \in \bbN \cup \{0\}$ be constants satisfying
\begin{equation}\label{eq.small margin}
\bfb \epsilon_0^d < 1,
\end{equation}
and
\begin{equation}\label{eq.large depth 0}
C\lambda^{\epsilon_0 R_{n_0}} < 1
\end{equation}
for some uniform constants $d \in (0,1)$ (independent of $F$) and $C \geq 1$ (depending only on $L$, $\lambda$, $\lambda^{1-\epsilon_0}\|DF^{-1}\|$, $\|F\|_{C^3}$ and $\bfb$). Suppose $F$ has $n_1$ nested $(L, \epsilon_0, \lambda)$-regular H\'enon-like returns given by \eqref{eq.returns} for some $n_1 \geq n_0$. Suppose that
\begin{equation}\label{eq.large depth 1}
K\lambda^{\epsilon_0 R_{n_1}} < 1,
\end{equation}
for some uniform constant $K = K(C, \|F^{R_{n_0}}|_{\cB^{n_0}}\|_{C^6}, \kappa_{n_0}) \geq 1$ (where $\kappa_{n_0}$ is the $n_0$th valuable curvature given in \eqref{eq.n val curv}). Then 
there exists uniform constants $\bfL \geq L$ and $\delta \in (\epsilon_0, 1)$ such that the following holds. Suppose that $F$ is $N$-times topologically renormalizable for some $n_1 \leq N \leq \infty$ with return times $\{R_n\}_{n=1}^N$. Then all $N$ renormalizations of $F$ are $(\bfL, \delta, \lambda)$-regular H\'enon-like (except possibly the last two if $N < \infty$).
\end{thmb}

Quantitative Pesin theory combined with some standard arguments in one-dimensional dynamics is used to show that any topological return map within a sufficiently deep renormalization depth must be H\'enon-like (see \secref{sec.combin}). A priori bounds is then needed to guarantee that regularity is preserved when we pass into deeper renormalization depths (see \secref{sec.preserve}).

The significance of the previous theorem is that it turns infinite regular H\'enon-like renormalizability of $F$ into a finite condition, provided that $F$ is known to be infinitely topologically renormalizable. The next result gives a criterion for guaranteeing the latter property.

\begin{thmc}
Suppose we are given the following data:
\begin{enumerate}[i)]
\item a one-parameter family $\{F_a\}_{a\in \frI} \subset \frHL^6$ depending $C^1$-smoothly on $a$;
\item constants $\bfb \geq 2$; $L \geq 1$ and $\lambda \in (0,1)$; and
\item an increasing sequence $\{R_n\}_{n=0}^\infty$ such that $R_0 = 1$ and $R_n/R_{n-1} \leq \bfb$ for all $n \in \bbN$.
\end{enumerate}
Let $\epsilon_0 \in (0,1)$ and $n_1 \in \bbN$ be the constants given in \eqref{eq.small margin} and \eqref{eq.large depth 1}. Suppose that for all $a \in \frI$, the map $F_a$ is $n_1$-times $(L, \epsilon_0, \lambda)$-regularly H\'enon-like renormalizable, and that $\{\Piod \circ \cR^{n_1}(F_a)\}_{a \in \frI}$ forms a full one-parameter family of 1D unimodal maps. Then for any $\bfb$-bounded renormalization type, there exists a parameter $a_* \in \frI$ such that $\cR^{n_1}(F_{a_*})$ realizes this type.
\end{thmc}

The {\it renormalization type} of an infinitely regular H\'enon-like renormalizable map referred to in Theorem C is defined in \secref{sec.combin} (see \eqref{eq.proj value seq}). It can be identified with the combinatorial type of some infinitely renormalizable unimodal map. The proof of Theorem C is the content of \secref{sec.realization}.

\begin{rem}\label{rem.finite}
While it is not done in this paper, it is possible to obtain explicit estimates of the constants $d$, $C$ and $K$ in \eqref{eq.small margin}, \eqref{eq.large depth 0} and \eqref{eq.large depth 1}. This means that for a given specific family of 2D maps (say, the H\'enon family $\frH$), Theorems B and C turn infinite regular H\'enon-like renormalizability in this family into an explicit finite-time checkable condition. This is illustrated in Examples \ref{ex.henon 0} and \ref{ex.henon 1}.
\end{rem}

\begin{ex}\label{ex.henon 0}
Consider a H\'enon map $F_{a,b} \in \frH$ (see \eqref{eq.henon fam}) restricted to a suitable bounded subset $U \subset \bbR^2$. Then $F_{a,b}$ has uniformly bounded $C^6$-norm. Moreover, there exists a uniform constant $c = c(U) \geq 1$ such that $\|DF_{a,b}|_{U}^{-1}\| < c/|b|$. Lastly, the $0$th valuable curvature of $F_{a,b}$ is exactly equal to $2$.

For $\lambda \in (0,1)$, consider the one parameter family of H\'enon maps $\frH_\lambda := \{F_{a,\lambda}\}_{a\in\bbR}$. Given $\bfb \geq 2$, a number $\epsilon_0 \in (0,1)$ can be chosen so that \eqref{eq.small margin} holds. Set $L = 1$. By the above observations, we see that for $\frH_\lambda$, the value of the uniform constant $K$ given in \eqref{eq.large depth 1} depends only on $\lambda$. Then let $\lambda_0 \in (0,1)$ be the largest number such that \eqref{eq.large depth 1} is satisfied when we set $\lambda = \lambda_0$ and $n_1 = 0$ (i.e. $C\lambda_0^{\epsilon_0} < 1$).

Fix $\lambda \in (0, \lambda_0)$. Since $\{\Piod(F_{a,\lambda})\}_{a \in \bbR} \equiv \frQ$ is a full family, it follows from the realization theorem that for any $\bfb$-bounded renormalization type, there exists a parameter value $a_*(\lambda) \in \bbR$ such that $F_{a_*(\lambda), \lambda}$ is infinitely regularly H\'enon-like renormalizable with this combinatorics.

Note that the above argument is non-perturbative, and does not rely on the robustness of the 1D renormalization convergence. In particular, any numerical estimates on the quantities $D$ and $K$ would immediately yield a definite lower bound on the value of the Jacobian $\lambda_0$.
\end{ex}

\begin{ex}\label{ex.henon 1}
Allowing for non-zero values of $n_1$ in \exref{ex.henon 0} enables us to potentially find infinitely regularly H\'enon-like renormalizable maps in the H\'enon family with Jacobians arbitrarily close to $1$ as follows.

Fix $\lambda \in (0,1)$, and consider the one-parameter family $\frH_\lambda$ of H\'enon maps with Jacobian $\lambda$. Suppose we can find an interval $\frI \subset \bbR$ of parameters such that for each map $F_a := F_{a, \lambda}$ with $a\in \frI$, there exists a sequence of $N \in \bbN$ nested H\'enon-like returns $\{(F_a^{R_n}, \Psi^n_a)\}_{n=1}^N$ with combinatorics of $\bfb$-bounded type. Additionally, suppose we can verify the following conditions.
\begin{enumerate}[i)]
\item The family $\{\Piod \circ \cR^N(F_a)\}_{a \in \frI}$ depends smoothly on the parameter $a$, and is full.
\item There exists $L \geq 1$ such that for all $a \in \frI$, the returns $\{(F_a^{R_n}, \Psi^n_a)\}_{n=1}^N$ are $(L, \epsilon_0, \lambda)$-regular.
\item We have $N \geq n_1$, where $n_1 \in \bbN$ is the smallest number such that \eqref{eq.large depth 1} holds.
\end{enumerate}
Then as before, we are guaranteed the existence of an infinitely regularly H\'enon-like renormalizable map $F_{a_*(\lambda), \lambda}$ with $a_*(\lambda) \in \frI$ that realizes any given $\bfb$-bounded renormalization type.

We expect that for reasonably small values of $\lambda$ (which would result in small values of $n_1$), it should be feasible to check these conditions numerically using a computer. 
\end{ex}

\subsection{Regular unicriticality}

The renormalization convergence in Theorem A gives a very detailed account of the geometry  near the critical value $v_0$ given in \eqref{crit value}. We now turn out attention to the global geometry of the dynamics. A priori bounds implies that that there is definite scaling when we pass from one renormalization depth to the next. This fact has the following important geometric consequence for infinitely renormalizable H\'enon-like maps.

\begin{thmd}
Consider a $C^6$-H\'enon-like map $F : D \to D$. Suppose that $F$ is infinitely regularly H\'enon-like renormalizable with combinatorics of bounded type. Then the Hausdorff dimension of the renormalization limit set $\Lambda_F$ given in \eqref{eq.limit set} is less than $1$. Consequently, $\Lambda_F$ is totally disconnected, minimal, and supports a unique invariant probability measure $\mu$.
\end{thmd}

The proof of Theorem D is the content of \secref{sec.shrink}.

Let $X \Subset D$ be a compact totally invariant set for a H\'enon-like map $F$. We say that $F$ is {\it uniformly partially hyperbolic on $X$} if every point $p \in X$ is infinitely forward and backward regular along some tangent direction $E^{ss}_p$ at $p$, and the constants of regularity are uniform in $p$.  The geometry of a 2D dynamical system is very well understood on uniformly partially hyperbolic sets. In particular, it is known that the leaves in the strong-stable and center laminations vary continuously, have uniformly bounded curvature, and are uniformly transverse to each other.

Suppose that $F$ has infinite nested regular H\'enon-like returns given by \eqref{eq.returns} with combinatorics of bounded type. Consider the regular critical value $v_0$ given in \eqref{eq.crit value}. Note that $v_0$ is both infinitely forward and backward regular. Thus, $v_0$ has well-defined strong-stable manifold $W^{ss}(v_0)$ and center manifold $W^c(v_0)$. The H\'enon-likeness of the return maps under $F$ forces $W^{ss}(v_0)$ and $W^c(v_0)$ to form a quadratic tangency at $v_0$. See \figref{fig.critval}. Thus, the orbit $\cO_{\crit}$ of $v_0$ is a regular quadratic critical orbit of $F$ (as defined in \subsecref{subsec.reg unicrit}).

\begin{figure}[h]
\centering
\includegraphics[scale=0.3]{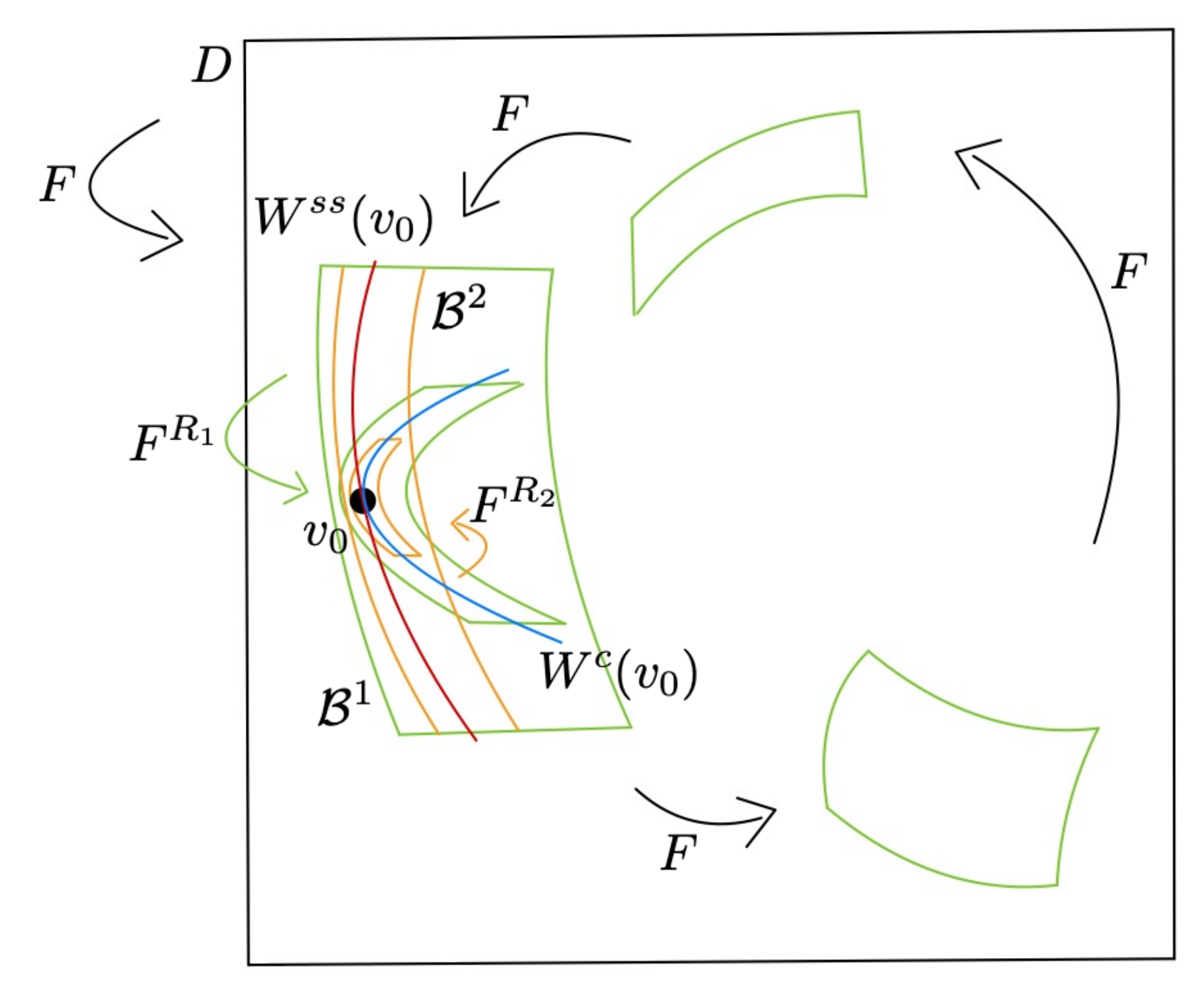}
\caption{The critical value $v_0$ of an infinitely regularly H\'enon-like renormalizable map $F$.}
\label{fig.critval}
\end{figure}

The existence of $\cO_{\crit}$ immediately implies that $F$ is not uniformly partially hyperbolic on $\Lambda_F$. However, our last main theorem states that this is the only obstruction, and that uniform regularity still holds outside a slow-exponentially shrinking neighborhood of $\cO_{\crit}$.

\begin{thme}
Consider a $C^6$-H\'enon-like map $F : D \to D$. Suppose that $F$ has infinite nested regular H\'enon-like returns given by \eqref{eq.returns} with combinatorics of bounded type. Then for any $\epsilon > 0$, there exists $L_\epsilon \geq 1$ such that for all $n \in \bbN$, the H\'enon-like return $(F^{R_n}, \Psi^n)$ is $(L_\epsilon, \epsilon, \lambda_\mu)$-regular. Moreover, $F$ is regularly unicritical on the renormalization limit set $\Lambda_F$ with the critical value $v_0$ given by \eqref{eq.crit value}.
\end{thme}

The study of 2D dynamics on a uniformly partially hyperbolic set $X$ is greatly facilitated by the fact that $X$ has a {\it local product structure}. This means that $X$ can be covered by finitely many charts, called {\it regular Pesin boxes}, which endows the set with locally defined vertical (strong-stable) and horizontal (center) directions that are invariant under the dynamics.

In our setting, any covering of $\Lambda_F$ by regular Pesin boxes must leave out points that are too close to the critical orbit $\cO_{\crit}$. In \cite{CLPY1}, we introduce new covering domains called {\it critical tunnels} and {\it valuable crescents} that uniformize the dynamics of $F$ near $\cO_{\crit}$ (which is fundamentally non-linear in nature). See \figref{fig.tunnel}. These new domains, together with regular Pesin boxes, completely cover $\Lambda_F$, resulting in a new type of dynamical structure that we call a {\it regular unicritical structure}.

\begin{figure}[h]
\centering
\includegraphics[scale=0.27]{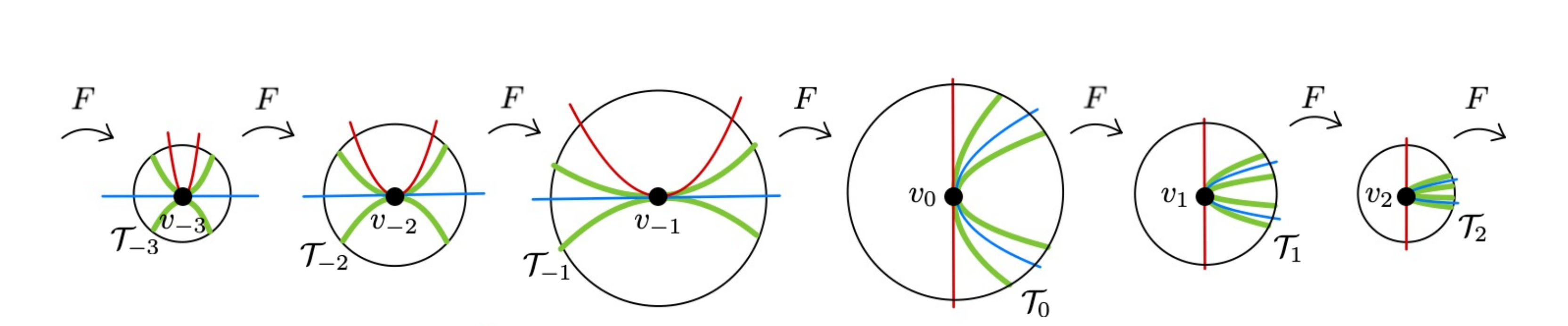}
\caption{Regular quadratic critical orbit $\cO_{\crit} = \{v_m\}_{m\in\bbZ}$ contained in critical tunnels $\{\cT_{-n}\}_{n =1}^\infty$ and valuable crescents $\{\cT_n\}_{n=0}^\infty$. For $m \in \bbZ$, the strong-stable and center manifolds of $v_m$ are indicated as red and blue curves respectively. The tunnel/crescent $\cT_m$ is the pinched region bounded between two green curves that contains $W^c(v_m)$. The diameter of $\cT_m$ shrinks slow-exponentially as $|m| \to \infty$.}
\label{fig.tunnel}
\end{figure}

Regular unicritical structures for regularly unicritical systems can fulfill a similar function as local product structures for uniformly partially hyperbolic systems. In \cite{CLPY1}, this new structure is used to characterize the local geometry of every strong-stable and center manifold in terms of its proximity to the critical orbit in an explicit way. Additionally, the following converse of the unicriticality theorem is proved.

\begin{thm}[\cite{CLPY1}]
Let $F : \cD \to F(\cD) \Subset \cD$ be a dissipative $C^3$-diffeomorphism defined on a Jordan domain $\cD \subset \bbR^2$. Suppose that $F$ is infinitely topologically renormalizable, and assume that $F$ is regularly unicritical on the renormalization limit set. Then the renormalizations of $F$ are eventually regular H\'enon-like.
\end{thm}


 \section{Convergence of the Straightening Charts}\label{sec.conv chart}
 
Let $r \geq 2$ be an integer, and consider a $C^{r+1}$-H\'enon-like map $F : D \to D$. For some $N \in \bbN \cup \{\infty\}$; $L \geq 1$ and $\epsilon, \lambda \in (0,1)$, suppose that $F$ has $N$ nested $(L, \epsilon, \lambda)$-regular H\'enon-like returns given by \eqref{eq.returns}. Furthermore, assume that $N$ is sufficiently large, so that for some smallest number $0 \leq n_0 \leq N$, we have
\begin{equation}\label{eq.proper depth 0}
\overline{K_0} \lambda^{\epsilon R_{n_0}} < \eta,
\end{equation}
where $\eta \in (0,1)$ is independent of $F$, and
\begin{equation}\label{eq.const 0}
{K_0} = {K_0}(L, \lambda, \epsilon, \lambda^{1-\epsilon}\|DF^{-1}\|, \|DF\|_{C^r}, r)\geq 1
\end{equation}
is a uniform constant.

For $n_0 \leq n \leq N$ and $m \in \bbZ$, denote $\cB^n_m := F^m(\cB^n)$. Observe that
$
\cB^{n+1}_{R_{n+1}} \Subset \cB^n_{R_n}.
$
Let
$$
v_0 \in \cZ_0 := \bigcap_{n=1}^N \cB^n_{R_n},
$$
be a point to be specified later (as the {\it critical value of $F$}). Without loss of generality, assume that $\Psi^n$ is centered at $v_0$.

In this section, we outline the description of the asymptotic behavior of the centered straightening charts $\{\Psi^n\}_{n=1}^N$ for the renormalizations of $F$ obtained in \cite{CLPY3}.

Define
$$
I^n_0 := \pi_h(B^n_0)
\matsp{and}
\cI^n_0 := (\Psi^n)^{-1}(I^n_0 \times \{0\}).
$$
Then it follows that
$
I^n_0 \Subset I^1_0
$ and $
\Psi^n|_{\cI^n_0} =\Psi^1|_{\cI^n_0}.
$
Denote $\cI^n_m := F^m(\cI^n_0)$ for $m \in \bbZ$. 

For $p_0 \in \cB^n_0$, write $z_0 := \Psi^n(p_0)$, and let
$$
E^h_{p_0} := D(\Psi^n)^{-1}(E^{gh}_{z_0})
\matsp{and}
E^{v,n}_{p_0} := D(\Psi^n)^{-1}(E^{gv}_{z_0}).
$$
Additionally, let
$$
E^{h,n}_{p_{R_n-1}} := DF^{R_n-1}(E^h_{p_0})
\matsp{and}
E^v_{p_{R_n-1}} := DF^{-1}(E^h_{p_{R_n}}) = DF^{R_n-1}(E^{v,n}_{p_0}).
$$
By increasing $L$ by a uniform amount if necessary (see \propref{grow irreg}), we may assume that every $q \in \cB^n_{R_n-1}$ is $(R_n-1)$-times backward $(L, \epsilon, \lambda)$-regular along $E^v_q$.

\begin{prop}[Vertical extension of charts]\cite[Proposition 3.2]{CLPY3}\label{vert prop nest}
For $n_0 \leq n \leq N$, the $n$th centered straightening chart can be extended to $\Psi^n : \hcB^n_0 \to \hB^n_0$ such that the following properties hold.
\begin{enumerate}[i)]
\item The quadrilateral $\hcB^n_0$ is vertically proper and $\eta$-vertical in $\cB^{n_0}_0$.
\item We have $\|(\Psi^n)^{\pm 1}\|_{C^r} < {K_0}$, and
\begin{equation}\label{eq.psi conv}
\|\Psi^n\circ (\Psi^{n+1})^{-1}-\Id\|_{C^r} <\lambda^{(1-\bepsilon)R_n}.
\end{equation}
\item Every point $q_0 \in \hcB^n_0$ is $R_n$-times forward $({K_0}, \epsilon, \lambda)$-regular along $E^{v,n}_{q_0}$.
\end{enumerate}
\end{prop}

\begin{rem}
In \secref{sec.combin}, we will show that $\hcB^n_0$ is $R_n$-periodic (and hence, we may assume that $\cB^n_0 =\hcB^n_0$). See \eqref{eq.min domain}.
\end{rem}

\begin{prop}[Locating the critical value]\cite[Proposition 3.4]{CLPY3}\label{crit value}
If $N = \infty$, then the following statements hold.
\begin{enumerate}[i)]
\item For any point $p_0 \in \cZ_0$, there exists a unique strong stable direction $E^{ss}_{p_0} \in \bbP^2_{p_0}$ such that
$$
\|E^{v, n}_{p_0}-E^{ss}_{p_0}\| < \lambda^{(1-\bepsilon)R_n}
\matsp{for}
n \geq n_0.
$$
Moreover, $p_0$ is infinitely forward $(L, \epsilon, \lambda)$-regular along $E^{ss}_{p_0}$.
\item Any point $p_{-1} \in \cZ_{-1} := F^{-1}(\cZ_0)$ is infinitely backward $(L, \epsilon, \lambda)$-regular along $E^v_{p_{-1}}$. Moreover, there exists a unique center direction $E^c_{p_{-1}} \in \bbP^2_{p_{-1}}$ such that
$$
\|E^{h, n}_{p_{-1}}-E^c_{p_{-1}}\| < \lambda^{(1-\bepsilon)R_n}
\matsp{for}
n \geq n_0.
$$
\item There exists a unique point $v_0 \in \cZ_0$ such that
$$
E^{ss}_{v_0} = DF(E^c_{v_{-1}}).
$$
Moreover, the strong stable manifold $W^{ss}(v_0)$ and the center manifold $F(W^c(v_{-1}))$ have a quadratic tangency at $v_0$.
\end{enumerate}
\end{prop}

We define the {\it critical value $v_0 \in \cZ_0$} as follows. If $N=\infty$, let $v_0$ be the point given in \propref{crit value} iii). Otherwise, let $v_0$ be the unique point in $\cI^N_{R_N}$ such that
$$
DF^{R_N}(E^h_{v_{-R_N}}) = E^{v,N}_{v_0}
$$
(recall that such a point exists since $\cI^N_{R_N}$ is a vertical quadratic curve in $\cB^N_0$). Define the {\it critical point} as $v_{-1} := F^{-1}(v_0)$.

\begin{thm}[Valuable charts]\cite[Theorem 3.5]{CLPY3}\label{crit chart}
Let ${K_0} \geq 1$ be the constant given in \eqref{eq.const 0}. There exist charts
$$
\Phi_0 : (\cB_0, v_0) \to (B_0, 0)
\matsp{and}
\Phi_{-1} : (\cB_{-1}, v_{-1}) \to (B_{-1}, 0)
$$
such that
\begin{itemize}
\item $\Phi_0$ is centered at $v_0$ and is genuine horizontal;
\item $\cB_0 \supset \cB^{n_0}_0$, $\cB_{-1} \supset \cB^{n_0}_{R_{n_0}-1}$ and $F(\cB_{-1}) \Subset \cB_0$;
\item $\|\Phi_i^{\pm 1}\|_{C^r} < {K_0}$ for $i \in \{0, -1\}$; and
\item we have
\begin{equation}\label{eq.henon trans}
\Phi_0 \circ F \circ \Phi_{-1}^{-1}(x,y) = (f_0(x) - \lambda y, x)
\matsp{for}
(x,y) \in B_{-1},
\end{equation}
where $f_0 : (\pi_h(B_{-1}), 0) \to (\pi_h(B_0), 0)$ is a $C^r$-map that has a unique critical point at $0$ such that
\begin{equation}\label{eq.second}
\|f_0''\|_{C^{r-2}} < {K_0}
\matsp{and}
\kappa_F:=\inf_{x\in\pi_h(B_{-1})} f_0''(x) > 0.
\end{equation} 
\end{itemize}
Moreover, the following properties hold for $n_0 \leq n \leq N$.
\begin{enumerate}[i)]
\item We have
$$
\|\Phi_0 \circ (\Psi^n)^{-1} - \Id\|_{C^r} <\lambda^{(1-\bepsilon)R_n}.
$$
\item Let
$$
H_n := \Phi_{-1}\circ F^{R_n-1}\circ (\Psi^n)^{-1}.
$$
Then $H_n(x,y) = (h_n(x), e_n(x,y))$, where $h_n : I^n_0 \to h_n(I^n_0)$ is a $C^r$-diffeomorphism and $e_n$ is a $C^r$-map such that
\begin{equation}\label{eq.first entry}
\lambda^{\bepsilon R_n} < |h_n'(x)| < \lambda^{-\bepsilon R_n}
\matsp{for}
x\in I^n_0
\matsp{and}
\|e_n\|_{C^r} <\lambda^{(1-\bepsilon)R_n}.
\end{equation}
\end{enumerate}
\end{thm}

\begin{figure}[h]
\centering
\includegraphics[scale=0.15]{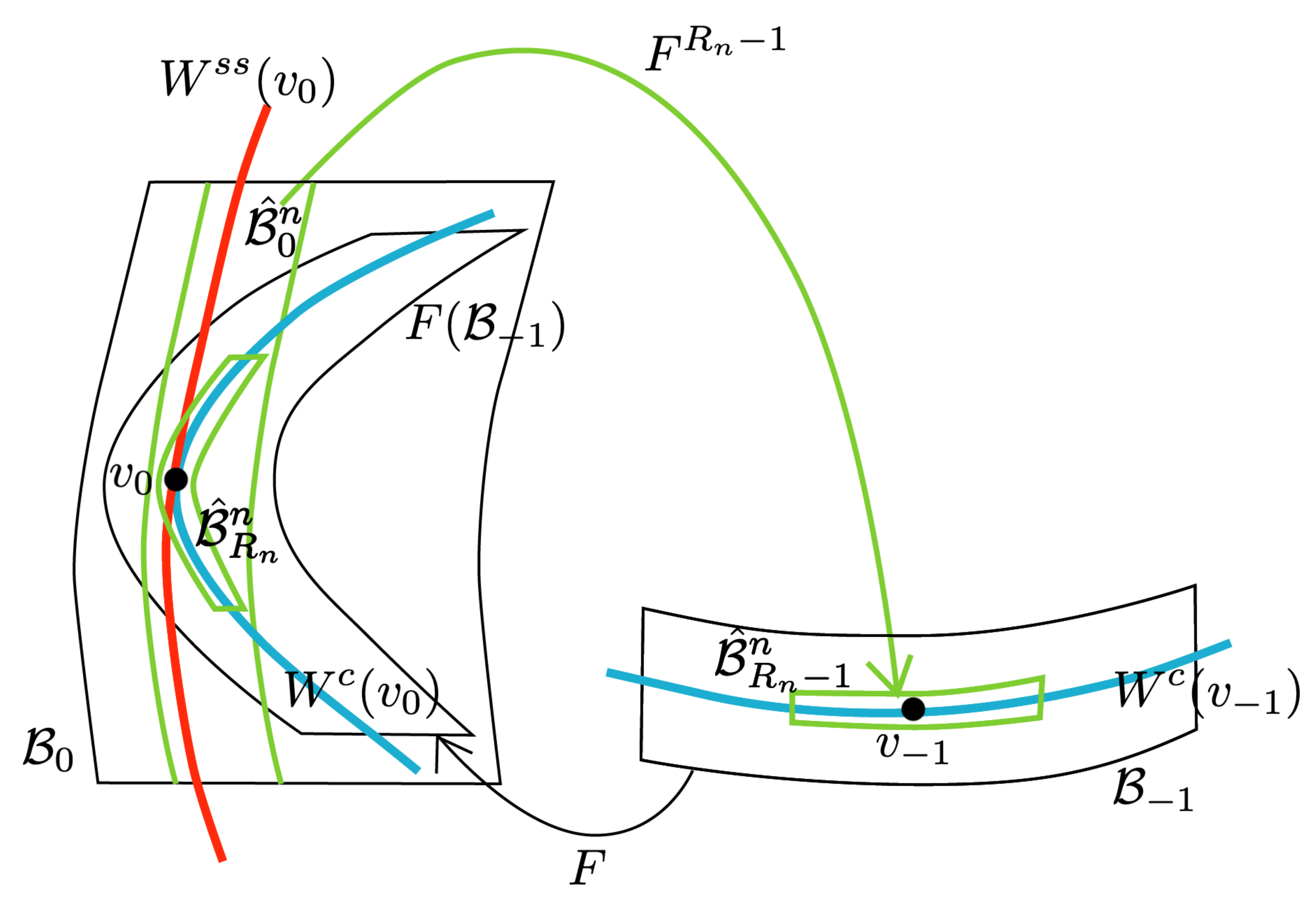}
\caption{Geometry near the critical value $v_0$ and the critical point $v_{-1}$ (if $N =\infty$). For $n \geq n_0$, we have $v_0 \in \hcB^n_0 \subset \cB_0$ and $v_{-1} \in \hcB^n_{R_n-1} \subset \cB_{-1}$. There exist charts $\Phi_0 : \cB_0 \to B_0$ and $\Phi_{-1} : \cB_{-1} \to B_{-1}$ such that $\Phi_0\circ F \circ \Phi_{-1}$ is H\'enon-like (see \eqref{eq.henon trans}). The charts $\Psi^n$ converges to $\Phi_0$}
\label{fig.chartconv}
\end{figure}

Denote
\begin{equation}\label{eq.Ihv}
I^{h/v}_i := \pi_{h/v}(B_i)
\matsp{and}
\cI^h_i := \Phi_i^{-1}(I^h_i \times \{0\})
\matsp{for}
i \in \{0, -1\}.
\end{equation}
Observe that
$$
I^h_0 \Supset I^{n_0}_0 \Supset I^{n_0+1}_0 \Supset \ldots
\matsp{and}
I^h_{-1} \Supset h_{n_0}(I^{n_0}_0) \Supset h_{n_0+1}(I^{n_0+1}_0) \Supset \ldots.
$$
Moreover, if $X \subset \cB^n_0$, then \eqref{eq.first entry} implies
\begin{equation}\label{eq.first entry squeeze}
\Phi_{-1} \circ F^{R_n-1}(X) \subset h_n(I^n_0) \times [-\lambda^{(1-\bepsilon)R_n}, \lambda^{(1-\bepsilon)R_n}].
\end{equation}

We record the following consequences of \thmref{crit chart}.

\begin{lem}\label{est value}
Let $f_0 : I^h_{-1} \to I^h_0$ be the map with a unique critical point at $0$ given in \thmref{crit chart}. Then
$$
\frac{\kappa_F}{2} x^2 < f_0(x) < \frac{{K_0}}{2}x^2
\matsp{and}
\kappa_F |x| < |f_0'(x)| < {K_0} |x|.
$$
\end{lem}

\begin{lem}\label{root at value}
Let
$$
I^{h,\pm}_{-1} := \{x \in I^h_{-1} \; | \; \pm x > 0\}
\matsp{and}
g_\pm := \left(f_0|_{I^{h,\pm}_{-1}}\right)^{-1}.
$$
Denote $\theta := {K_0}/\kappa_F$. Then for $1 \leq i \leq r$, we have
$$
|g_\pm^{(i)}(t)| < \frac{\bar \theta}{|t|^{i-1/2}}
\matsp{for}
t > 0.
$$
\end{lem}

Define $P_{-1} : (\cB_{-1}, v_{-1}) \to (I^h_{-1}, 0)$ and $P^n_0 : (\hcB^n_0, v_0) \to (I^n_0, 0)$ for $n_0 \leq n \leq N$ by
$$
P_{-1} := \pi_h \circ \Phi_{-1}
\matsp{and}
P^n_0 := \pi_h \circ \Psi^n.
$$
Denote
$$
I^n_{R_n-1} := P_{-1}(\hcB^n_{R_n-1}) = P_{-1}(\cI^n_{R_n-1})= h_n(I^n_0).
$$
Define the {\it $n$th (valuable) projection map} $\cP^n_0 : \hcB^n_0 \to \cI^n_0$ by
$$
\cP^n_0 := (\Psi^n)^{-1} \circ \Pi_h \circ \Psi^n.
$$
Observe that $\cP^n_0|_{\cI^n_0} = \Id$.

We record the following immediate consequence of \thmref{reg chart} and Propositions \ref{for gt} and \ref{back dt}.

\begin{lem}\label{flat}
For $n_0 \leq n\leq N$, denote $\lambda_n := \lambda^{(1-\bepsilon)R_n}$. Then for $0 < t < \lambda^{-\bepsilon R_n}$, the following statements hold.
\begin{enumerate}[i)]
\item Let $\tiE_{p_0} \in \bbP^2_{p_0}$ be a $t$-horizontal direction at $p_0 \in \hcB^n_0$. Then $\tiE_{p_{R_n-1}}$ is $(1+t)\lambda_n$-horizontal in $\cB_{-1}$.
\item Let $E_{p_{R_n-1}} \in \bbP^2_{p_{R_n-1}}$ be a $t$-vertical direction at $p_{R_n-1} \in \hcB^n_{R_n-1}$. Then $E_{p_0}$ is $t\lambda_n$-vertical in $\hcB^n_0$.
\item Let $\Gamma^h_0$ be a curve that is $t$-horizontal in $C^r$ in $\hcB^n_0$. Then $F^{R_n-1}(\Gamma^h_0)$ is $(1+t)^r\lambda_n$-horizontal in $C^r$ in $\hcB_{-1}$.
\item Let $\Gamma^v_{R_n-1}$ be a curve that is $t$-vertical in $C^r$ in $\hcB^n_{R_n-1}$. Then $F^{-R_n+1}(\Gamma^v_{R_n-1})$ is $t\lambda_n$-vertical in $C^r$ in $\hcB^n_0$.
\end{enumerate}
\end{lem}

By \lemref{flat} iii), $\cI^n_{R_n-1}$ is $\eta_n$-horizontal in $\cB_{-1}$. Thus, there exists a $C^r$-map $g_n : I^n_{R_n-1} \to \bbR$ with $\|g_n\|_{C^r} < \lambda_n$ such that
$$
\Phi_{-1}(\cI^n_{R_n-1}) = \{(x, g_n(x)) \; | \; x \in I^n_{R_n-1}\}.
$$
Define $G_n : I^n_{R_n-1} \to \Phi_{-1}(\cI^n_{R_n-1})$ by $G_n(x) := (x, g_n(x)).$ Define the {\it $n$th critical projection map $\cP^n_{-1} : P_{-1}^{-1}(I^n_{R_n-1}) \to \cI^n_{R_n-1}$} by
\begin{equation}\label{eq.-1 proj}
\cP^n_{-1} := \Phi_{-1}^{-1}\circ G_n \circ P_{-1}.
\end{equation}

\begin{lem}\cite[Lemma 4.2]{CLPY3}\label{first entry hor graph}
For $n_0 \leq n \leq N$, let $\Gamma_0$ be a horizontal curve in $\hcB^n_0$. Then 
$$
F^{R_n-1}|_{\Gamma_0} = (\cP^n_{-1}|_{\Gamma_{R_n-1}})^{-1}\circ F^{R_n-1} \circ \cP^n_0|_{\Gamma_0}.
$$
\end{lem}

\begin{figure}[h]
\centering
\includegraphics[scale=0.2]{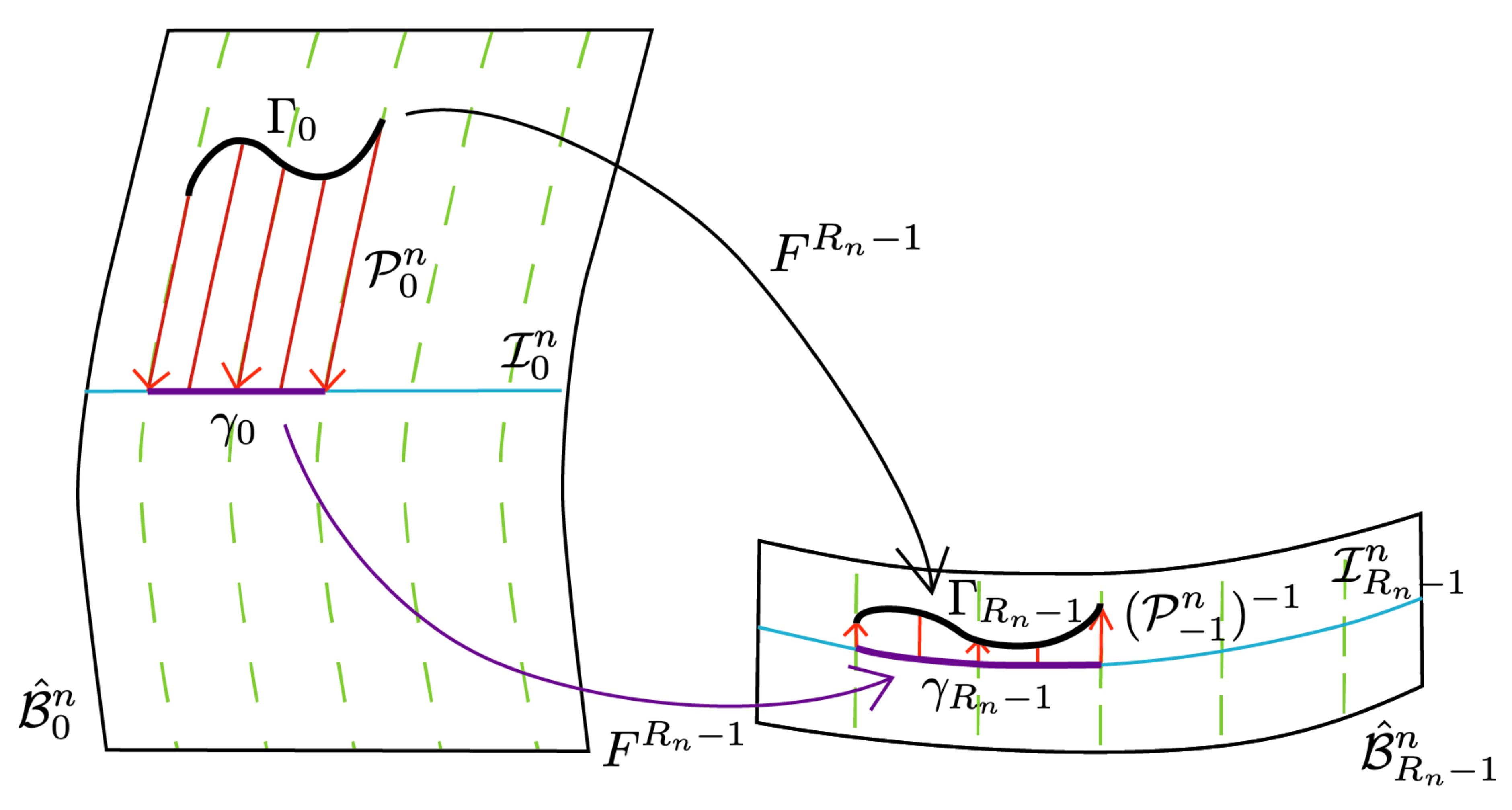}
\caption{Projections $\cP^n_0 : \hcB^n_0 \to \cI^n_0$ and $\cP^n_{-1} : \hcB^n_{R_n-1} \to \cI^n_{R_n-1}$ near the critical value $v_0$ and critical point $v_{-1}$ respectively. On any horizontal curve $\Gamma_0 \subset \hcB^n_0$, the iterate $F^{R_n-1}$ commutes with these projections.}
\label{fig.proj}
\end{figure}


\section{Avoiding the Critical Value}\label{sec.crit avoid}

For some $N \in \bbN \cup \{\infty\}$, let $F$ be the $N$-times regularly H\'enon-like renormalizable map considered in \secref{sec.conv chart}. Recall the constants $n_0$, $\eta$ and ${K_0}$ given in \eqref{eq.proper depth 0} and \eqref{eq.const 0}, and the constant $\kappa_F$ given in \eqref{eq.second}. Additionally, assume that $n_0 \leq N$ is the smallest number such that
\begin{equation}\label{eq.proper depth 1}
\overline{K_1} \lambda^{\epsilon R_{n_0}} < \eta,
\end{equation}
where
\begin{equation}\label{eq.const 1}
K_1 = K_1({K_0}, \kappa_F) \geq 1
\end{equation}
is a uniform constant. In this section, we show that if a (finite or infinite) orbit of a point avoids getting ``too close'' to the critical value $v_0$, then it has uniform regularity.

For $n_0 \leq n \leq N$, recall that the $n$th centered straightening chart $\Psi^n : \cB^n_0 \to B^n_0$ extends vertically to a domain $\hcB^n_0 \supset \cB^n_0$ that is vertically proper in $\cB_0 \supset \cB^{n_0}_0$ (see \propref{vert prop nest} and \thmref{crit chart}). Let  $z = (a,b) \in B_0 = I^h_0 \times I^v_0$. For $t \geq 0$, define
$$
V_z(t) := [a - t, a+t] \times I^v_0.
$$
For $p \in \hcB^n_0$ and $t \geq 0$, denote
$$
\cV_p^n(t) := (\Psi^n)^{-1}(V_{\Psi^n(p)}(t)) \subset \hcB^n_0.
$$

We record the following immediate consequences of Lemmas \ref{est value} and \ref{root at value}, and  \eqref{eq.proper depth 1}.

\begin{lem}\label{quad flat}
For $n_0 \leq n \leq N$, let $E_{p_{-1}} \in \bbP^2_{p_{-1}}$ be a $\lambda^{\bepsilon R_n}$-horizontal direction at $p_{-1} \in \cB_{-1}$. If
$$
p_0 \in \hcB^n_0 \setminus \cV^n_{v_0}(\lambda^{\bepsilon R_n})
$$
then $E_{p_0}$ is $\lambda^{-\bepsilon R_n}$-horizontal in $\hcB^n_0$. Similarly, let $\Gamma_{-1}$ be $\lambda^{\bepsilon R_n}$-horizontal curve in $\cB_{-1}$. If
$$
\Gamma_0:= F(\Gamma_{-1}) \subset \hcB^n_0 \setminus \cV^n_{v_0}(\lambda^{\bepsilon R_n})
\matsp{with}
t > \lambda^{\bepsilon R_n},
$$
then $\Gamma_0$ is $\lambda^{-\bepsilon R_n}$-horizontal in $C^r$ in $\hcB^n_0$.
\end{lem}

\begin{lem}\label{quad straight}
For $n_0 \leq n \leq N$, let $\tiE_{p_0} \in \bbP^2_{p_0}$ be a $\lambda^{\bepsilon R_n}$-vertical direction at $p_0 \in \hcB^n_0$. If
$$
p_0 \in \hcB^n_{R_n} \setminus \cV^n_{v_0}(\lambda^{\bepsilon R_n}),
$$
then $\tiE_{p_0}$ is $\lambda^{-\bepsilon R_n}$-vertical in $\cB_{-1}$. Similarly, let $\tiGamma_0$ be $\lambda^{\bepsilon R_n}$-vertical curve in $\hcB^n_0$. If
$$
\tiGamma_0 \subset \hcB^n_{R_n} \setminus \cV^n_{v_0}(\lambda^{\bepsilon R_n}),
$$
then $\tiGamma_{-1} := F^{-1}(\tiGamma_0)$ is $\lambda^{-\bepsilon R_n}$-vertical in $C^r$ in $\cB_{-1}$.
\end{lem}

\begin{prop}\label{pull back vert in bn}
For ${n_0}\leq n \leq N$, let $p_0 \in \hcB^n_{R_n} \setminus \cV_{v_0}^n(\lambda^{\bepsilon R_n})$. If $E_{p_0}$ is $\lambda^{\bepsilon R_n}$-vertical in $\hcB^n_0$, then $E_{p_{-R_n}}$ is $\lambda^{(1-\bepsilon) R_n}$-vertical in $\hcB^n_0$. Moreover, $p_{-R_n}$ is $R_n$-times forward $(C{K_0}, \bepsilon, \lambda)$-regular along $E_{p_{-R_n}}$ for some uniform constant $C \geq 1$ independent of $F$. Consequently, if $p_{kR_n} \in \hcB^n_0 \setminus \cV_{v_0}^n(\lambda^{\bepsilon R_n})$ for all $k \in \bbN$, then $p_0$ is infinitely forward $(CK_0, \bepsilon, \lambda)$-regular.
\end{prop}

\begin{proof}
Consider a Q-linearization
$$
\{\Phi_{p_{-m}} : \cU_{p_{-m}} \to U_{p_{-m}}\}_{m=0}^{R_n}
$$
of $F$ along the $R_n$-backward orbit of $p_0$ with vertical direction
$$
E_{p_0}^{v,n} := (D\Psi^n)^{-1}\left(E_{\Psi^n(p_0)}^{gh}\right).
$$
Note that since $(F^{R_n}, \Psi^n)$ is a H\'enon-like return, we have
$$
D\Psi^n\left(E_{p_{-R_n}}^{v,n}\right) = E_{\Psi^n(p_{-R_n})}^{gv}.
$$

Denote
$$
E_{p_{-1}}^{h,n} := D\Phi_{p_{-1}}\left(E_0^{gh}\right)
\matsp{and}
E_{p_{-1}}^h := D\Phi_{-1}\left(E_{\Phi_{-1}(p_{-1})}^{gh}\right),
$$
where $\Phi_{-1} : \cU_{-1} \to U_{-1}$ is the chart defined over the critical point given in \thmref{crit chart}. By \thmref{reg chart} ii) and \eqref{eq.first entry}, we see that
$$
\|DF^{-R_n+1}|_{E_{p_{-1}}^{h,n}}\|\; , \; \|DF^{-R_n+1}|_{E_{p_{-1}}^h}\| < \lambda^{-\bepsilon R_n}.
$$
Hence, it follows from \propref{hor angle shrink} that
$$
\measuredangle(E_{p_{-1}}^{h,n}, E_{p_{-1}}^h) < \lambda^{(1-\bepsilon) R_n}.
$$
Thus, by \lemref{est value}, we have
$$
\measuredangle(E_{p_{-1}}^{h,n}, E_{p_{-1}}) > \lambda^{\bepsilon R_n}.
$$

For $1 \leq i \leq R_n$, denote
$$
\theta_{-i} := \measuredangle(E_0^{gh}, D\Phi_{p_{-i}}(E_{p_{-i}})).
$$
Choose a suitable uniform constant $c \in (0, \pi/2)$ independent of $F$, and let $1 \leq M \leq R_n$ be the smallest number such that $\theta_{-M} > c$. By \thmref{reg chart}, we see that
$$
\theta_{-i} > \lambda^{-(1-\bepsilon)i}\theta_{-1} > \lambda^{-(1-\bepsilon)i}\lambda^{\bepsilon R_n}.
$$
Consequently, $M < \bepsilon R_n$. Let $M' := CM$ for some suitable uniform constant $C \geq 1$ independent of $F$.

By \corref{ver hor cons}, we have
\begin{equation}\label{eq.pull back align}
\|DF^i|_{E_{p_{-R_n}}}\| \asymp \|DF^i|_{E_{p_{-R_n}}^{v,n}}\|
\matsp{for}
0 \leq i < R_n-M'
\end{equation}
By \propref{vert prop nest} iii), $p_{-R_n}$ is $R_n$-times forward $({K_0},\epsilon, \lambda)$-regular along $E^{v,n}_{p_{-R_n}}$. Hence, \eqref{eq.pull back align} implies that $p_{-R_n}$ is $(R_n-M')$-times forward $(C{K_0},\epsilon, \lambda)$-regular along $E_{p_{-R_n}}$.
By \propref{ext deriv bound}, we have
\begin{equation}\label{eq.short tame}
\lambda^{\bepsilon R_n}< \lambda^{(1+\bepsilon) M'}< \|DF^i|_{\tiE_{p_{-M'}}}\| < \lambda^{-\bepsilon M'} < \lambda^{-\bepsilon R_n}
\end{equation}
for any $\tiE_{p_{-M'}} \in \bbP^2_{p_{-M'}}$. We conclude that for $0\leq i < M'$, we have
$$
\lambda^{\bepsilon R_n} < \frac{\|DF^{R_n-M'+i}|_{E_{p_{-R_n}}^{v,n}}\|}{\|DF^{R_n-M'+i}|_{E_{p_{-R_n}}}\|} < \lambda^{-\bepsilon R_n}.
$$
The $(CK_0, \bepsilon, \lambda)$ forward regularity of $p_{-R_n}$ along $E_{p_{-R_n}}$ follows.
\end{proof}

\begin{prop}\label{ss in bn}
For ${n_0}\leq n \leq N$, let $p_0 \in \hcB^n_0$. If $p_0$ is infinitely forward $(\bar{K_0}, \bepsilon, \lambda)$-regular, then $W^{ss}(p_0)$ is $\lambda^{(1-\bepsilon)R_n}$-vertical and vertically proper in $\hcB^n_0$.
\end{prop}

\begin{proof}
The verticality of $W^{ss}(p_0)$ follows immediately from \propref{vert angle shrink}. Consider a Q-linearization
$$
\{\Phi_{p_m} : \cU_{p_m} \to U_{p_m}\}_{m=0}^{\infty}
$$
of $F$ along the infinite forward orbit of $p_0$ with vertical direction $E^{ss}_{p_0}$. By \thmref{stable}, we have
\begin{equation}\label{eq.ss in bn 1}
\Phi_{p_m}(W^{ss}_{\loc}(p_m)) \subset \{(0, y) \in U_{p_m} \; | \; y \in \bbR\}.
\end{equation}

Let
$$
\cV_{p_0} := \cV_{p_0}^n(\lambda^{\bepsilon R_n}) \subset \hcB^n_0.
$$

Let $M$ be a nearest integer to $R_n/2$. By \lemref{size reg nbh}, we see that $\cU_{p_M} \supset \bbD_{p_M}\left(\lambda^{\bepsilon M}\right).$ \propref{vert prop nest} iii) and \lemref{trunc neigh fit} implies that \corref{ver hor cons} applies to any point $q_0 \in \cV_{p_0}$. Consequently, $F^M(\cV_{p_0}) \subset \cU_{p_M}$, and 
\begin{equation}\label{eq.ss in bn 2}
\Phi_{p_M}(F^M(\cV_{p_0})) \subset (-\lambda^{\bepsilon R_n}, \lambda^{\bepsilon R_n}) \times (-\lambda^{(1-\bepsilon)M}, \lambda^{(1-\bepsilon)M}).
\end{equation}

For $q_0 \in \cV_{p_0}$, denote
$$
\hE^{v/h}_{q_m} := D(F^m\circ(\Psi^n)^{-1})\left(E_{\Psi^n(q_0)}^{gv/gh}\right).
$$
By Propositions \ref{vert prop nest} iii), \ref{ext deriv bound} and \ref{ver hor cons}, we have
$$
\|DF^m|_{\hE^v_{q_0}}\| < {K_0}\lambda^{(1-\bepsilon) m}
\matsp{and}
\|DF^m|_{\hE^h_{q_0}}\| > {K_0}^{-1}\lambda^{\bepsilon m}.
$$

Let $M_0 \leq M$ be the smallest number number such that
$$
{K_0}\lambda^{(1-\bepsilon)M_0} < \diam(\cU_{p_{M_0}}) \asymp {K_0}^{-1}\lambda^{\bepsilon M_0}.
$$
Then it follows from \propref{size reg nbh} that $q_m \in \cU_{p_m}$ for all $M_0 \leq m \leq R_n$. Define
$$
\tiE^{v/h}_{q_m} := D\Phi_{p_m}^{-1}\left(E_{\Phi_{p_m}(q_m)}^{gv/gh}\right).
$$
By \propref{grow irreg}, $p_{M_0}$ is infinitely forward $(\overline{K_0}\lambda^{-\bepsilon M_0}, \epsilon, \lambda)$-regular. Hence, by \thmref{reg chart} and \corref{ver hor cons}, we have
$$
\left\|DF^{(M - M_0)}|_{\tiE^h_{q_{M_0}}}\right\| > \lambda^{\bepsilon M}.
$$
Thus, \propref{hor angle shrink} implies that
$$
\measuredangle(\tiE^h_{q_M}, \hE^h_{q_M}) < \lambda^{(1-\bepsilon)M}. 
$$
On the other hand, since $\|D\Phi_{p_M}^{\pm 1}\| < \lambda^{-\bepsilon M}$, it follows that
$$
\measuredangle(E^{ss}_{q_M}, \hE^h_{q_M}) > \lambda^{\bepsilon M}.
$$

We conclude by \eqref{eq.ss in bn 1} and \eqref{eq.ss in bn 2} that $W^{ss}_{\loc}(p_M)$ is vertically proper in $F^M(\cV_{p_0})$.
\end{proof}

\begin{prop}\label{value or sink}
For $n_0 \leq n \leq N$, let $\cC_0 \subset \cB^n_0$ be a connected set totally invariant under $F^{d R_n}$ for some $d \in \bbN$ such that $d \bepsilon < 1$. Denote $\cC_m := F^m(\cC_0)$ for $m\in\bbZ$. If
$$
\cV_{v_0}^n(\lambda^{\bepsilon R_n}) \cap \cC = \varnothing,
\matsp{where}
\cC := \bigcup_{i=0}^{d-1}\cC_{iR_n},
$$
then every point in $\cC$ is infinitely forward regular. Moreover, there exists a chart $\Phi : \cD_0 \to D_0$ such that the following statements hold:
\begin{enumerate}[i)]
\item $\cC_0 \subset \cD_0$;
\item $\cD_0$ is $\lambda^{(1-\bepsilon)R_n}$-vertical and vertically proper in $\hcB^n_0$;
\item for $p \in \cC_0$, we have
$$
\Phi^{-1}(\{x\} \times \pi_v(D_0)) \subset W^{ss}(p)
\matsp{where}
\Phi(p) = (x,y) \in D_0;
\hspace{5mm}
and
$$
\item the map $H := \Phi \circ F^{dR_n}\circ \Phi^{-1}$ is of the form $H(x,y)= (h(x), v(x,y))$, where $h : \pi_h(D_0) \to \pi_h(D_0)$ is a diffeomorphism.
\end{enumerate}
\end{prop}

\begin{proof}
The infinite forward regularity of any point $p \in \cC$ is given in \propref{pull back vert in bn}. Let $W^{ss}_{\loc}(p)$ be the connected component of $W^{ss}(p) \cap \hcB^n_0$ containing $p$. Define
$$
\cD := \bigcup_{p \in \cC} W^{ss}_{\loc}(p).
\matsp{and}
\cD_0 := \bigcup_{p \in \cC_0} W^{ss}_{\loc}(p).
$$
By \propref{ss in bn}, the foliation of each component of $\cD$ given by the local strong-stable manifolds is $\lambda^{(1-\bepsilon)R_n}$-vertical and vertically proper in $\hcB^n_0$. Let $\Phi : \cD_0 \to D_0$ be a genuine horizontal chart that rectifies this vertical foliation over $\cD_0$. Then the map $H := \Phi \circ F^{dR_n}\circ \Phi^{-1}$ preserves the vertical foliation.

By \thmref{crit chart} ii), \eqref{eq.henon trans} and the fact that
$$
\cD \cap \cV_{v_0}^n(\lambda^{\bepsilon R_n}) = \varnothing,
$$
it follows that $h:= \Piod(H)$ is a diffeomorphism.
\end{proof}


\section{Combinatorial Structure of Renormalization}\label{sec.combin}

In this section, we show that at sufficiently deep renormalization depths (i.e. beyond the depth $n_0$ given by \eqref{eq.proper depth 0}), the dynamical structure of a 2D H\'enon-like map closely resembles that of a 1D unimodal map. See \figref{fig.valuestruct}. In particular, we prove that topological renormalizations at these depths with combinatorics of bounded type are guaranteed to be regular H\'enon-like, as long as they are not {\it trivial} in the sense defined below.

Following \cite{CPTr}, we say that a diffeomorphism (in any dimension) is {\it generalized Morse-Smale (of $\theta$-bounded type)} for some $\theta \in \bbN$ if
\begin{itemize}
\item the $\omega$-limit set of any forward orbit is a periodic point;
\item the $\alpha$-limit set of any backward orbit is a periodic point; and
\item all periodic orbits have periods bounded by $\theta$.
\end{itemize}
Note that a diffeomorphism of an interval to itself is generalized Morse-Smale of either $1$-bounded type if orientation-preserving, or $2$-bounded type if orientation-reversing. A renormalization of a map is {\it trivial} if the induced return map is a generalized Morse-Smale diffeomorphism of $2$-bounded type.

\subsection{For unimodal maps}\label{subsec.ren combin 1d}

\begin{lem}\label{sink combin}
Consider a unimodal map $f : I \to I$ with the critical point $c \in I$. Suppose that $f''(c) > 0$. Then the following statements hold.
\begin{enumerate}[i)]
\item If $f(c) > c$, then $c$ converges to either a fixed attracting or parabolic sink.
\item If $f^2(c) < c$, then $c$ converges to either a fixed or $2$-periodic attracting or parabolic sink.
\item If $f^3(c) > f^2(c)$, then $c$ converges to a fixed attracting or parabolic sink.
\end{enumerate}
If none of the above cases hold, then $J := [f(c), f^2(c)]$ is the minimal invariant interval containing $c$.
\end{lem}

Consider a unimodal map $f : I \to I$ with the critical point $c \in I$. For concreteness, assume that $f''(c) > 0$. For $\theta \geq 1$, we say that $f$ has {\it $\theta$-bounded kneading} if $f(c) < c$ and $f^{1+\theta}(c) < c$. Recall that $f$ is valuably renormalizable if there exists an $R$-periodic interval $I^1 \subset I$ for some integer $R \geq 2$ such that $f^R(I^1)$ contains the critical value $v$ for $f$. Note that in this case, $f$ has $\theta$-bounded kneading for some $\theta \leq R$. The {\it renormalization type $\tau(f)$ of $f$} is given by the order of points in $\{f^i(v)\}_{i=0}^{R-1} \subset I$. If $f$ is $N$-times valuably renormalizable, then its {\it $N$-combinatorial type} is defined as
$$
\tau_N(f) := (\tau(f), \tau(\cR(f)), \ldots, \tau(\cR^{N-1}(f))).
$$

\begin{lem}\label{1d struct}
Let $f : I \to I$ be a unimodal map with critical value $v$. If $f$ is non-trivially topologically renormalizable with return time $R \geq 2$, then $f$ is valuably renormalizable. In this case, $I^1 := [v, f^R(v)]$ is the minimal $R$-periodic interval containing $v$.
\end{lem}

\subsection{For H\'enon-like maps}\label{subsec.ren combin 2d}

For some $N \in \bbN \cup \{\infty\}$, let $F$ be the $N$-times regularly H\'enon-like renormalizable map considered in \secref{sec.crit avoid}. Suppose that the combinatorics of  renormalizations of $F$ are of $\bfb$-bounded type for some integer $\bfb \geq 2$. Moreover, assume that $\epsilon$ is sufficiently small so that $\bfb\bepsilon < 1$. By only considering every other returns if necessary, we may also assume without loss of generality that $r_n := R_{n+1}/R_n \geq 3$ for ${n_0}\leq n \leq N$.

Let $z = (a,b), w = (c,d) \in B_0 = I^h_0 \times I^v_0$. Denote
$$
m := \min\{a, c\}
\matsp{and}
M := \max\{a, c\}.
$$
For $t \geq 0$, define
$$
V_{[z,w]}(t) := [m -t, M + t] \times I^v_0.
$$
For $n_0 \leq n \leq N$; $p, q \in \hcB^n_0$ and $t \geq 0$, denote
$$
\cV_{[p, q]}^n(t) := (\Psi^n)^{-1}(V_{[\Psi^n(p), \Psi^n(q)]}(t)).
$$

Let $s \in \{0, 1, 2\}$. For $n_0 \leq n \leq N-s$ and $k \geq -1$, denote
$$
a^n_k := P^n_0(v_{kR_n})
\matsp{and}
b^{n,s}_k:= P^n_0(v_{kR_n+R_{n+s}}) = a^n_{k+R_{n+s}/R_n}.
$$
Define
$$
\hB^{n,s}_{kR_n} := V_{[a^n_k, b^{n,s}_k]}(\lambda^{\bepsilon R_n})
\matsp{and}
\hcB^{n,s}_{kR_n} := (\Psi^n)^{-1}(\hB^{n,s}_{kR_n}).
$$
In particular, we have
$$
\hcB^n_0 \supset \hcB^{n,0}_0 := \cV^n_{[v_0, v_{R_n}]}(\lambda^{\bepsilon R_n}).
$$
See \figref{fig.valuestruct}.

\begin{lem}\label{quick return}
Let $n_0 \leq n\leq N$. Suppose that $F^{R_n}|_{\cB^n_0}$ is non-trivially topologically renormalizable with combinatorics of $\bfb$-bounded type. Then 
$$
\cV_{v_{kR_n}}^n(\lambda^{\bepsilon R_n}) \cap \cV_{v_{-R_n}}^n(\lambda^{\bepsilon R_n}) =\varnothing
\matsp{for}
k = O(\bfb).
$$
\end{lem}

\begin{proof}
Let $\delta \in (\bepsilon, 1)$ with $\bfb\bdelta < 1$. Suppose towards a contradiction that for some $k = O(\bfb)$, we have
\begin{equation}\label{eq.0, -1}
V_{v_{kR_n}}(\lambda^{\bdelta R_n}) \cap V_{v_{-R_n}}(\lambda^{\bdelta R_n})\neq\varnothing.
\end{equation}
Without loss of generality, assume that $k$ is the smallest number for which \eqref{eq.0, -1} holds.

For $y \in I^v_0$, consider
$$
J^n_0 \subset (-\lambda^{\bdelta R_n}, +\lambda^{\bdelta R_n})
\matsp{and}
\cJ^n_0 := \Psi^{-n}(J^n_0 \times\{y\}) \subset \cV_{v_0}^n(\lambda^{\bdelta R_n}).
$$
For $i \in \bbN$, denote $\cJ^n_i := F^i(\cJ^n_0)$.

By \propref{ext deriv bound}, we see that
$$
|\cJ^n_{iR_n-1}| < \lambda^{-i\bepsilon R_n}|J^n_0| < \lambda^{\udelta R_n}
\matsp{for}
i = O(\bfb).
$$
Moreover, since
$$
\cJ^n_{iR_n} \cap V_{v_{-R_n}}(\lambda^{\bdelta R_n}) = \varnothing
\matsp{for}
i < k,
$$
we can argue by induction using \lemref{flat} iii) and \lemref{quad flat} that $\cJ^n_{(k+1)R_n-1}$ is $\lambda^{(1-\bepsilon)R_n}$-horizontal in $C^r$ in $\cB_{-1}$. Then it follows from \eqref{eq.0, -1}, \eqref{eq.henon trans} and \lemref{est value} that
$$
|P^n_0(\cJ^n_{(k+1)R_n})| < \lambda^{\udelta R_n}|\cJ^n_{(k+1)R_n-1}| < \lambda^{\udelta R_n}|J^n_0|.
$$
We conclude that
$$
F^{(k+1)R_n}(\cV_{v_0}^n(\lambda^{\bdelta R_n})) \Subset \cV_{v_0}^n(\lambda^{\bdelta R_n}).
$$

Denote $E^h_{p_i} := DF^i(E_{p_0}^h)$ for $i \in \bbN$. Arguing by induction using \lemref{flat} i) and \lemref{quad flat}, we also see that $E_{p_{(k+1)R_n-1}}^h$ is $\lambda^{(1-\bepsilon)R_n}$-horizontal in $\cB_{-1}$. Consequently, by \eqref{eq.0, -1}, \eqref{eq.henon trans} and \lemref{est value}, we have
$$
\measuredangle(E^h_{p_{(k+1)R_n}}, E^{v,n}_{p_{(k+1)R_n}}) < \lambda^{\udelta R_n}.
$$
Hence, by \thmref{crit chart} ii), it follows that 
$$
\|DF^{R_n}|_{E^h_{p_{(k+1)R_n}}}\| < \lambda^{\udelta R_n}.
$$
Since, by \propref{ext deriv bound}, we have
$$
\|D_{p_0}F^{(k+1)R_n}\| < \lambda^{-\bepsilon R_n},
$$
we conclude that
$$
\|D_{p_0}F^{(k+2)R_n}\| \asymp \|D_{p_0}F^{(k+2)R_n}|_{E^h_{p_0}}\| < \lambda^{-\bepsilon R_n}\lambda^{\udelta R_n} = \lambda^{\udelta R_n}.
$$
Applying \propref{ext deriv bound} again, we have
$$
\|D_{p_0}F^{2(k+1)R_n}\| \leq \|D_{p_0}F^{(k+2)R_n}\|\cdot\|D_{p_{k+2(R_n)}}F^{kR_n}\| < \lambda^{\udelta R_n}\lambda^{-\bepsilon R_n} = \lambda^{\udelta R_n}.
$$
Thus, under $F^{2(k+1)R_n}$, there exists a unique fixed sink $q_0$ that attracts the orbit of every point in $\cV^n_{v_0}(\lambda^{\bepsilon R_n})$. Since $\cV^n_{v_0}(\lambda^{\bepsilon R_n})$ maps into itself under $F^{(k+1)R_n}$, we see that $q_0$ is fixed under $F^{(k+1)R_n}$.

Denote
$$
\cV^n_{v_0} := \cV^n_{v_0}(\lambda^{\bepsilon R_n})
\matsp{and}
R_{n+1} := (k+1)R_n.
$$
The set $\partial \cV^n_{v_0} \setminus \partial \hcB^n_0$ consists of two vertical curves $\gamma^{l, 0}$ and $\gamma^{r, 0}$: the former to the left and the latter to the right of $v_0$.

For $i \geq 0$, let $\cB^{n+1, -i}_0$ be the component of $F^{-i R_{n+1}}(\cV^n_{v_0}) \cap \hcB^n_0$ that contain $\cV^n_{v_0}$. Proceeding inductively, suppose that $\cB^{n+1, -i}_0$ is a quadrilateral such that $\partial \cB^{n+1, -i}_0 \setminus \partial \hcB^n_0$ consists of two vertical curves $\gamma^{l, -i}$ and $\gamma^{r, -i}$:  the former to the left of $\gamma^{l, -i+1}$ and the latter to the right of $\gamma^{r, -i+1}$. Let $y \in I^v_0$, and denote
$$
\cI^{-i, y}_0 := (\Psi^n)^{-1}(I^n_0 \times \{y\})\cap \cB^{n+1, -i}_0
 \matsp{and}
\cI^{-i, y}_j := F^j(\cI^{-i, y}_0)
\matsp{for}
j \in \bbN.
$$
Arguing as above, we see that $\cI^{-i,y}_{R_{n+1}-1}$ is $\lambda^{(1-\bepsilon)R_n}$-horizontal in $\cB_0$, and hence, $\cI^{-i,y}_{R_{n+1}}$ is vertical quadratic in $\hcB^n_0$ whose endpoints are contained in $\gamma^{r, -i}$. Moreover, by \lemref{flat} iv) and \lemref{quad straight}, we also see that the induction hypothesis holds for $\cB^{n+1, -i-1}_0$.

By \propref{pull back vert in bn}, we see that as $i \to \infty$, the curves $\gamma^{l, -i}$ and $\gamma^{r, -i}$ converge to subarcs $\gamma^l$ and $\gamma^r$ respectively of the strong-stable manifold $W^{ss}(w_0)$ of a $R_{n+1}$-periodic saddle $w_0 \in \gamma^r := W^{ss}_{\loc}(y_0)$. Moreover, $\gamma^l$ and $\gamma^r$ are $\lambda^{(1-\bdelta)R_n}$-vertical and vertically proper in $\hcB^n_0$. It follows that these curves bound the immediate basin $\cB^{n+1}_0 \subset \hcB^n_0$ of $q_0$.

Let $R_n < \hR_{n+1} \leq \bfb R_n$. We claim that any $\hR_{n+1}$-periodic Jordan domain $\cD^{n+1}_0 \Subset \cB^n_0$ induces a trivial renormalization of $F^{R_n}|_{\cB^n_0}$.

If $\cD^{n+1}_0 \cap \cB^{n+1}_0 = \varnothing$, then by \propref{value or sink}, $\cD^{n+1}_0$ induces a trivial renormalization of $F^{R_n}|_{\cB^n_0}$. Assume that $\cD^{n+1}_0 \cap \cB^{n+1}_0 \neq \varnothing$. Then $\cD^{n+1}_0\ni q_0$ and $\hR_{n+1} = R_{n+1}$. Define
$$
\cC^{n+1}_0 := \bigcap_{i \in \bbN} F^{iR_{n+1}}(\cD^{n+1}_0).
$$
By \propref{pull back vert in bn}, every point $p \in \cC^{n+1}_0$ that does not eventually map into $\cB^{n+1}_0$ is infinitely forward $(CK_0, \bdelta, \lambda)$-regular. Let $W^{ss}_{\loc}(p)$ be the connected component of $W^{ss}(p) \cap \hcB^n_0$ that contain $p$, which must be $\lambda^{(1-\bdelta)R_n}$-vertical and vertically proper in $\hcB^n_0$ by \propref{ss in bn}.

Let $p^l_0$ and $p^r_0$ be the leftmost and the rightmost points in $\cC^{n+1}_0$ respectively. Let $\hcD^{n+1}_0$ be the quadrilaterals vertically proper in $\hcB^n_0$ enclosed between $W^{ss}_{\loc}(p^l_0)$ and $W^{ss}_{\loc}(p^r_0)$, so that $\cC^{n+1}_0 \subset \hcD^{n+1}_0$. Similarly, let $\cE^{n+1}_0$ be the quadrilaterals vertically proper in $\hcB^n_0$ enclosed between $W^{ss}_{\loc}(w_0)$ and $W^{ss}_{\loc}(p^r_0)$.

Observe that
\begin{equation}\label{eq.quick return}
\hcD^{n+1}_{iR_n} \cap \cV^n_{v_0}(\lambda^{\bepsilon R_n}) = \varnothing
\matsp{for}
1 \leq i < R_{n+1}/R_n.
\end{equation}
Let $\cJ_0$ be a genuine horizontal arc contained in $\cE^{n+1}_0$. Using \eqref{eq.quick return}, a similar argument as above shows that $\cJ_{iR_n}$ for $0 \leq i < R_{n+1}/R_n$ is $\lambda^{-\bdelta R_n}$-horizontal in $\hcB^n_0$. Consequently,
$$
F^{R_{n+1}}(\cE^{n+1}_0) \cap \cB^{n+1}_0 =\varnothing.
$$
It follows that $\cE^{n+1}_0$ is connected and $F^{R_{n+1}}(\cE^{n+1}_0) = \cE^{n+1}_0$. Applying \propref{value or sink}, we conclude that $\cD^{n+1}_0$ induces a trivial renormalization of $F^{R_n}|_{\cB^n_0}$. This is a contradiction, and therefore, \eqref{eq.0, -1} cannot hold.
\end{proof}

\begin{prop}\label{min domain}
Let $n_0 \leq n \leq N$. Suppose that $F^{R_n}|_{\cB^n_0}$ is twice non-trivially topological renormalizable with combinatorics of $\bfb$-bounded type. Then the following statements hold:
\begin{enumerate}[i)]
\item $|a^n_i - a^n_j| > \lambda^{\bepsilon R_n}$ for $-1 \leq i, j \leq 2$ with $i \neq j$;
\item $a^n_0 = 0< a^n_m < a^{n}_1 $ for $m \in \{-1, 2\}$; and
\item $F^{R_n}(\hcB^{n,0}_0) \Subset \hcB^{n,0}_0$.
\end{enumerate}
\end{prop}

\begin{proof}
For $0 \leq i \leq \bfb$, mapping $u^n_{-1}$ and $u^n_i$ by $F^{R_n}$, and applying \thmref{crit chart}, it follows from \lemref{quick return} that
\begin{equation}\label{eq.0,i}
|a^n_0 -a^n_{i+1}| > \lambda^{\bepsilon R_n}|a^n_{-1} - a^n_i| > \lambda^{\bepsilon R_n}.
\end{equation}

Now, suppose towards a contradiction that $a^n_0 = 0 < a^n_{-1} < a^{n}_1$ is not true. Then we have
$$
\lambda^{\bepsilon R_n} <a^{n}_1 < a^n_{-1}- \lambda^{\bepsilon R_n}.
$$
Denote
$$
J^n_0 := [-\lambda^{\bepsilon R_n}, a^n_{-1} - \lambda^{\bepsilon R_n}]
\comma
\chB^n_0 := J^n_0\times I^v_0
\matsp{and}
\chcB^n_0 := (\Psi^n)^{-1}(\chB^n_0).
$$
By \lemref{flat} iii) and \eqref{eq.henon trans}, we see that
$$
F^{R_n}(\chcB^n_0) \Subset \chcB^n_0 \setminus \cV^n_{v_0}(\lambda^{\bepsilon R_n}).
$$
It follows from \propref{value or sink} that $\chcB^n_0$ induces a trivial renormalization of $F$.

Define
$$
\chcD^n_0 = \chcD^{n,0}_0 := \chcB^n_0 \cup \cV^n_{v_{-R_n}}(\lambda^{\bepsilon R_n}).
$$
Observe that $F^{R_n}(\chcD^n_0)\Subset \chcB^n_0$. For $i \in \bbN$, let
$$
\cD^{n,-i}_0 := (F^{R_n}|_{\hcB^n_0})^{-1}(\chcD^{n, -i+1}_0).
$$
Arguing as in the proof of \lemref{quick return}, we see that $\cD^{n,-i}_0$ is a quadrilateral vertically proper in $\hcB^n_0$ that is bounded between two $\lambda^{(1-\bepsilon)R_n}$-vertical curves $\gamma^{l, -i}$ and $\gamma^{r, -i}$:  the former to the left of $\gamma^{l, -i+1}$ and the latter to the right of $\gamma^{r, -i+1}$. Moreover, there exists an $R_n$-periodic saddle $w_0$ such that $\gamma^{r, -i}$ converges to the local strong-stable manifold $W^{ss}_{\loc}(w_0)$ as $i\to \infty$.

Write
$$
\hcB^n_0 \setminus W^{ss}_{\loc}(w_0) = \cD^n_0 \sqcup \cE^n_0,
$$
where $\cD^n_0$ is the connected component that contains $\chcD^n_0$. Observe that  $F^{R_n}(\cE^n_0) \subset \cE^n_0$. Applying \propref{value or sink}, we conclude that $\hcB^n_0$ induces a trivial renormalization of $F$, which is a contradiction.

Next, suppose towards a contradiction that
$$
|a^{n}_1 - a^n_2| < \lambda^{\bepsilon R_n}.
$$
By the $R_n$-times regularity of $F$ on $\hcB^n_0$, we conclude that
\begin{equation}\label{eq.1,2 after}
\|v_{kR_n}-v_{(k+1)R_n}\| < \lambda^{-\bepsilon R_n}|a^{n}_1 - a^n_2| < \lambda^{\bepsilon R_n}
\matsp{for}
1 \leq k < \bfb.
\end{equation}

For $l \in \{1, 2\}$, let $\cB^{n+l}_0 \Subset \cB^{n+l-1}_0$ be the $R_{n+l}$-periodic Jordan domain that induces a non-trivial renormalization of $F^{R_n}|_{\cB^n_0}$. By \propref{value or sink}, we may assume without loss of generality that a $\lambda^{\bepsilon R_n}$-neighborhood of $\cB^{n+l}_0$ contains $v_0$ (and hence, also $v_{R_{n+l}}$).

\propref{exist saddle} implies the existence of a saddle point $x_0 \in \cB^{n+1}_0$ of period $dR_{n+1}$ for some $d \leq r_{n+1}$. For $0 \leq i < R_{n+1}$, let $W^{ss}_{\loc}(x_{iR_n})$ be the connected component of $W^{ss}(x_{iR_n}) \cap \hcB^n_0$ containing $x_{iR_n}$. Clearly,
\begin{equation}\label{eq.disjoint}
W^{ss}_{\loc}(x_{iR_n}) \cap \cB^{n+1}_{jR_n} = \varnothing
\end{equation}
for $0 \leq j < r_n$ such that $i\neq j$. Moreover, it follows from Propositions \ref{ss in bn} and \ref{value or sink} that $W^{ss}_{\loc}(x_{iR_n})$ is $\lambda^{(1-\bepsilon)R_n}$-vertical and vertically proper in $\hcB^n_0$.

Mapping $v_{R_{n+1}}$ and $v_0$ by $F^{R_n}$ and applying \thmref{crit chart}, it follows from \eqref{eq.0,i} that
$$
|a^{n}_1 - a^n_{r_n+1}| > \lambda^{\bepsilon R_n}.
$$
Thus, by \eqref{eq.1,2 after}, we have $a^n_{r_n+1} < a^n_k < a^{n}_1$ for $2 \leq k \leq r_n$. This contradicts \eqref{eq.disjoint}. Therefore, claim i) holds.

Suppose towards a contradiction that $a^n_2 < a^{n}_1$ is not true. Then we have
$$
a^{n}_1 < a^n_2- \lambda^{\bepsilon R_n}.
$$
Let $y \in I^v_0$, and consider
$$
A^n_0 := [a^n_{-1} + \lambda^{\epsilon R_n}, a^{n}_1]
\matsp{and}
\cA^n_0 := (\Psi^n)^{-1}(A^n_0\times \{y\}).
$$
Denote $\cA^n_i := F^i(\cA^n_0)$. Applying \thmref{crit chart} and \lemref{quad flat}, we see that $\cA^n_{R_n}$ is $\lambda^{-\bepsilon R_n}$-horizontal in $\hcB^n_0$. This means that  $f_n := \Piod(F_n)$, where
$$
F_n := \Psi^n\circ F^{R_n} \circ(\Psi^n)^{-1},
$$
maps $A^n_0$ onto $f_n(A^n_0)$ as an orientation-preserving diffeomorphism. Moreover,
$$
f_n(A^n_0) \supset [\lambda^{\bepsilon R_n}, a^n_2 -\lambda^{(1-\bepsilon)R_n}] \Supset A^n_0.
$$

Let
$$
L^n_0 := \{(a^{n}_1, t) \in \hcB^n_0\}
\matsp{and}
\cL^n_0 := (\Psi^n)^{-1}(L^n_0).
$$
For $i \in \bbN$, define
$$
\cL^n_{-i} := F^{-R_n}(\cL^n_{-i+1}\cap \cA^n_{R_n}).
$$
Applying \lemref{quad straight} and \lemref{flat} iv) and arguing by induction, we see that $\cL^n_{-i}$ is $\lambda^{(1-\bepsilon)R_n}$-vertical and vertically proper in $\hcB^n_0$. Moreover, by \lemref{pull back vert in bn}, we see that any point $p \in \cL^n_{-i}$ is $iR_n$-times forward $(CK_0, \bepsilon, \lambda)$-regular along the tangent direction to $\cL^n_{-i}$ at $p$. It follows that $\cL^n_{-i}$ converges as $i \to \infty$ to the local strong-stable manifold of some $R_n$-periodic saddle $z_0$. Let $\cB^{n, r}_0$ be the connected components of $\cB^n_0 \setminus W^{ss}(z_0)$ containing $v_{R_n}$. It follows that $F^{R_n}(\cB^{n, r}_0) \subset \cB^{n, r}_0$. This is a contradiction. Claim ii) now follows.

\thmref{crit chart} implies that
$$
F^{R_n}(\cV_{[v_0, v_{-R_n}]}^n(\lambda^{\bepsilon R_n})) \Subset \cV_{[v_0, v_{R_n}]}^n(\lambda^{\bepsilon R_n}).
$$
We similarly conclude that
$$
F^{R_n}(\cV_{[v_{-R_n}, v_{R_n}]}^n(\lambda^{\bepsilon R_n})) \Subset \cV_{[v_0, v_{2R_n}]}^n(\lambda^{\bepsilon R_n}).
$$
Thus, claim iii) holds.
\end{proof}

\begin{defn}
For $n_0 \leq n < N$ and $1 \leq s \leq N-n$, we say that the H\'enon-like return $(F^{R_n}, \Psi^n)$ has {\it 1D-like structure of depth $s$} if:
\begin{enumerate}[i)]
\item $\hcB^{n, s}_{lR_n} \cap \hcB^{n, s}_{kR_n} = \varnothing$ for $0 \leq l,k < R_{n+s}/R_n$ with $l \neq k$;
\item $a^n_0 = 0< b^{n, s}_0< a^n_k  \, ,\, b^{n, s}_k< b^{n,s}_1  < a^n_1  $ for $2 \leq k < R_{n+s}/R_n$; and
\item $F^{R_n}(\hcB^{n,s}_{kR_n}) \Subset \hcB^{n,s}_{(k+1)R_n \md{R_{n+s}}}$ for $0 \leq k < R_{n+s}/R_n$.
\end{enumerate}
\end{defn}

\begin{prop}\label{value struct bound}
Let $n_0 \leq n \leq N$. Suppose that $F^{R_n}|_{\cB^n_0}$ is twice non-trivially topological renormalizable with combinatorics of $\bfb$-bounded type. Then for $m = n-s$ with $s = O(1)$, the H\'enon-like return $(F^{R_m}, \Psi^m)$ has 1D-like structure of depth $s$. In particular, $\hcB^{n, 0}_0$ is $R_n$-periodic.
\end{prop}

\begin{figure}[h]
\centering
\includegraphics[scale=0.25]{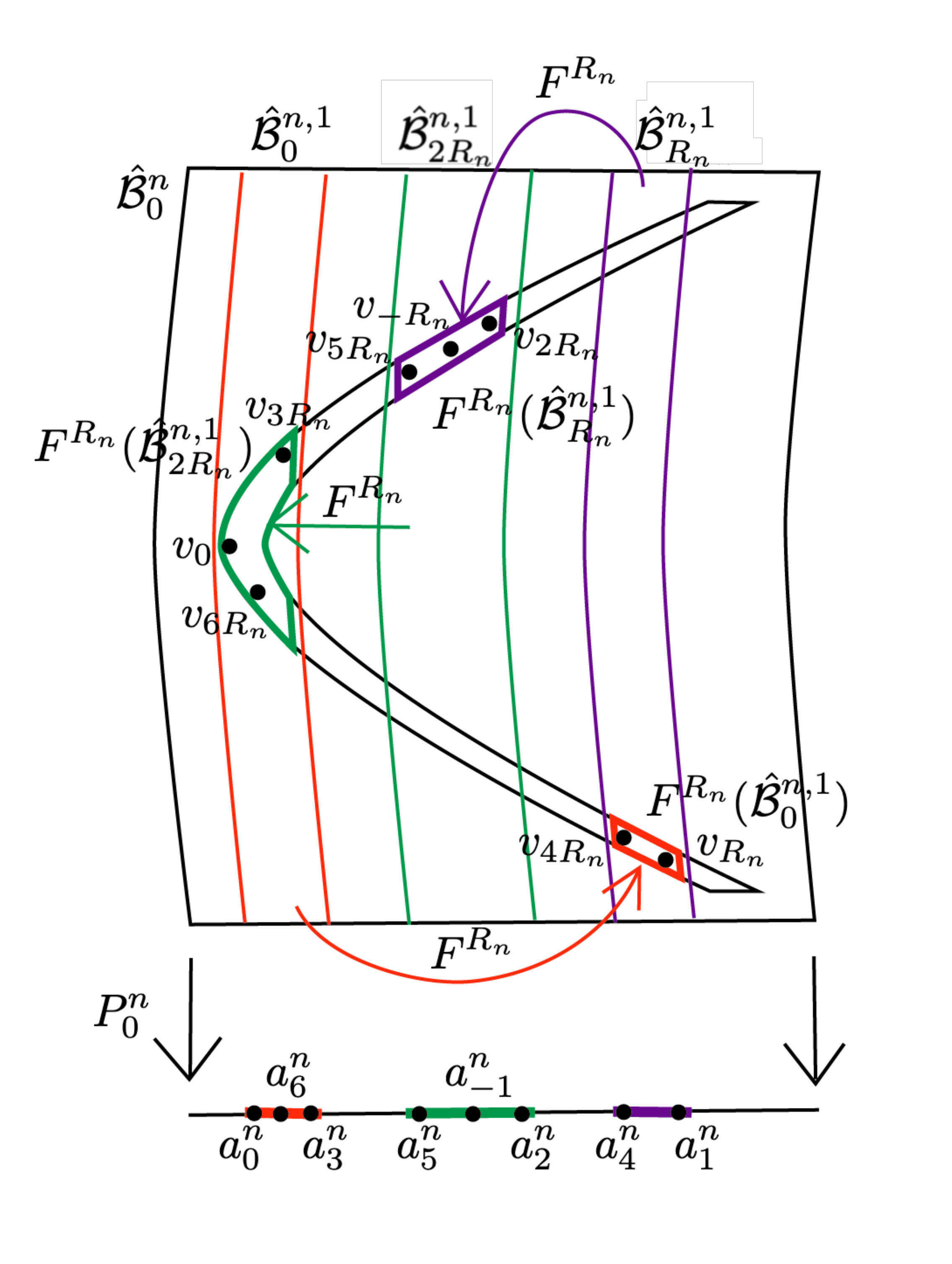}
\caption{The combinatorial structure of the $n$th renormalization of $F$ for $n \geq n_0$ (for $r_n := R_{n+1}/R_n = 3$). The $R_{n+1}$-periodic domains $\hcB^{n,1}_0$, $\hcB^{n,1}_{R_n}$ and $\hcB^{n,1}_{2R_n}$ containing $v_0$, $v_{R_n}$ and $v_{2R_n}$ respectively are vertically proper and pairwise disjoint in $\hcB^n_0$. Moreover, $F^{R_n}(\hcB^{n,1}_{kR_n}) \Subset \hcB^{n,1}_{(k+1)R_n \md{R_{n+1}}}$. Under the projection $P^n_0 : \hcB^n_0 \to I^n_0 \subset \bbR$, the orbit $\{v_{kR_n}\}_{k=-1}^{r_n}$ of the critical value are mapped to $\{a^n_k\}_{k=-1}^{r_n}$.}
\label{fig.valuestruct}
\end{figure}

\begin{proof}
Suppose towards a contradiction that for some $1 \leq j, i <2r_m$ with $j < i$, we have
$$
|a^m_i - a^m_j| < \lambda^{\bepsilon R_m}.
$$
Applying $F^{(2r_m - i)R_m}$ to $v_{iR_m}$ and $v_{jR_m}$, we see by the regularity of $F$ that
$$
\|v_{2R_{m+1}} - v_{j'R_m}\| < \lambda^{\bepsilon R_m},
$$
where $j' := j+2r_m-i$. Note that $j' \neq 0 \md{r_m}$. Hence, 
$$
\cB^{m+1}_{j'R_m} \cap \cB^{m+1}_0 = \varnothing.
$$
Moreover, by \propref{min domain}, we have
\begin{equation}\label{eq.m,j}
\lambda^{\bepsilon R_m} < a^m_{2r_m} - \lambda^{\bepsilon R_m} < a^m_{j'} <a^m_{2r_m} + \lambda^{\bepsilon R_m}  <a^m_{r_m}.
\end{equation}

\propref{exist saddle} implies the existence of a saddle point $x_0 \in \cB^{m+2}_0$ of period $dR_{m+2}$ for some $d \leq r_{m+2}$. For $0 \leq k < R_{m+2}/R_m$, let $W^{ss}_{\loc}(x_{kR_m})$ be the connected component of $W^{ss}(x_{kR_m}) \cap \hcB^m_0$ containing $x_{kR_m}$. Then \eqref{eq.m,j} implies that $W^{ss}_{\loc}(x_{j'R_m})$ intersect $\cB^{m+1}_0$. This is a contradiction. Claim i) follows. Then Claim ii) and iii) follow by \propref{min domain} and \thmref{crit chart} respectively.
\end{proof}

By \propref{value struct bound}, we may henceforth assume without loss of generality that for all $n_0 \leq n \leq N$ such that $F^{R_n}|_{\cB^n_0}$ is twice non-trivially renormalizable, we have
\begin{equation}\label{eq.min domain}
\cB^n_0 := \hcB^{n,0}_0 := \cV^n_{[v_0, v_{R_n}]}(\lambda^{\bepsilon R_n}).
\end{equation}

Let $0 \leq n < N$. Suppose that $F^{R_{n+1}}|_{\cB^{n+1}_0}$ is twice non-trivially topological renormalizable with combinatorics of $\bfb$-bounded type. Consider the sequence of points
\begin{equation}\label{eq.proj value seq}
\{a_k^n := \pi_h\circ \Psi^n(v_{kR_n})\}_{k=0}^{r_n-1} \subset I^n_0 := \pi_h(B^n_0) \subset \bbR.
\end{equation}
The {\it renormalization type $\tau(\cR^n(F))$ of $\cR^n(F)$} is given by the order of points in \eqref{eq.proj value seq}. Additionally, the {\it $N$-renormalization type of $F$} is defined as
$$
\tau_N(F) := (\tau(F), \tau(\cR(F)), \ldots, \tau(\cR^{N-1}(F))).
$$

\begin{lem}\label{unproject}
Let $n_0 \leq n \leq N$. Suppose that $F^{R_n}|_{\cB^n_0}$ is twice topological renormalizable with combinatorics of $\bfb$-bounded type. For $s\in \{1, 2\}$ and $m := n-s$, let $\Gamma_0$ be a $\lambda^{-\bepsilon R_m}$-horizontal curve in $\cB^n_0$. Then the following statements hold for $1 \leq k \leq R_n/R_m$:
\begin{enumerate}[i)]
\item $\Gamma_{(k-1)R_m}$ is $\lambda^{-\bepsilon R_m}$-horizontal in $\cB^m_0$; and
\item $\Gamma_{kR_m-1}$ is $\lambda^{(1-\bepsilon)R_m}$-horizontal in $\cB_{-1}$.
\end{enumerate}
\end{lem}

\begin{proof}
The result is an immediate consequence of Lemmas \ref{flat} iii) and \ref{quad flat}, and \propref{value struct bound}.
\end{proof}

\begin{thm}[H\'enon-likeness of deep returns]\label{ren is henon like}
Suppose that $F^{R_N}|_{\cB^N_0}$ is three times non-trivially topologically renormalizable with combinatorics of $\bfb$-bounded type. Then $F$ is $(N+1)$-times $(CK_0, \bepsilon, \lambda)$-regularly H\'enon-like renormalizable, where and $C, K_0 \geq 1$ are uniform constants (the former independent of $F$, and the latter given in \eqref{eq.const 0}).
\end{thm}

\begin{proof}
For $l \in \{1, 2\}$, let $\cB^{N+l}_0 \Subset \cB^{N+l-1}_0$ be an $R_{N+l}$-periodic Jordan domain with
$$
r_{N+l-1} := R_{N+l}/R_{N+l-1} \leq \bfb.
$$
Define
$$
\cC^{N+l}_0 := \bigcap_{i=1}^\infty F^{i R_{N+l}}(\cB^{N+l}_0)
\matsp{and}
\cC^{N+l} := \bigcup_{i=0}^{R_{N+l}/R_N-1}F^{iR_N}(\cC^{N+l}_0).
$$
By \propref{value or sink}, we see that
$$
\cV_{v_0}^N(\lambda^{\bepsilon R_N}) \cap \cC^{N+l} \neq \varnothing.
$$
Without loss of generality, assume that
$$
\cV_{v_0}^N(\lambda^{\bepsilon R_N}) \cap \cC^{N+l}_0 \neq \varnothing.
$$
By \eqref{eq.first entry} and \propref{ext deriv bound}, it follows that
$$
\dist(v_{R_{N+l}}, \cC_0^{N+l}) <\lambda^{\bepsilon R_N}.
$$

For $m \geq -1$, let
$$
a^N_m := \pi_h \circ \Psi^N(v_{mR_N}).
$$
Denote
$$
F_{N+1} := \Psi^N\circ F^{R_{N+1}} \circ (\Psi^N)^{-1}
\matsp{and}
f_{N+1} := \Piod(F_{N+1}).
$$
Note that $a^N_0 = 0$ and $a^N_{r_N} = f_{N+1}(0)$. Moreover, by \eqref{eq.first entry}, we see that
$$
|a^N_{(i+1)r_N} - f_{N+1}(a^N_{ir_N})| < \lambda^{(1-\bepsilon)R_N}
\matsp{for}
i \in \bbN.
$$

Define
$$
J_0 := (-\lambda^{\bepsilon R_N}, a^N_{r_N}+\lambda^{\bepsilon R_ N})
\matsp{and}
\cD_0 := (\Psi^N)^{-1}(J_0 \times I^v_0).
$$
Denote $\cD_i := F^i(\cD_0)$ for $i \in \bbN$. By the regularity of $F$, we see that $\cD_{iR_N}$ is contained in the $\lambda^{\bepsilon R_N}$-neighborhood of $\cB^{N+1}_{iR_N}$ for all $1 \leq i \leq \bfb$. We claim that
\begin{equation}\label{eq.hit crit}
a^N_{-1} \in \pi_h \circ \Psi^N(\cD_{(r_N-1)R_N}).
\end{equation}

Suppose towards a contradiction that we have
\begin{equation}\label{eq.hit crit not early}
a^N_{-1} \in \pi_h \circ \Psi^N(\cD_{(\hr-1)R_N})
\matsp{for some}
\hr < r_N.
\end{equation}
\propref{exist saddle} implies the existence of a saddle point $q_0 \in \cB^{N+1}_0$ of period $dR_{N+1}$ for some $d \leq r_{N+1}$. For $0 \leq i < R_{N+1}$, let $W^{ss}_{\loc}(q_{iR_N})$ be the connected component of $W^{ss}(q_{iR_N}) \cap \cB^N_0$ containing $q_{iR_N}$. It follows from Propositions \ref{ss in bn} and \ref{value or sink} that $W^{ss}_{\loc}(q_{iR_N})$ is $\lambda^{(1-\bepsilon)R_N}$-vertical and vertically proper in $\cB^N_0$.

By \eqref{eq.hit crit not early}, either $\pi_h\circ \Psi^N(\cB^{N+1}_0)$ is contained in a $\lambda^{\bepsilon R_N}$-neighborhood of $\pi_h\circ \Psi^N(\cB^{N+1}_{\hr R_N})$, or vice versa. In the former case, $W^{ss}(q_0)$ intersects $\cB^{N+1}_{\hr R_N}$, and in the latter case, $W^{ss}(q_{\hr R_N})$ intersects $\cB^{N+1}_0$. In either case, we have a contradiction. Hence, \eqref{eq.hit crit not early} does not hold.

Suppose towards a contradiction that
\begin{equation}\label{eq.hit crit too late}
a^N_{-1} \not\in \pi_h \circ \Psi^N(\cD_{(r_N-1)R_N})
\end{equation}
For $y \in I^v_0$, let
$$
\cJ^y_0 := (\Psi^N)^{-1}(J_0 \times \{y\})
\matsp{and}
\cJ^y_i := F^i(\cJ^y_0)
\matsp{for}
i \in \bbN.
$$
Arguing inductively using Lemmas \ref{flat} iii) and \ref{quad flat}, we see that 
$$
\cJ^y_{R_{N+1}} \cap \cV_{v_0}^N(\lambda^{\bepsilon R_N}) =\varnothing,
$$
and $\cJ^y_{R_{N+1}}$ is $\lambda^{-\bepsilon R_N}$-horizontal. Hence, $f_{N+1}$ maps $J_0$ as a diffeomorphism onto $f_{N+1}(J_0)$.

If $f_{N+1}$ is orientation-reversing, then $a^N_{2r_N}$ and $a^N_{r_N}$ are $\lambda^{(1-\bepsilon) R_N}$-close to the left and right endpoints of $f_{N+1}(J_0)$ respectively. \lemref{quick return} implies that $\lambda^{\bepsilon R_N} < a^N_{2r_N}$. Thus, we see that $f_{N+1}(J_0) \Subset J_0$. We conclude that
$$
F^{R_{N+1}}(\cD_0) \Subset \cD_0 \setminus \cV^N_{v_0}(\lambda^{\bepsilon R_N}).
$$
Applying \propref{value or sink}, it follows that $\cD_0$ induces a trivial renormalization of $F^{R_N}|_{\cB^N_0}$. Then by a similar argument as in the proof of \lemref{quick return}, we see that $\cB^{N+1}_0$ also induces a trivial renormalization of $F^{R_N}|_{\cB^N_0}$. This is a contradiction.

If $f_{N+1}$ is orientation-preserving, then $a^N_{r_N}$ and $a^N_{2r_N}$ are $\lambda^{(1-\bepsilon) R_N}$-close to the left and right endpoints of $f_{N+1}(J_0)$ respectively. Proceeding inductively on $k \in \bbN$, suppose that $\cJ_{kR_{N+1}}$ is $\lambda^{-\bepsilon R_N}$-horizontal; and every $p_0 \in \cD_0$ is $kR_{N+1}$-times forward $(CK_0, \bepsilon, \lambda)$-times regular along
$$
\hE^{v,k}_{p_0} := DF^{-kR_{N+1}}(E^{v,N}_{p_{kR_{N+1}}}).
$$
It follows that
$$
\dist_{C^r}(\cJ^y_{kR_{N+1}}, \cJ^{y'}_{kR_{N+1}}) < \lambda^{(1-\bepsilon)kR_{N+1}}
\matsp{for}
y, y' \in I^v_0.
$$
By a similar argument as above used to disprove \eqref{eq.hit crit not early}, we see that 
$$
\cJ^y_{(kr_N +i)R_N} \cap \cV_{v_0}^n(\lambda^{\bepsilon R_N}) = \varnothing
\matsp{for}
i < r_N.
$$
Hence, arguing inductively using Lemmas \ref{flat} iii) and \ref{quad flat}, it follows that $\cJ^y_{(k+1)R_{N+1}-1}$ is $\lambda^{(1-\bepsilon) R_N}$-horizontal.

Let 
$$
\hcJ^y_k := \bigcup_{i=0}^{k-1} \cJ^y_{iR_{N+1}}
\matsp{and}
\hJ^y_k := \pi_h \circ \Psi^N(\hcJ^y_k).
$$
If
\begin{equation}\label{eq.miss once}
\cJ^y_{(k+1)R_{N+1}} \cap \cV_{v_0}^n(\lambda^{\bepsilon R_N}) \neq \varnothing,
\end{equation}
then $f_{N+1}|_{\hJ^y_k}$ is a $C^r$-map on the interval $\hJ^y_k$ that maps $\hJ_{k-1}$ as an orientation-preserving diffeomorphism to $f_{N+1}(\hJ_{k-1})$, and maps the unique turning point in $\hJ_k \setminus \hJ_{k-1}$ to an image that is $\lambda^{(1-\bepsilon)R_N}$-close to $0$. This is clearly impossible. It follows by induction that \eqref{eq.miss once} does not hold for all $k \in \bbN$.

Let
$$
\cD := \bigcup_{i=0}^\infty F^{iR_{N+1}}(\cD_0)
\matsp{and}
\cA := \bigcap_{i=0}^{\infty} F^{iR_{N+1}}(\cD)
$$
Then the above observations imply that $\cA_0$ is a totally invariant connected set disjoint from $\cV^N_{v_0}(\lambda^{\bepsilon R_N})$, whose basin contains $\cD$. By a similar argument as in the proof of \lemref{quick return}, we see that $\cB^{N+1}_0$ induces a trivial renormalization of $F^{R_N}|_{\cB^N_0}$. This is a contradiction. Hence, we conclude that \eqref{eq.hit crit} holds. Moreover, arguing as in the proofs of Propositions \ref{min domain} and \ref{value struct bound}, we conclude that $\cD_0$ is $R_{N+1}$-periodic.

For $p_0 \in \cD_0$, let
\begin{equation}\label{eq.vert field}
E^{v, N+1}_{p_0} := DF^{-R_{N+1}}(E^h_{p_{R_{N+1}}}).
\end{equation}
Then \propref{pull back vert in bn} implies that $p_0$ is $R_{N+1}$-times forward $(CK_0, \bepsilon, \lambda)$-regular along $E^{v,N+1}_{p_0}$. Denote $\cB^{N+1}_0 := \cD_0$, and let $\Psi^{N+1} : \cB^{N+1}_0 \to B^{N+1}_0$ be a genuine horizontal chart that rectifies the vertical direction field given by \eqref{eq.vert field}. It follows from \eqref{eq.hit crit} that
$$
(F^{R_{N+1}}, \Psi^{N+1} : \cB^{N+1}_0 \to B^{N+1}_0)
$$
is a H\'enon-like return. 

It remains to prove that any point $p_0 \in \cB^{N+1}_{R_{N+1}}$ is backward $(CK_0, \bepsilon, \lambda)$-regular along $E^h_{p_0}$. By the regularity of the $N$th H\'enon-like return, $p_0$ is $R_N$-times backward $(L, \epsilon, \lambda)$-regular along $E^h_{p_0}$. Proceeding by induction, suppose that for some $1 \leq l < r_{N+1}$, the point $p_0$ is $lR_N$-times backward $(CK_0, \bepsilon, \lambda)$-regular along $E^h_{p_0}$. By \propref{vert angle shrink},
$$
E^{v,N+1}_{p_{-lR_N}} := DF^{-lR_N}(E^h_{p_0})
$$
is $\lambda^{(1-\bepsilon)R_N}$-vertical in $\cB^N_0$. Arguing as in the proof of \propref{pull back vert in bn}, we see that
$$
\lambda^{\bepsilon R_N} < \frac{\|DF^{-i}|_{E_{p_{-lR_N}}^{v,N+1}}\|}{\|DF^{-i}|_{E_{p_{-lR_n}}^h}\|} < \lambda^{-\bepsilon R_N}
\matsp{for}
1 \leq i \leq R_N.
$$
Concatenating with the $lR_N$-times backward regularity of $p_0$, we conclude that $p_0$ is actually $(l+1)R_N$-times backward $(\bC K_0, \bepsilon, \lambda)$-regular along $E^h_{p_0}$ (with $\bepsilon$ increased some uniform amount from the $l$th step).
\end{proof}

In \cite{CLPY3}, we use the 1D-like structures of sufficiently deep H\'enon-like returns of $F$ to prove the following result.

\begin{thm}\cite[Theorem 5.1]{CLPY3}\label{crit rec}
We have
$$
\bigcap_{n=1}^\infty \cB^n_{R_n} = \{v_0\}.
$$
Consequently, the orbit of $v_0$ is recurrent.
\end{thm}


 \section{A Priori Bounds}\label{sec.a priori}

Let $r \geq 2$ be an integer, and consider a $C^{r+4}$-H\'enon-like map $F : D \to D$. For some $N \in \bbN \cup \{\infty\}$; $L \geq 1$ and $\epsilon, \lambda \in (0,1)$, suppose that $F$ has $N$ nested $(L, \epsilon, \lambda)$-regular H\'enon-like returns given by \eqref{eq.returns} with combinatorics of $\bfb$-bounded type for some integer $\bfb \geq 3$. By only considering every other returns if necessary, we may also assume without loss of generality that $r_n := R_{n+1}/R_n \geq 3$ for ${n_0}\leq n \leq N$. Assume that $\epsilon$ is sufficiently small so that $\bfb\bepsilon < 1$. Also assume that $N$ is sufficiently large, so that for some smallest number $0 \leq n_0 \leq N$, we have \eqref{eq.proper depth 1}. Lastly, suppose that $F^{R_N}|_{\cB^N_0}$ is twice non-trivially topologically renormalizable (so that \propref{value struct bound} applies).

Recall that $P^n_0 := \pi_h \circ \Psi^n$ for $n_0 \leq n \leq N$.

\begin{lem}\label{quad mapping}
Let $\kappa_F, K_1 > 0$ be the constants given in \thmref{crit chart} and \eqref{eq.const 1} respectively. Consider a $C^{r+3}$-map $g : I \to \bbR$ on an interval $I \subset I^h_{-1}$ such that $\|g\|_{C^2} < \underline{\kappa_F}$. Denote $G(x) := (x,g(x)).$ Then there exist $a \in I^h_0$ and a $C^r$-diffeomorphism $\psi: I \to \psi(I)$ with
$$
\|\psi^{\pm 1}\|_{C^r} < K_1(1+\|g\|_{C^{r+3}})
$$
such that we have
\begin{equation}\label{eq.quad map}
Q(x) := P^n_0 \circ F \circ \Phi_{-1}^{-1}\circ G(x) = \kappa_F \cdot (\psi(x))^2+a
\end{equation}
where defined.
\end{lem}

For $n_0 \leq n \leq N$, define a sequence of maps $\{H^n_i\}_{i=0}^\infty$ as follows. First, let $H^{n_0}_i := F^i.$ Proceeding inductively, suppose $H^{n-1}_i$ is defined. Write $i = j+kR_n$ with $k \geq 0$ and $0\leq j < R_n$. Define
$$
H^n_i := H^{n-1}_j \circ \cP^n_0 \circ F^{kR_n},
$$
where
$$
\cP^n_0 := (\Psi^n)^{-1} \circ \Pi_h \circ \Psi^n
$$
is the $n$th projection map near the critical value $v_0$. Observe that $H^n_i$ is well-defined on $\cB^n_0$.

\begin{lem}\label{H diffeo near tip}
Let $s \in \{1, 2\}$ and $n_0 \leq n \leq N-s$. Then $H^n_i|_{\cI^{n+s}_1}$ is a diffemorphism for $0\leq i < R_{n+s}$.
\end{lem}

Recall the definition of $\cP^n_{-1}$ for $n_0 \leq n \leq N$ given in \eqref{eq.-1 proj}. 

\begin{lem}\label{lem unproj}
For $s\in \{1, 2\}$ and $n_0 \leq n \leq N-s$, let $\Gamma_0$ be a $C^r$-curve which is $\lambda^{-\bepsilon R_n}$-horizontal in $\cB^{n+s}_0$. Then for $1 \leq k \leq R_{n+s}/R_n$, we have
$$
F^{kR_n-1}|_{\Gamma_0} = \left(\cP^{n_0}_{-1}|_{\Gamma_{kR_n-1}}\right)^{-1}\circ H^n_{kR_n-1}|_{\Gamma_0}.
$$
\end{lem}

\begin{figure}[h]
\centering
\includegraphics[scale=0.2]{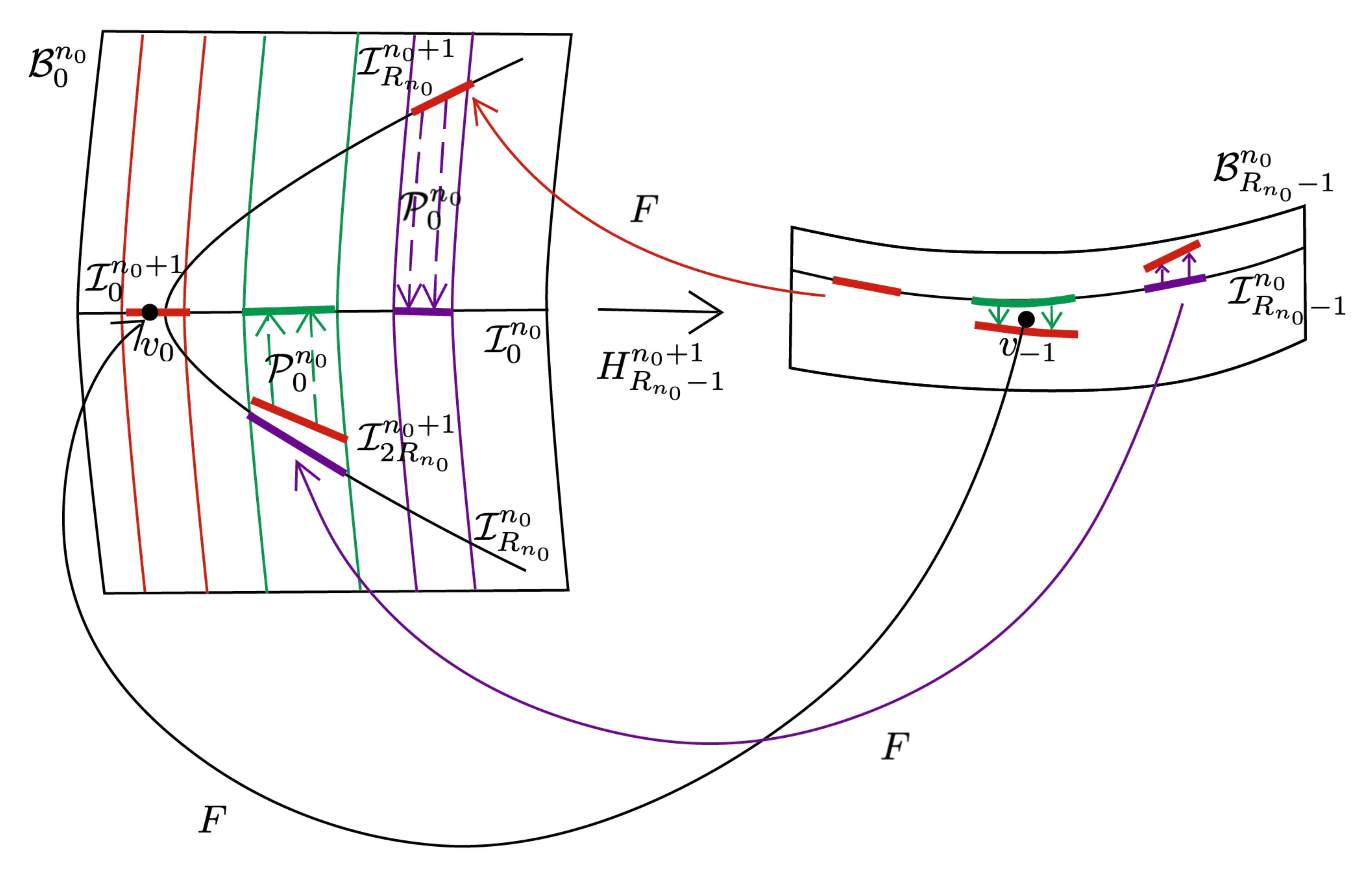}
\caption{Visualization of the map $H^{n_0}_i$ for $0 \leq i < R_{n_0+1}$ acting on the horizontal curve $\cI^{n_0+1}_0 \subset \cI^{n_0}_0$ (for $r_{n_0} := R_{n_0+1}/R_{n_0} = 3$). The orbit of $\cI^{n_0+1}_0$ makes returns to $\cB^{n_0}_0 \ni v_0$ under $F^{kR_n}$ for $0 \leq k < r_{n_0}$. At these moments, the projection map $\cP^{n_0}_0$ is applied to $\cI^{n_0+1}_{kR_{n_0}}$ to bring it down to $\cI^{n_0}_0$. These projections are then ``undone'' in $\cB^{n_0}_{R_{n_0}-1} \ni v_{-1}$ to return to $\cI^{n_0+1}_{(k+1)R_{n_0-1}}$. For $n > n_0$, the multiple projections (at various depths) can be applied to the orbit of $\cI^n_0$ near $v_0$ before they are undone near $v_{-1}$.}
\label{fig.hdiffeo}
\end{figure}

We also define another sequence of maps $\{\hH_i\}_{i=0}^{R_N-1}$ as follows (if $N=\infty$, then $R_N = \infty$). If $i < 2R_{n_0}$, let
$
\hH_i := F^i.
$
Otherwise, let $n_0 \leq n < N$ be the largest number such that $i \geq 2R_n$, and define
$
\hH_i := H^n_i.
$
Observe that by \lemref{lem unproj}, we have
\begin{equation}\label{eq.proj to original}
\hH_{R_n-1}|_{\cI^n_0} = H^{n-1}_{R_n-1}|_{\cI^n_0} = \cP^{n_0}_{-1}|_{\cI^n_{R_n-1}}\circ F^{R_n-1}|_{\cI^n_0}.
\end{equation}

\begin{thm}\label{a priori}
There exists a uniform constant
$$
\bfK = \bfK(L, \lambda, \epsilon, \lambda^{1-\epsilon}\|DF^{-1}\|, \|DF\|_{C^5}, \|F^{R_{n_0}}|_{\cB^{n_0}}\|_{C^6}, \kappa_F)\geq 1
$$
such that for all $n_0 \leq n \leq N$, we have
$$
\Dis(\hH_i, \cI^n_0) < \bfK
\matsp{for}
0\leq i < R_n.
$$
\end{thm}

\begin{cor}\label{a priori cor}
For $n_0 \leq n \leq N$, let $h_n : I^n_0 \to h_n(I^n_0)$ be the diffeomorphism given in \thmref{crit chart} ii). Then $\Dis(h_n, I^n_0) < \bfK$, where $\bfK > 1$ is the uniform constant given in \thmref{a priori}.
\end{cor}

Observe that any number $2R_{n_0} \leq i <R_N$ can be uniquely expressed as
$$
i = j + a_{n_0}R_{n_0} + a_{n_0+1} R_{n_0+1} + \ldots + a_nR_n
$$
for some $n_0 \leq n < N$, where
\begin{enumerate}[i)]
\item $0 \leq j < R_{n_0}$;
\item $0 \leq a_m < r_m$ for ${n_0} \leq m < n$; and
\item $2 \leq a_n < 2r_n$.
\end{enumerate}
In this case, we denote
$$
i := j + [a_{n_0}, a_{n_0+1}, \ldots, a_n].
$$
We extend this notation to $i < 2R_{n_0}$ by writing
$$
i = j + [a_{n_0}]
\matsp{for some}
a_{n_0} \in \{0, 1\}
$$

We record the following easy observation.

\begin{lem}\label{H decomp}
Let $2R_{n_0} \leq i <R_N$ be given by
$$
i = j + [a_{n_0}, \ldots, a_n].
$$
Then we have
$$
\hH_i = H^n_i = F^j\circ \left(\cP^{n_0}_0 \circ F^{a_{n_0}R_{n_0}}\right) \circ \ldots \circ \left(\cP^n_0 \circ F^{a_nR_n}\right).
$$
\end{lem}

For $n_0 \leq n \leq N$, we define a collection of arcs $\{\cJ^n_i\}_{i=0}^{R_n-1}$ by
\begin{equation}\label{eq.Js}
\cJ^n_i := \hH_i(\cI^n_0)
\matsp{for}
0 \leq i < R_n.
\end{equation}
See \figref{fig.Js}.

\begin{lem}\label{J spread}
Let $n_0 \leq n \leq N$ and $0\leq i < R_n$. If 
$$
i = [0, \ldots, 0, a_m, a_{m+1}, \ldots, a_k]
$$
for some ${n_0} \leq m \leq k < n$, then we have
$
\cJ^n_i \subset \cI^m_0.
$
Moreover, we have
$$
 \cJ^n_{i + l} = H^{m-1}_l(\cJ^n_{i})
 \matsp{for}
 0 \leq l < R_m.
 $$
\end{lem}

\begin{figure}[h]
\centering
\includegraphics[scale=0.2]{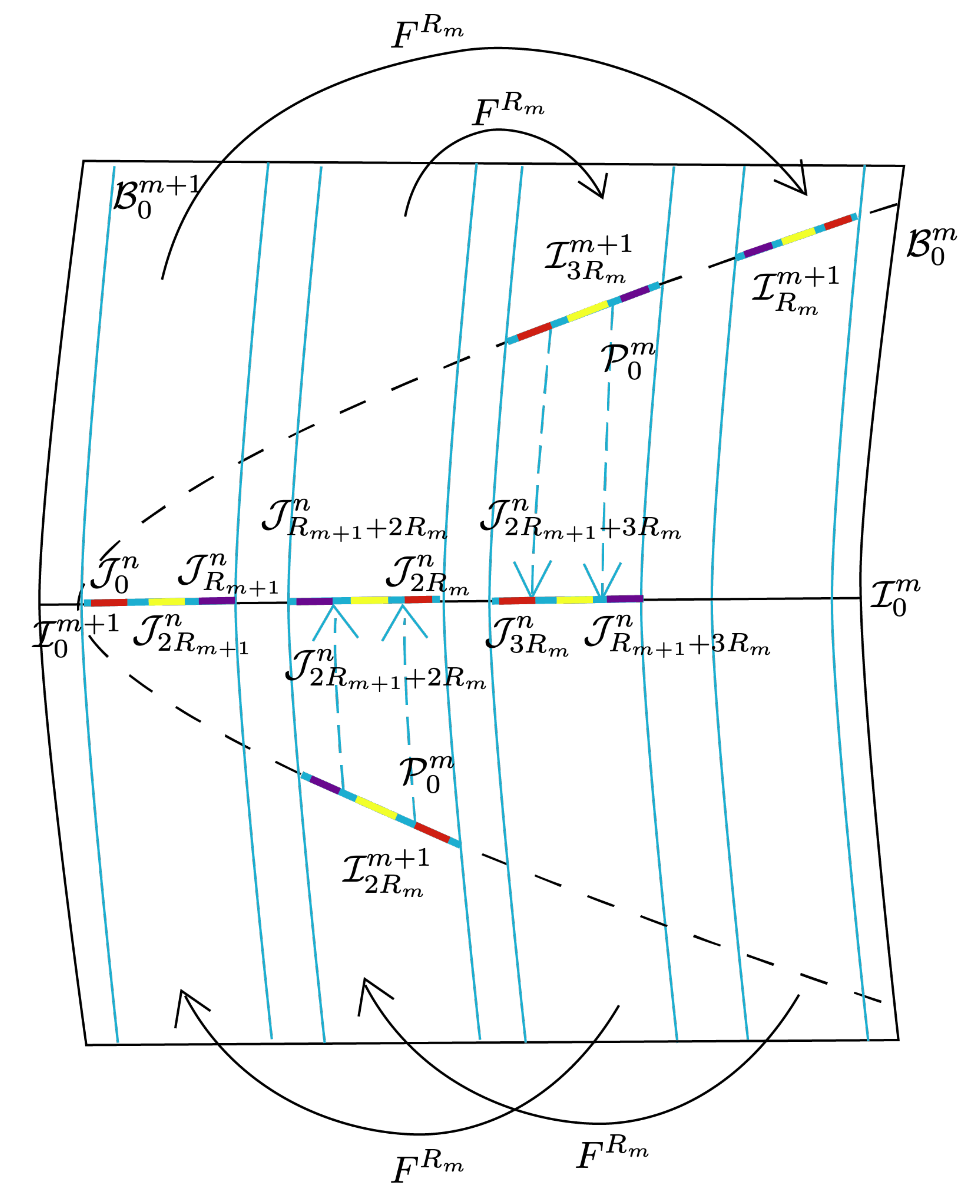}
\caption{Arcs $\cJ^n_i := \hH_i(\cI^n_0)$ with $0 \leq i <R_n$ that are contained in $\mathcal{I}^m_0$ for some $m < n$. For $0 \leq k < r_{m+1}$, we have $\cJ^n_{kR_{m+1}} \subset \cI^{m+1}_0$. For $2 \leq l < r_m$, we have $\cJ^n_{kR_{m+1}+lR_m}= \cP^m_0\circ F^{R_m}(\cJ^n_{kR_{m+1}})$.}
\label{fig.Js}
\end{figure}

\begin{lem}\label{on center track}
For $n_0 \leq n \leq N$ and $0 \leq i < R_n$, we have
$
\cJ^n_{i} \subset \cI^{n_0}_{i \md{R_{n_0}}}.
$
\end{lem}

Let $\Gamma : [0, 1] \to \bbR^2$ be a parameterized Jordan arc. For
$$
0 \leq a < b < c < d \leq 1,
$$
consider the subarcs $\Gamma_1 := \Gamma(a,b)$ and $\Gamma_2 := \Gamma(c,d)$ of $\Gamma$. We denote $\Gamma_1 <_\Gamma \Gamma_2.$ Let $\Gamma_3$ be another subarc of $\Gamma$. We denote
$
\Gamma_1 \leq_\Gamma \Gamma_3
$
if either
$
\Gamma_1 <_\Gamma \Gamma_3
$ or $
\Gamma_1 = \Gamma_3.$

Henceforth, we consider $\cI^{n_0}_0$ with parameterization given by
$$
\cI^{n_0}_0(t) := (\Psi^{n_0})^{-1}(t,0)
\matsp{for}
t \in I^{n_0}_0.
$$
Note that
$\cI^{n_0}_0 \circ P^{n_0}_0 = \cP^{n_0}_0.$
Moreover,
$$
P^{n_0}_0(v_0) = 0 < P^{n_0}_0(v_{R_{n_0}}).
$$

\begin{lem}\label{J order}
For $s \in \{1, 2\}$; $n_0 \leq n \leq N-s$ and $1 < k < R_{n+s}/R_n$, we have
$$
\cJ^{n+s}_0  <_{\cI^{n_0}_0} \cJ^{n+s}_{kR_n} <_{\cI^{n_0}_0}\cJ^{n+s}_{R_n}.
$$
\end{lem}

Let $i \geq 2R_{n_0}$ be a number given by
$$
i = [0, \ldots, 0, a_m, a_{m+1}, \ldots, a_k]
$$
for some ${n_0} \leq m \leq k$ so that $a_m > 0$. Denote
$$
\hm(i) := m,
\hspace{5mm}
\hk(i) := k
\matsp{and}
\ha(i) := a_m.
$$
We extend this notation to the case when
$
i = a_{n_0}R_{n_0}
$ with $
a_{n_0} \in \{0, 1\}
$
by letting
$$
\hm(i) := 1,
\hspace{5mm}
\hk(i) := 1
\matsp{and}
\ha(i) := a_{n_0}.
$$

\begin{prop}\label{H decomp fine}
Let $n_0 \leq n \leq N$ and $i = j+sR_{n_0}$ with $0\leq j < R_{n_0}$ and $0 \leq s <R_n/R_{n_0}$. For $0\leq l \leq s$, denote
$$
\hm_l := \hm(lR_{n_0}),
\hspace{5mm}
\hk_l := \hk(lR_{n_0})
\matsp{and}
\ha_l := \ha(lR_{n_0}).
$$
If $\hm_l = \hk_l$, let
$$
\chcI^n_l := F^{lR_{n_0}-1}(\ticI^n_{0, i}).
$$
Otherwise, let
$$
\chcI^n_l := \cI^{\hm_l+1}_{\ha_lR_{\hm_l}-1}.
$$
Then $\chcI^n_l$ is $\lambda^{(1-\bepsilon)R_{\hm_l}}$-horizontal. Moreover, define
$$
\chH_l :=  \cP^{\hm_l}_0 \circ F \circ \left(\cP^{n_0}_{-1}|_{\chcI^n_l}\right)^{-1} \circ F^{R_{n_0}-1}|_{\cI^{n_0}_0}.
$$
Then we have
$$
\hH_i|_{\ticI^n_{0, i}} = F^j|_{\cI^{n_0}_0} \circ \chH_s \circ \ldots\circ \chH_4 \circ \chH_3 \circ \cP^{n_0}_0 \circ F^{2R_{n_0}}|_{\ticI^n_{0, i}}.
$$
\end{prop}


\section{Uniform $C^1$-Bounds}\label{sec.c1 bound}

\subsection{For unimodal maps}

Let $s \geq 1$ be an integer, and consider a normalized $C^{s+3}$-unimodal map $f : I \to I \in \frU^{s+3}$. Recall that this means $f'(0)=0$ and $f''(0) = 2$. Let $\psi_f$ be the $C^s$-diffeomorphism given in \lemref{factor} so that $f(x) = 2(\psi_f(x))^2$. An elementary computation shows that
$$
2 = f''(0) = 2(\psi_f'(0))^2.
$$
Hence, $\psi_f'(0) = 1$.

For $K \geq 1$, we say that $f$ has {\it $K$-bounded non-linearity} if
\begin{equation}\label{eq.1d dist bound}
\sup_{x,y \in I}\frac{\psi_f'(x)}{\psi_f'(y)} \leq K.
\end{equation}
Denote the space of all normalized $C^{s+3}$-unimodal maps with $K$-bounded non-linearity by $\frU^{s+3}(K)$. Observe that if $f : I \to I$ is in $\frU^{s+3}(K)$, then
\begin{equation}\label{eq.psi deriv}
K^{-1}\leq |\psi_f'(x)| \leq K
\matsp{for}
x\in I.
\end{equation}

\begin{lem}\label{1d upper bound}
Let $f :I \to I \in \frU^4(K)$ for some $K \geq 1$. Then $|I| < 2|K|^2$. Consequently, we have $\|f'\| < C$ for some uniform constant $C = C(K) \geq 1$.
\end{lem}

\begin{proof}
By \eqref{eq.psi deriv}, we see that $|\psi_f(I)| > K^{-1}|I|$. If the length of an interval is bigger than $2|K|^2$, then $|f(I)| > (2|K|)^2 > 4|K|^2$. Thus, iterated images of $I$ under $f$ become unbounded. This is a contradiction.

We compute
$$
f'(x) = 2\psi_f(x)\psi_f'(x).
$$
The result follows.
\end{proof}

\begin{lem}\label{1d lower bound}
Let $f : I \to I$ be in $\frU^4(K)$ for some $K \geq 1$, and let $n \in \bbN$. If the critical orbit of $f$ does not converge to an $n$-periodic sink, then there exists a uniform constant $\rho_n = \rho_n(K) > 0$ such that $|f^n(0)| > \rho_n$. In particular, $|I| > \rho_1$.
\end{lem}

\begin{proof}
By \lemref{1d upper bound}, there exists a uniform constant $C = C(K) \geq 1$ such that $\|f'\|< C$. For $n \in \bbN$, let
$$
l_n := \frac{1}{4K^2C^n}
\matsp{and}
J_n := (-l_n, l_n).
$$
Observe that
$$
|f^n(J_n)| < C^n|\psi_f(J_n)|^2 < \frac{1}{16K^2C^n}.
$$
Hence, if $f^n(0) \in (-l_n/8, l_n/8)$, then $f^n(J_n) \Subset J_n$. The result now follows from
$$
|(f^n)'(x)| < 2|\psi_f(x)||\psi_f'(x)|C^n < 1/2
\matsp{for}
x \in J_n.
$$
\end{proof}

\begin{prop}\label{c1 bound geo}
Let $f : I \to I$ be in $\frU^4(K)$ for some $K \geq 1$. Suppose that $f$ is non-trivially renormalizable, so that there exists an $R$-periodic interval $I^1$ such that $f^R(I^1)$ contains the critical value $v$ for $f$. Denote by $c^1$ the critical point for $f^R|_{I^1}$. Assume that $\cRod(f)$ has $\chi$-bounded kneading for some $\chi \geq R$. Let $J$ be a connected component of $I \setminus \{f^i(c^1)\}_{i=0}^{3R+1}.$ Then we have $|J| > \rho$, where $\rho = \rho(K, \chi) \in (0,1)$ is a uniform constant.
\end{prop}

\begin{proof}
By \lemref{1d struct}, we have $I^1 := [v, f^R(v)] \ni c^1$. Denote $I^1_i := f^i(I^1)$ for $0 \leq i < R$. The fact that we have $|f^{l+i}(v)-f^i(c^1)| > \rho$ for for $l \in \{0, R\}$ and $1\leq i < R$ follows from Lemmas \ref{1d upper bound} and \ref{1d lower bound}.

By the assumption on bounded kneading, there exists a smallest integer $r_1 \leq \chi$ such that $f^{(r_1+1)R}(c^1) < c^1$. By \lemref{1d lower bound}, there exists a uniform constant $\rho_1 = \rho_1(K, \chi) > 0$ such that
\begin{equation}\label{eq.knead}
v< f^{r_1R}(v) < f^{R}(v)- \rho_1.
\end{equation}

Let $L_i := [f^{iR}(v), f^{(i-1)R}(v)]$ for $2 \leq i \leq r_1$. If $r_1=2$, then it follows from \eqref{eq.knead} that $|L_2| > \rho$. If $r_1 > 2$, then observe that $f^R$ maps $L_i$ diffeomorphically onto $L_{i+1}$ for $i < r_1$. By \lemref{1d upper bound}, $\|(f^R)'\| < C$ for some uniform constant $C \geq 1$, and
$$
L_2 \sqcup L_3 \sqcup \ldots \sqcup L_{r_1} \supset [c^1, f^{R}(v)].
$$
This implies that $|L_2| > \rho'$ for some uniform constant $\rho'$.

Let $J_0$ be the gap between $I^1_k$ and $I^1_l$ with $0 \leq k < l < R$. If $J_m := f^m(J_0)$ with $m = \bchi$ maps onto an interval $I^1_i$ for some $0 \leq i < R$, then by \lemref{1d upper bound} and \lemref{1d lower bound}, we have $|J_0| > C^{-m}\tirho$ for some uniform constants $C = C(K)\geq 1$ and $\tirho=\tirho(K, \chi) > 0$. Thus, we may assume, after replacing $J_0$ with $J_R$ if necessary, that $\partial J_0 \ni f^{k+R}(v)$.

Map $J_0$ by $f^{r_1R -k-R}$. Since
$$
I^1_{l+r_1R - k-R} \cap I^1_0= \varnothing,
$$
the image $J_{r_1R -k-R}$ of the gap must contain $(c^1, f^{R}(v))$. The result now follows from \eqref{eq.knead}.
\end{proof}

\subsection{For H\'enon-like maps}\label{subsec.unif c1}

For an integer $r \geq 2$ and a constant $K \geq 1$, let $\frHL^{r+2}(K)$ be the space of all normalized $C^{r+2}$-H\'enon-like maps whose 1D profiles are contained in $\frU^{r+2}(K)$. Additionally, for $\beta \in (0,1)$, let $\frHL^{r+2}_\beta(K)$ be the set of all H\'enon-like maps in $\frHL^{r+2}(K)$ that are $\beta$-thin in $C^{r+2}$.

For some $N \in \bbN\cup\{\infty\}$, let $F$ be the $N$-times regularly H\'enon-like renormalizable $C^{r+4}$-map with combinatorics of $\bfb$-bounded type considered in \secref{sec.a priori}. If $N < \infty$, suppose that $F^{R_N}|_{\cB^N_0}$ is twice topologically renormalizable with combinatorics of $\bfb$-bounded type.

Let $\bfK \geq 1$ be the uniform constant given in \thmref{a priori}. Assume that $n_0 \leq N$ is the smallest number such that
\begin{equation}\label{eq.proper depth 2}
\overline{K_2} \lambda^{\epsilon R_{n_0}} < 1,
\end{equation}
where
\begin{equation}\label{eq.const 2}
K_2 = K_2(\bfK, \bfb) \geq 1
\end{equation}
is a uniform constant.

For $n_0 \leq n\leq N$, denote
$$
I^n_0 := \pi_h(B^n_0)
\matsp{and}
\cI^n_0 := (\Psi^n)^{-1}(I^n_0 \times \{0\}).
$$
Define
$$
\tiF_n := \Psi^n\circ F^{R_n}\circ(\Psi^n)^{-1}
\matsp{and}
\tif_n := \Piod(\tiF_n).
$$

\begin{prop}\label{c1 bound 2d}
There exists a uniform constant $K = K(\bfK) \geq 1$ such that for all $n_0 \leq n \leq N$, we have $\|\tiF_n\|_{C^1} < K.$
\end{prop}

\begin{proof}
The result follows immediately from \corref{a priori cor} and \lemref{1d upper bound}.
\end{proof}

\begin{prop}\label{2d close 1d}
For $n_0 \leq n \leq N$, we have
$$
\|\tif_n^i - \Piod(\tiF_n^i)\|_{C^0} < \lambda^{(1-\bepsilon) R_n}
\matsp{for}
i = O(\bfb).
$$
\end{prop}

\begin{proof}
Denote $\Pi_h(x,y) := (x,0)$. It suffices to show that
$$
\|(\tiF_n \circ \Pi_h)^i - \tiF_n^i \circ \Pi_h\|_{C^0} < \lambda^{(1-\bepsilon) R_n}.
$$

By \propref{c1 bound 2d}, $\|\tiF_n\|_{C^1}$ is uniformly bounded. Moreover, by \thmref{crit chart} ii), we have
$$
\|\tiF_n - \tiF_n \circ \Pi_h \|_{C^{r+3}} < \lambda^{(1-\bepsilon)R_n}.
$$
The result now follows from \lemref{vary compose}.
\end{proof}

\begin{prop}\label{exp scale}
For $n_0 \leq n < N$ and $-r_n \leq k \leq 2r_n$, denote
$$
u^n_k := \Psi^n(v_{kR_n})
\matsp{and}
a^n_k := \pi_h(u^n_k).
$$
Let $J$ be a connected component of $I^n_0 \setminus \{a^n_k\}_{k=-r_n}^{2r_n}.$ Then we have $|J| > \rho|I^n_0|$, where $\rho = \rho(\bfK, \bfb) \in (0,1)$ is a uniform constant.
\end{prop}

\begin{proof}
Let $v_0$ be the critical value of $F$ defined in \secref{sec.conv chart}. Denote $u^n_0 :=\Psi^n(v_0)$. Note that $u^n_0$ is a point of tangency between foliation by vertical quadratic curves $\lambda^{(1-\bepsilon)R_n}$-close to the image of the horizontal foliation by $F_n$, and foliation by $\lambda^{(1-\bepsilon)R_n}$-vertical curves. On the other hand, $(f_n(0), 0)$ is the unique tangency between the image curve under the degenerate H\'enon map $\iota(f_n)$, and the genuine vertical foliation. Since 
$$
\|F_n-\iota(f_n)\|_{C^{r+3}} <\lambda^{(1-\bepsilon)R_n},
$$
we see that
$$
|\pi_h(u^n_0)-f_n(0)| < \lambda^{(1-\bepsilon)R_n}.
$$
The result now follows from Propositions \ref{c1 bound geo} and \ref{2d close 1d}.
\end{proof}

\begin{thm}\label{rescale}
For $n_0 \leq n \leq N$, there exist $\sigma_n \asymp |I^n_0|^{1/2}$; $\tau_n \in I^n_0$ and a $C^{r+3}$-diffeomorphism $\phi_n$ defined on the interval $\pi_v(B^n_0)$ such that for
$$
\cS^n(x,y) := (\sigma_n^{-2}x +\tau_n, \sigma_n^{-1}y+\tau_n),
$$
and
$$
\cY^n(x,y) = (x,\phi_n(y))
\matsp{and}
\Phi^n := (\cY^n)^{-1} \circ \Psi^n,
$$
the following statements hold.
\begin{enumerate}[i)]
\item There exist uniform constants $0< \rho_1 < \rho_2 < 1$ depending only on $\bfK$ and $\bfb$ such that $\rho_1^n < \sigma_n < \rho_2^n.$
\item The distortion of $\phi_n$ is bounded by $\bfK$, and $\|\phi_n^{\pm 1}\|_{C^1} < K$, where $K = K(\bfK) \geq 1$ is a uniform constant.
\item We have
$$
\cR^n(F) := \cS^n\circ\Phi^n \circ F^{R_n}\circ (\cS^n \circ \Phi^n)^{-1} \in \frHL^{r+3}_{\lambda_n}(\bfK),
$$
where $\lambda_n := \lambda^{(1-\bepsilon)R_n}$.
\end{enumerate}
\end{thm}

\begin{proof}
Consider the maps
$$
F_0 := \Phi_0 \circ F \circ \Phi_{-1}^{-1}
\matsp{and}
H_n := \Phi_{-1}\circ F^{R_n-1}\circ (\Psi^n)^{-1}.
$$
By \thmref{crit chart}, we have
$$
F_0(x,y) = (f_0(x)-\lambda y, x)
\matsp{and}
H_n(x,y) = (h_n(x), e_n(x,y)),
$$
where $f_0$ is a map with a unique critical point at $0$ with $f_0''(0) > 0$; $h_n$ is a diffeomorphism; and $e_n$ is a map such that $\|e_n\|_{C^{r+3}} < \lambda^{(1-\bepsilon) R_n}$. Moreover, \corref{a priori cor} states that $h_n$ has $\bfK$-bounded distortion.

The map $\tiF_n := (\Psi^n) \circ F^{R_n-1}\circ (\Psi^n)^{-1}$ is of the form
$$
\tiF_n(x,y) = (g_n(x,y), h_n(x)),
$$
where $g_n(\cdot, y)$ for $y \in \pi_v(\cB^n_0)$ is a unimodal map. By \thmref{crit chart} i), we see that
$$
\|g_n(\cdot, y) - f_0 \circ h_n(\cdot)\|_{C^{r+3}} < \lambda^{(1-\bepsilon)R_n}.
$$

We claim that $|h_n(I^n_0)|^2 \asymp |I^n_0|$. Write $I^n_0 = [a_n, b_n]$ and $h_n(I^n_0) = [\alpha_n, \beta_n]$. For $-r_n\leq k \leq 2r_n$, denote
$$
u^n_k := \Psi^n(v_{kR_n})
\matsp{and}
a^n_k := \pi_h(u^n_k).
$$
Recall that $a_n$ and $b_n$ are $\lambda^{\bepsilon R_n}$-close to $a^n_0 = 0$ and $a^n_{r_n}$ respectively, and that $a^n_{-1} \in I^n_0$. Additionally, observe that $g_n(0)$, $g_n(a^n_{-1})$ and $g_n(a^n_{r_n})$ are $\lambda^{(1-\bepsilon) R_n}$-close to $a^n_{r_n}$, $0$ and $a^n_{2r_n}$ respectively. By \propref{exp scale}, $|a^n_{r_n}|$ and $|a^n_{2r_n}|$ are commensurate with $|I^n_0|$. The claim now follows from the fact that $f_0(\alpha_n)$ and $f_0(\beta_n)$ are $\lambda^{(1-\bepsilon)R_n}$-close to $g_n(a_n)$ and $g_n(b_n)$ respectively.

Define
$$
\check\cY^n(x,y) := (x, h_n(y))
\matsp{and}
\chPhi^n := (\check\cY^n)^{-1}\circ \Psi^n.
$$
It is easy to check that
$$
\chF_n :=\chPhi^n \circ F^{R_n} \circ (\chPhi^n)^{-1}
$$
is a H\'enon-like map. Denote $\chf_n := \Piod(\chF_n)$, and let
$$
\check S^n(x) = \sigma_n^2 x+\tau_n
$$
be the unique orientation-preserving affine map on $\bbR$ such that $\check S^n \circ \chf_n \circ (\check S^n)^{-1} \in \frU^{r+3}(\bfK)$.

\lemref{1d lower bound} implies that $\sigma_n^2 \asymp |I^n_0|$. Property i) now follows from \propref{exp scale}. Let
$$
\phi_n(y) := \sigma_n^{-1}h_n(y).
$$
Properties ii) and iv) follow immediately. Lastly, since $h_n$ has bounded distortion, and $|h_n(I^n_0)| \asymp |I^n_0|^{1/2} \asymp \sigma_n$, it follows that $\|\phi_n^{\pm 1}\|_{C^1}$ is uniformly bounded. Thus, we also have Property iii).
\end{proof}


\section{Preservation of Regularity}\label{sec.preserve}

For $N \in \bbN$, let $F$ be the $N$-times $(L, \epsilon, \lambda)$-regularly H\'enon-like renormalizable map with combinatorics of $\bfb$-bounded type considered in \subsecref{subsec.unif c1} (with $n_0 \leq N$ satisfying \eqref{eq.proper depth 2}). For $n_0 \leq n \leq N$, we have by \thmref{rescale}:
$$
F_n := \cR^n(F) = \cS^n \circ \Phi^n \circ F^{R_n} \circ (\cS^n \circ \Phi^n)^{-1},
$$
where $\|(\Phi^n)^{\pm 1}\|_{C^1} < K$ for some uniform constant $K= K(\bfK) \geq 1$, and
$$
\cS^n(x,y) := (\sigma_n^{-2}(x +\tau_n), \sigma_n^{-1}(y+\tau_n))
$$
for some $\sigma_n \in (0,1)$ and $|\tau_n| \asymp \sigma_n^2$. Moreover, by \propref{exp scale}, there exist uniform constants $0 < \rho_1 < \rho_2 < 1$ depending only on $\bfK$ and $\bfb$ such that $\rho_1^n < \sigma_n < \rho_2^n$. Let $f_n := \Piod(F_n)$.

For $p \in \cB^n_0$, let
$$
E_p^{v,n} := (D\Phi^n)^{-1}(E_z^{gv})
\matsp{where}
z := \Phi^n(p).
$$
Moreover, since $\Phi^n$ is genuinely horizontal, we have
$$
E_p^{gh} = (D\Phi^n)^{-1}(E_z^{gh}).
$$
Denote $P^n_0 := \pi_h \circ \Phi^n$.

\begin{lem}\label{hor deriv bound}
For $n_0 \leq n < N$ and $0 \leq k < r_n$, let $p \in \cB^{n+1}_{kR_n} \subset \cB^n_0$. Then we have
$$
\frac{1}{K} < \|DF^{iR_n}|_{E_p^{gh}}\| < \|DF^{iR_n}\|< K
\matsp{for}
0\leq i \leq r_n - k,
$$
where $K = K(\bfK, \bfb) \geq 1$ is a uniform constant.
\end{lem}

\begin{proof}
The upper bound follows immediately from \propref{c1 bound 2d}. Denote $z := \cS^n \circ \Phi^n(p)$, and $z_i = (x_i, y_i) :=F_n^i(z).$ Propositions \ref{value struct bound} and \ref{exp scale} imply that for $0 \leq j < r_n -k-1$, there exists a uniform constant $\rho = \rho(\bfK, \bfb) \in (0,1)$ such that $|x_j| > \rho$. Thus, $|f_n'(x_j)|$ is uniformly bounded below.

By the thinness of $F_n$, we see that for all $w = (u,v)$ in the domain of $F_n$, we have
$$
\|D(\pi_h\circ F_n)|_{E_w^{gh}}\| = |f_n'(u)| + O(\lambda^{(1-\bepsilon)R_n}).
$$
Let $h_n := \psi_{f_n}$ be the diffeomorphism given in \lemref{factor}. Then
$$
\|DF_n|_{E_w^{gh}}\| > c|h_n'(u)|
$$
for some uniform constant $c > 0$. Let
$$
\alpha_j := |f_n'(x_j)| - \lambda^{(1-\bepsilon)R_n}
\matsp{for}
0 \leq j < r_n-k-1,
$$
and $\alpha_{r_n-k-1} := c|h_n'(x_{r_n-k-1})|$. Then we have
$$
\|D(\pi_h\circ F_n^i)|_{E_{z_0}^{gh}}\| \geq \alpha_0\ldots \alpha_{i-1}.
$$
The desired lower bound follows.
\end{proof}

\begin{lem}\label{hor deriv out}
For $n_0 \leq n \leq N$, let $p_0 \in \cB^n_0$. Then
$$
\frac{1}{K^{n-n_0}} < \|DF^T|_{E_{p_0}^{gh}}\| < K^{n-n_0}
\matsp{for}
0\leq T < R_n,
$$
where $K = K(\bfK, \bfb) \geq 1$ is a uniform constant.
\end{lem}

\begin{proof}
Write
$$
T = t_0 + t_{n_0} R_{n_0} + \ldots + t_{n-1} R_{n-1}
$$
with $0 \leq t_0 < R_{n_0}$ and $0 \leq t_k < r_k$ for $n_0 \leq k < n$. Denote
$$
T_k := t_{k+1}R_{k+1} + \ldots + t_{n-1}R_{n-1}.
$$
Then clearly, we have
$$
K^{-1} \cdot d_{n_0}\cdot \ldots \cdot d_{n-1} \leq \|DF^T|_{E_{p_0}^{gh}}\| \leq K \cdot D_{n_0}\cdot \ldots \cdot D_{n-1},
$$
where
$$
D_k := \|D_{p_{T_k}}F^{t_kR_k}\|
\matsp{and}
d_k := \left\|D\left(P^k_0 \circ F^{t_k R_k}\right)|_{E^{gh}_{p_{T_k}}}\right\|
$$
for $n_0\leq k < n$. The result now follows from \lemref{hor deriv bound}.
\end{proof}

Let $v_0$ be the critical value of $F$. For $k \geq -r_n$, denote
$$
u^n_k := \cS^n\circ\Phi^n(v_{kR_n})
\matsp{and}
a^n_k := \pi_h(u^n_k).
$$

Consider an increasing sequence of renormalization depths
$$
n_0 \leq n_1 < n_2 < \ldots < n_k \leq N.
$$
We say that this sequence is {\it tempered} if for $1 \leq i < k$, we have
$$
\rho_1^{n_{i+1}-n_i} > \lambda^{\bepsilon R_{n_i}},
$$
where $\rho_1$ is given in \propref{exp scale}.

\begin{lem}\label{temper in}
Consider a tempered sequence $\{n_i\}_{i=1}^k$. Let
$$
S = s_1R_{n_1} + \ldots + s_{k-1}R_{n_{k-1}}
\matsp{and}
\hS := S + s_kR_{n_k};
$$
where $1 \leq s_i < r_{n_i}$ for $1 \leq i \leq k$. For $p_0 \in \cB^{n_k+1}_{R_{n_k+1}}$, define
$$
\hz =(\hx, \hy) := \cS^{n_k}\circ \Phi^{n_k}(p_{S-\hS})
\matsp{and}
\hE_{\hz} := D(\cS^{n_k}\circ \Phi^{n_k} \circ F^S)(E_{p_{-\hS}}^{gh}).
$$
Then $\hE_{\hz}$ is $\theta$-horizontal for some uniform constant $\theta = \theta(\bfK, \bfb) \geq 1$. Moreover, we have
\begin{equation}\label{eq.temper in}
\frac{1}{K^k} < \left\|DF^{\hS}|_{E_{p_{-\hS}}^{gh}}\right\| < K^k
\end{equation}
where $K = K(\bfK, \bfb) \geq 1$ is a uniform constant.
\end{lem}

\begin{proof}
Proceeding by induction on $k$, suppose the result holds for $k' < k$. We first show that $\hE_{\hz}$ is uniformly horizontal. Denote
$$
S' := s_1R_{n_1} + \ldots + s_{k-2}R_{n_{k-2}};
$$
and for $0 \leq i \leq s_{k-1}$,
$$
z_i = (x_i, y_i) := \cS^{n_{k-1}}\circ \Psi^{n_{k-1}}(p_{(i-s_{k-1})R_{n_{k-1}}})
$$
and
$$
\hE_{z_i} := D(\cS^{n_{k-1}}\circ \Phi^{n_{k-1}} \circ F^{S'+iR_{n_{k-1}}})(E_{p_{-S}}^{gh}).
$$
Then by Propositions \ref{value struct bound} and \ref{exp scale}, it follows that for $0 \leq j < s_{k-1}$
$$
|x_j-a^{n_{k-1}}_0| > \rho_1.
$$
Thus, $\hE_{z_j}$ is $(1/\rho_1)$-horizontal by \lemref{quad flat}. Propositions \ref{value struct bound} and \ref{exp scale} also imply that
$$
|\hx - a^{n_k}_0| > \rho_1.
$$
Since
$$
z_{s_{k-1}} = \cS^{n_{k-1}}\circ \Phi^{n_{k-1}} \circ (\cS^{n_k}\circ \Phi^{n_k})^{-1}(\hz),
$$
it follows from \thmref{rescale} that
$$
\Delta := |x_{s_{k-1}}-a^{n_{k-1}}_0| > {K_0}^{-1}\rho_1^{n_k-n_{k-1}} > \lambda^{\bepsilon R_{n_{k-1}}}.
$$
Thus, by \lemref{quad flat}, $\hE_{z_{s_{k-1}}}$ is $O(1/\Delta)$-horizontal. Under
$$
D(\cS^{n_k}\circ \Phi^{n_k}\circ(\cS^{n_{k-1}}\circ \Phi^{n_{k-1}})^{-1}),
$$
the distance $\Delta = |x_{s_{k-1}}-a^{n_{k-1}}_0|$ is rescaled to $\rho_1 = |\hx - a^{n_k}_0|$. We conclude that
$$
\hE_{\hz} = D(\cS^{n_k}\circ \Phi^{n_k}\circ(\cS^{n_{k-1}}\circ \Phi^{n_{k-1}})^{-1})(\hE_{z_{s_{k-1}}})
$$
is $\theta$-horizontal for some uniform constant $\theta \geq 1$.

Since $\hE_{\hz}$ is uniformly horizontal, we see that $\|D\pi_h|_{\hE_{\hz}}\| > K^{-1}.$
Thus, by \lemref{hor deriv bound}, we see that
$$
\left\|DF_{n_k}^{s_k}|_{\hE_{\hz}}\right\| > K^{-1}.
$$
By the induction hypothesis, we have
$$
\left\|DF^S|_{E^{gh}_{p_{-\hS}}}\right\| > K^{-(k-1)}.
$$
Concatenating the above two inequalities, \eqref{eq.temper in} follows.
\end{proof}

\begin{thm}\label{pres reg}
Fix $\delta \in (\bepsilon, 1)$ such that $\bfb \bdelta < 1$. Then there exists a uniform constant $\bfL = \bfL(\bfK, \bfb) \geq 1$ such that the following holds. For $m \in \bbN \cup \{\infty\}$, suppose that $F_{n_0}$ is $(m+2)$-times topologically renormalizable with combinatorics of $\bfb$-bounded type. Then $F$ has $n_0+m$ nested $(\bfL, \delta, \lambda)$-regular H\'enon-like returns.
\end{thm}

\begin{proof}
Proceeding by induction, suppose that for $n_0 \leq M < n_0+m$, the map $F$ has $M$ nested $(\bfL, \delta, \lambda)$-regular H\'enon-like returns
$$
\{(F^{R_n}, \Phi^n : \cB^n_0 \to B^n_0)\}_{n=1}^M.
$$
By \thmref{ren is henon like}, $F$ has a $(\overline{\bfL}, \bdelta, \lambda)$-regular H\'enon-like return
$$
(F^{R_{M+1}}, \Phi^{M+1} : \cB^{M+1}_0 \to B^{M+1}_0).
$$
We claim that this return is $(\bfL, \delta, \lambda)$-regular.

Let $p_0 \in \cB^{M+1}_0$ and
$$
E_{p_0}^{v/h} := (D\Phi^{M+1})^{-1}(E_{\Phi^{M+1}(p_0)}^{gv/gh}).
$$
Let $R_{n_0} \leq T < R_{M+1}$. Write
$$
T = t_0 + t_1R_{n_1} + \ldots + t_kR_{n_k},
$$
with $0\leq t_0 < R_{n_0}$; $n_{i-1} \leq n_i \leq M$ and $1 \leq t_i < r_{n_i}$ for $1 \leq i\leq k$. \lemref{hor deriv out} implies that
$$
\frac{1}{K^k} < \|DF^T|_{E_{p_0}^h}\| < K^k
$$
By \eqref{eq.proper depth 2}, we have $K < \lambda^{-\epsilon R_{n_i}}.$ Together with \propref{jac bound}, this implies that $p_0$ is $R_{M+1}$-times forward $(K, \epsilon, \lambda)$-regular horizontally along $E_{p_0}^h$. By \propref{transverse for reg}, it follows that $p_0$ is $R_{M+1}$-times forward $(\bfL, \delta, \lambda)$-regular (vertically) along $E_{p_0}^v$.

Let $q_0 \in \cB^{M+1}_{R_{M+1}}$ and $S = lR_{n_0}$ for some $1 \leq l < R_{M+1}/R_{n_0}$. Write
$$
S = s_1R_{n_1} + \ldots + s_kR_{n_k},
$$
where $1 \leq s_i < r_{n_i}$ for $1 \leq i \leq k$. Let $1 \leq m \leq k$ be the smallest number such that $\{n_i\}_{i=m}^k$ is tempered. Denote
$$
S' := s_mR_{n_m} + s_{m+1}R_{n_{m+1}} + \ldots + s_kR_{n_k},
$$
and let
$$
\hE_{q_{-1}}:= DF^{S'-1}(E_{q_{-S'}}^{gh}).
$$
By Lemmas \ref{flat} i) and \ref{temper in}, we see that $\hE_{q_{-1}}$ is $\lambda^{(1-\bepsilon)R_{n_k}}$-horizontal in $\cU_{-1}$, and
\begin{equation}\label{eq.temper in pf}
K^{-(k-m)} < \|DF^{-S'+1}|_{\hE_{q_{-1}}}\| < K^{k-m}.
\end{equation}
Let
$$
E_{q_{-1}}^v := DF^{-1}(E_{q_0}^h) = D\Phi_{-1}^{-1}(E_{\Phi_{-1}(q_{-1})}^{gv}).
$$
Since $\|\Phi_{-1}^{\pm 1}\|_{C^1} < {K_0}$ by \thmref{crit chart}, we have
\begin{equation}\label{eq.temper in jac}
{K_0}^{-1}\leq \frac{\|DF^{-S'+1}|_{E_{q_{-1}}^v}\|\cdot\|DF^{-S'+1}|_{\hE_{q_{-1}}}\|}{\Jac_{p_{-1}}F^{-S'+1}} \leq {K_0}.
\end{equation}
Substituting in \propref{jac bound} and \eqref{eq.temper in pf}, we obtain
$$
\bL^{-1}K^{-(k-m)} \lambda^{-(1-\bepsilon)S'} < \|DF^{-S'+1}|_{E_{q_{-1}}^v}\| < \bL K^{k-m} \lambda^{-(1+\bepsilon)S'}.
$$
By \eqref{eq.proper depth 2}, we have $\bL K < \lambda^{-\epsilon R_{n_i}}$, and hence,
\begin{equation}\label{eq.reg pres 1}
K \lambda^{-(1-\bepsilon)S'} < \|DF^{-S'+1}|_{E_{q_{-1}}^v}\| < K \lambda^{-(1+\bepsilon)S'}.
\end{equation}

Denote
$$
\hE_{q_{-S'}}^v := DF^{-S'+1}(E_{q_{-1}}^v).
$$
By \propref{ext deriv bound}, we have
\begin{equation}\label{eq.reg pres 2}
\bK^{-1}\lambda^{\bepsilon(S-S')} < \|DF^{-(S-S')}|_{\hE_{q_{-S'}}^v}\| < \bK\lambda^{-(1+\bepsilon)(S-S')}.
\end{equation}

Since the gap between $n_{m-1}$ and $n_m$ is not tempered, we have
$$
\rho_1^{n_m} < \rho_1^{n_m-n_{m-1}} < \lambda^{\bepsilon R_{n_{m-1}}}.
$$
Denote $\omega := \logl\rho_1$. Then $n_m > (\bepsilon/\omega) R_{n_{m-1}}$. There exists a uniform constant $\hR = \hR(\epsilon, \omega, \bfb) \in \bbN$ such that for all $R \geq \hR$, we have
$$
2^{(\epsilon/\omega)R} > \frac{\bfb}{\epsilon} R.
$$
By uniformly increasing $n_0$ if necessary, we may assume that $R_{n_0} \geq R$, so that
$$
\bepsilon S' > \bepsilon R_{n_m} > \bepsilon 2^{n_m} > \bepsilon 2^{(\bepsilon/\omega)R_{n_{m-1}}} > \bfb R_{n_{m-1}} > S-S'.
$$
Therefore,
\begin{equation}\label{eq.reg pres 3}
\|DF^{-(S-S')}|_{\hE_{q_{-S'}}^v}\| > \bK^{-1}\lambda^{\bepsilon(S-S')} > \bK^{-1}\lambda^{\bepsilon S'}\lambda^{-(1-\bepsilon)(S-S')}.
\end{equation}

Concatenating \eqref{eq.reg pres 1} with \eqref{eq.reg pres 2} and \eqref{eq.reg pres 3}, we conclude by \propref{jac bound} that $q_0$ is $R_{M+1}$-times backward $(\bfL, \delta, \lambda)$-regular (vertically) along $E_{q_0}^h$.
\end{proof}


\section{Realization of Renormalization Combinatorics}\label{sec.realization}

\subsection{For unimodal maps}

Consider a $C^2$-unimodal map $f$ with critical point $c$. For concreteness, assume $f''(c) > 0$. For $\eta > 0$, we say that $f$ has {\it $\eta$-gap} if $f(c) < c-\eta$, and {\it double $\eta$-gap} if $f(c) < f^2(c) < c-\eta$. By \lemref{sink combin}, if $f$ has double $\eta$-gap for some $\eta >0$, then $c$ converges to a sink of period 1 or 2. Lastly, for $\chi \in \bbN$, we say that $f$ has {\it $(\eta, \chi)$-kneading} if
$$
f_{a_2}^{1+\chi}(c)+\eta < c < f_{a_2}^{1+i}(c)-\eta
\matsp{for}
1 \leq i < \chi.
$$

Let $\frI \subset \bbR$ be an interval, and consider a $C^1$-smoothly parameterized family $\frf = \{f_a\}_{a \in \frI}$ of $C^2$-unimodal maps (i.e. $f_a$ depends $C^1$-smoothly on the parameter $a$). For $\eta >0$ and $\chi \geq 2$, we say that $\frf$ is {\it $(\eta, \chi)$-full} if the following conditions hold.
\begin{itemize}
\item For all $a \in \frI$, the map $f_a$ has $\eta$-gap and $\chi$-bounded kneading.
\item There exists $a_1 \in \frI$ such that $f_{a_1}$ has double $\eta$-gap.
\item There exists $a_2 \in \frI$ such that $f_{a_2}$ has $(\eta, \chi)$-kneading.
\end{itemize}

Recall the definition of renormalization type $\tau(f)$ of a valuably renormalizable unimodal map $f$ given in \subsecref{subsec.ren combin 1d}.

\begin{prop}\label{1d full}
Consider a $C^1$-smoothly parameterized family $\frf = \{f_a\}_{a \in \frI} \subset \frU^2(K)$ for some $K \geq 1$. Suppose that $\frf$ is $(\eta_0, \bfb)$-full for some $\eta_0 >0$ and $\bfb \geq 2$. Then for any $\bfb$-bounded renormalization type $T$, there exist a uniform constant $\eta_1 = \eta_1(K, \bfb) >0$ and an interval $\frI_1 \subset \frI$ such that $\tau(f_a) = T$ for $a \in \frI_1$, and $\frf_1 := \{\cRod(f_a)\}_{a\in\frI_1} \subset \frU^2$ is $(\eta_1, \bfb)$-full.
\end{prop}

\begin{proof}
The Intermediate Value Theorem by Milnor-Thurston \cite{MiTh} implies that there exist a parameter interval $\hat\frI_1 \subset \frI$ such that $\tau(f_a) = T$ for $a \in \frI_1$, and $\hat\frf_1 := \{\cRod(f_a)\}_{a\in\hat\frI_1}$ is full.

Clearly, there exist $K' = K'(K, \bfb) \geq 1$ such that $\hat\frf_1 \subset \frU^2(K')$. Let $\frI_1'$ be a maximal subinterval of $\hat\frI_1$ such that for $a \in \frI_1'$, the critical point $0$ of $\cRod(f_a)$ does not converge to a fixed attracting sink. Then by \lemref{1d lower bound}, we see that there exists a uniform constant $\eta_1 = \eta_1(K') >0$ such that $f_a$ has $\eta_1$-gap. Moreover, observe that there exist $a_1 \in \partial \frI_1'$ such that the critical point $0$ of $\cRod(f_{a_1})$ converges to a fixed parabolic sink of flip type. By decreasing $\eta_1$ a uniform amount if necessary, we see that $\cRod(f_{a_1})$ has double $\eta_1$-gap.

Applying the Intermediate Value Theorem again, we can restrict $\frI_1'$ to a smaller subinterval such that for $a \in \frI_1$, the map $\cRod(f_a)$ has $\bfb$-bounded kneading. Moreover, the endpoints of $\frI_1$ are $a_1$ and $a_2$, so that for $\hf_{a_2} := \cRod(f_{a_2})$, we have
$$
\hf_{a_2}^{1+\chi}(0) < 0 < \hf_{a_2}^{1+i}(0)
\matsp{for}
1 \leq i < \chi,
$$
and $0$ does not converge to a sink of period less than $\chi$. By decreasing $\eta_1$ a uniform amount if necessary, it follows from \lemref{1d lower bound} that $\hf_{a_2}$ has $(\eta_1, \chi)$-kneading.
\end{proof}

\subsection{For H\'enon-like maps}

For $N \in \bbN$, let $F$ be the $N$-times $(L, \epsilon, \lambda)$-regularly H\'enon-like renormalizable map with combinatorics of $\bfb$-bounded type considered in \subsecref{subsec.unif c1}. Let $\eta_1 = \eta_1(\bfK, \bfb) >0$ be the constant given in \propref{1d full}. We assume that $N$ is sufficiently large, so that for some $0 \leq n_0 \leq N$, we have (in addition to \eqref{eq.proper depth 2}):
\begin{equation}\label{eq.proper depth 3}
\lambda^{\epsilon R_{n_0}} < c\eta_1
\end{equation}
for some sufficiently small constant $c \in (0,1)$ independent of $F$.

For $0 \leq n \leq N$, let $F_n := \cR^n(F)$ and $f_n := \Piod(F_n)$. Denote the domain of $F_n$ by $D^n$. Recall the definition of the renormalization type $\tau(F_n)$ of $F_n$ given in \subsecref{subsec.ren combin 2d} (see \eqref{eq.proj value seq}).

\begin{prop}\label{2d ren from 1d ren}
Suppose that $f_N$ is valuably renormalizable with return time $r_N \leq \bfb$, and that $\cRod(f_N)$ has $\eta_1$-gap and $\bfb$-bounded kneading. Then $F$ is $(N+1)$-times H\'enon-like renormalizable, and $\tau(F_N) = \tau(f_N)$.
\end{prop}

\begin{proof}
Let $J_0$ be the $r_N$-periodic interval of $f_N$ containing the critical value $f_N(0)$. Denote the critical point of $g := f_N^{r_N}|_{J^{N+1}}$ by $c$. By \lemref{1d struct}, we can assume that
$$
J_0 := [g(c), c] \cup [c, g^2(c)].
$$
Denote $J_i := g^i(J_0)$ for $0 \leq i < r_N$. We claim that $J_i$ and $J_j$ for $i \neq j$ are uniformly far apart. By \lemref{1d upper bound}, there exists a uniform constant $\rho_1 >0$ such that if $g^2(c) < c + \rho_1$, then $g^3(c) < c-\rho_1$. Considering the two cases $g^2(c) < c + \rho_1$ and $g^2(c) > c + \rho_1$ separately, and arguing as in the proof of \propref{c1 bound geo}, the claim follows.

For $0 \leq i < r_N$, let $\tiJ_i$ be an interval that compactly contains $J_i$, and the components of $\tiJ_i \setminus J_i$ have lengths commensurate to $\lambda^{\bepsilon R_n}$. Define
$$
W_i := \tiJ_i \times \pi_v(D^N).
$$
Observe that $W_i \cap W_j =\varnothing$ if $i \neq j$. Moreover, by \propref{2d close 1d}, it follows that we have $F_n(W_i)\Subset W_{i+1 \md{r_n}}$ for $0\leq i < r_n$.

Let
$$
\cW_i := (\cS^N\circ\Phi^N)^{-1}(W_i).
$$
Denote $R_{N+1} := r_NR_N$. By \thmref{rescale}, we see that if $i > 0$, then
$$
\cW_i \cap \cV_{v_0}^N(\rho_1^N) = \varnothing
$$
for some uniform constant $\rho_1 \in (0,1)$. Using Lemmas \ref{flat} iii) and \ref{quad flat} and proceeding by induction, one can show that under $F^{R_{N+1}}$, horizontal foliation of $\cW_0$ maps to a foliation by vertical quadratic curves in $\cW_0$. Similarly, using Lemmas \ref{flat} iv) and \ref{quad straight} and proceeding by induction, one can show that under $F^{-R_{N+1}}$, horizontal foliation of $F^{R_{N+1}}(\cW_0)$ maps to a $\lambda^{(1-\bepsilon)R_N}$-vertical foliation of $\cW_0$. Let $\Psi^{N+1}$ be a genuine horizontal chart that rectifies this vertical foliation. Then it follows immediately that $(F^{R_{N+1}}, \Psi^{N+1})$ is a H\'enon-like return.
\end{proof}

\begin{proof}[Proof of Theorem E]
Let $\frF = \{F_a\}_{a\in\frI} \subset \frHL^6$ be a $C^1$-smoothly parameterized family of H\'enon-like maps that satisfy the following properties. For $a \in \frI$, the map $F_a$ is $n_0$-times $(L, \epsilon, \lambda)$-regularly H\'enon-like renormalizable with combinatorics of $\bfb$-bounded type, where $n_0$ is sufficiently large so that \eqref{eq.proper depth 2} and \eqref{eq.proper depth 3} are satisfied.

For some $N \geq n_0$, suppose that $F_a$ has $N$ H\'enon-like returns $\{(F_a^{R_n}, \Phi_a^n)\}_{n=1}^N$ with combinatorics of $\bfb$-bounded type. For $n_0 \leq n \leq N$,  \thmref{pres reg} implies that $(F^{R_n}, \Phi^n)$ is $(\bfL, \delta, \lambda)$-regular for some uniform constants $\bfL \geq 1$ and $\delta \in (\bepsilon,1)$ with $\bfb \bdelta < 1$. Moreover, by \thmref{rescale}, we have $\cR^n(F_a) \in \frHL^5_{\lambda_n}(\bfK)$ with $\lambda_n := \lambda^{(1-\bdelta) R_n}$.

Proceeding by induction, suppose that there exists an interval of parameters $\frI_N \subset \frI$ such that the following properties hold.
\begin{itemize}
\item For $a \in \frI_N$, the map $F_a$ is $N$-times H\'enon-like renormalizable with combinatorics of $\bfb$-bounded type.
\item For all $a, a' \in \frI_N$, we have $\tau_N(F_a) = \tau_N(F_{a'})$.
\item Denote $F_{N,a} := \cR^N(F_a)$ and $f_{N,a}:= \Piod(F_{N,a})$. Then $\frf_N := \{f_{N, a}\}_{a\in\frI_N}$ forms a $(\eta_1/2, \bfb)$-full $C^1$-smoothly parameterized family.
\end{itemize}

Let $T$ be a renormalization type with return time $r_n \leq \bfb$. \lemref{1d full} implies that there exists an interval $\frI_{N+1} \subset \frI_N$ such that $\tau(f_{N, a}) = T$ for $a \in \frI_{N+1}$, and $\{\cRod(f_{N, a})\}_{a\in\frI_{N+1}}$ is $(\eta_1, \bfb)$-full. By \propref{2d ren from 1d ren}, $F_a$ is $(N+1)$-times H\'enon-like renormalizable, and $\tau(F_{N,a}) = \tau(f_{N, a})$. Moreover, we see from \propref{2d close 1d} that
$$
\|f_{N, a}^{r_N} - \Piod(F_{N,a}^{r_N})\|_{C^0} < \lambda^{(1-\bdelta)R_N} < c\eta_1.
$$
It follows that $\{\Piod\circ\cR^{N+1}(F_a)\}_{a\in\frI_{N+1}}$ is $(\eta_1/2, \bfb)$-full.
\end{proof}


\section{Uniform $C^r$-Bounds}\label{sec.cr bound}

Let $F$ be the infinitely regularly H\'enon-like renormalizable map with combinatorics of $\bfb$-bounded type considered in \subsecref{subsec.unif c1} (with $N=\infty$). For $n \geq n_0$, denote 
$$
\tiF_n = p\cR^n(F) := \Psi^n \circ F^{R_n} \circ (\Psi^n)^{-1}
\matsp{and}
\tif_n := \Piod(F_n).
$$
By \corref{a priori cor}, there exists a uniform constant $\bfK \geq 1$ such that $\tif_n$ has $\bfK$-bounded non-linearity.

Consider the arcs
$$
\cI^n_0 := (\Psi^n)^{-1}(I^n_0\times\{0\}) = \cI^h_0 \cap \cB^n_0 \ni v_0
$$
and $\cI^n_i := F^i(\cI^n_0)$ for $i \in \bbN$. Let $\{\cJ^n_i\}_{i=0}^{R_n-1}$ be the collection of arcs given in \eqref{eq.Js}. Recall that for $n_0 \leq m \leq n$; $0 \leq k < R_n/R_m$ and $0 \leq i < R_m$, we have
\begin{equation}\label{eq.J dist}
\cJ^n_0 := \cI^n_0
\comma
\cJ^n_{kR_m} \subset \cJ^m_0
\matsp{and}
\cJ^n_{i+kR_m} = \hH_i(\cJ^n_{kR_m}).
\end{equation}
Moreover, $\{\cJ^n_i\}_{i=0}^{R_n-1}$ is pairwise disjoint by \lemref{J order}.

The map
$$
\phi_0 := P_0|_{\cI^h_0} : \cI^h_0 \to I^h_0
$$
gives a parameterization of $\cI^h_0$ by its arclength. For $n \geq n_0$ and $0 \leq l < R_n/R_{n_0}$, let
$$
J^n_{lR_{n_0}} := \phi_0(\cJ^n_{lR_{n_0}}).
$$
Observe that $\{J^n_{lR_{n_0}}\}_{l=0}^{R_n/R_{n_0}-1}$ is a pairwise disjoint set of intervals 
contained in $\bbR$. Moreover,
\begin{equation}\label{eq.first gen J}
J^{n+1}_{kR_n} = \Piod \circ \tiF_n^k (J^{n+1}_0)
\matsp{for}
0 \leq k < r_n.
\end{equation}

Let $\gamma \subset \Gamma$ be $C^1$-curves in $\bbR^2$. We say that $\gamma$ is {\it commensurable with $\Gamma$} if $|\gamma| \asymp |\Gamma|$.

\begin{prop}\label{a priori geo}
Let $n \geq n_0$ and $0 \leq i < R_n$. Then any arc $\cJ^{n+1}_{i+ kR_n}$ for some $0 \leq k< r_n$, or any component of
$$
\cJ^n_i \setminus \bigcup_{k=0}^{r_n-1} \cJ^{n+1}_{i+kR_n}
$$
is commensurable with $\cJ^n_i$. Consequently, there exists a uniform constant $\rho\in(0,1)$ such that
$$
\sum_{i = 0}^{R_n-1} |\cJ^n_i| < O(\rho^n).
$$
\end{prop}

\begin{proof}
Denote the critical point of $\tif_n$ by $c^n$. Then by \propref{c1 bound geo}, we see that each component of
$$
J^n_0 \setminus \bigcup_{k=0}^{3r_n+1}\tif_n^k(c^n)
$$
is commensurate with $J^n_0$. Thus, by \propref{2d close 1d} and \eqref{eq.first gen J}, this implies the result in the case $i=0$. The case $0 < i < R_n$ then follows immediately from \thmref{a priori} and \eqref{eq.J dist}.
\end{proof}

The map
$$
\phi_{-1} := P_{-1}|_{\cI^{n_0}_{R_{n_0}-1}} : \cI^{n_0}_{R_{n_0}-1} \to I^{n_0}_{R_{n_0}-1}
$$
gives a parameterization of $\cI^{n_0}_{R_{n_0}-1}$. Denote
$$
J^n_{lR_{n_0}-1} := \phi_{-1}(\cJ^n_{lR_{n_0}-1})
\matsp{for}
1 \leq l \leq R_n/R_{n_0}.
$$
Observe that $\{J^n_{lR_{n_0}-1}\}_{l=1}^{R_n/R_{n_0}}$ is a pairwise disjoint set of intervals contained in $\bbR$. Define
\begin{equation}\label{eq.gammas}
\gamma^n_{-1} := \bigcup_{l=3}^{R_n/R_{n_0}-1} J^n_{lR_{n_0}-1} \subset I_{-1}^h
\matsp{and}
\gamma^n_0 := \bigcup_{l=3}^{R_n/R_{n_0}-1} J^n_{lR_{n_0}} \subset I_0^h.
\end{equation}

\propref{H decomp fine} gives the following decomposition of $\hH_{R_n-1}$:
$$
\hH_{R_n-1}|_{\cI^n_0} = F^{R_{n_0}-1}|_{\cI^{n_0}_0} \circ \chH_{\frac{R_n}{R_{n_0}}-1} \circ \ldots \circ \chH_3 \circ \cP^{n_0}_0 \circ F^{2R_{n_0}}|_{\cI^n_0}.
$$
where for $3\leq l < R_n/R_{n_0}$, we have
$$
\chH_l :=  \cP^{\hm_l}_0 \circ F \circ \left(\cP^{n_0}_{-1}|_{\chcI^n_l}\right)^{-1} \circ F^{R_{n_0}-1}|_{\cI^{n_0}_0}.
$$
Define
$$
\Gamma^n_{-1} := \bigcup_{l=3}^{R_n/R_{n_0}-1} \chcI^n_l \subset \cU_{-1} \subset \bbR^2.
$$

\begin{lem}\label{pieces line up}
For $n\in\bbN$ and $3 \leq l < R_n/R_{n_0}$, the map $P_{-1}$ restricts to a diffeomorphism from $\chcI^n_l$ to $J^n_{lR_{n_0}-1}$ (and hence, also from $\Gamma^n_{-1}$ to $\gamma^n_{-1}$). Define
$$
g^n_{-1} := \pi_v \circ\Phi_{-1} \circ (P_{-1}|_{\Gamma^n_{-1}})^{-1}.
$$
Then
$$
\|g^n_{-1}|_{(-t, t)}\|_{C^r} =O(t^{1/\bepsilon}).
$$
\end{lem}

\begin{proof}
The first claim follows immediately from \propref{H decomp fine}.

Observe that $\hm_l$ is the largest integer such that
$$
\{0\} \cup J^n_{lR_{n_0}-1} \subset J^{\hm_l}_{R_{\hm_l}-1}.
$$
Moreover,
$$
J^n_{lR_{n_0}-1} \subset J^{\hm_l+1}_{\ha_lR_{n_0}-1}
\matsp{and}
0 \notin J^{\hm_l+1}_{\ha_lR_{n_0}-1}.
$$

By \propref{H decomp fine}, $\chcI^n_l$ is $\lambda^{(1-\bepsilon)R_{\hm_l}}$-horizontal. Additionally, by \propref{a priori geo}, we have
$$
\dist(0, \chI^n_l) \asymp \rho^{\hm_l}
$$
for some uniform constant $\rho \in (0,1)$. The estimate on $g^n_{-1}$ follows.
\end{proof}

Let $G : \cI \to \cJ$ be a $C^1$-diffeomorphism between two $C^1$-curves $\cI, \cJ \subset \bbR^2$. Define the {\it zoom-in operator} $\bfZ$ by
$$
\bfZ(G)(t) := |\cJ|^{-1}\cdot \phi_{\cJ}^{-1}\circ G\circ \phi_{\cI}(|\cI| t),
$$
where $\phi_{\cI} : [0, |\cI|] \to \cI$ is the parameterization of $\cI$ by its arclength (and $\phi_{\cJ}$ similarly defined). Note that $\bfZ(G) : [0, 1] \to [0,1]$.

This rest of this section is devoted to proving the following theorem.

\begin{thm}\label{cr bound proj}
There exists a universal constant $K > 0$ such that for all $n \geq n_0$ sufficiently large and $1 \leq i < R_n$, we have
$$
\|\bfZ(\hH_i|_{\cI^n_0})\|_{C^r} <K.
$$
\end{thm}

Define
$$
\bfq(x) := \sign(x)x^2.
$$
Denote $\chI^h_0 := \bfq^{-1}(I^h_0)$. For $n \geq n_0$ and $0 \leq l < R_n/R_{n_0}$, let $\chJ^n_{lR_{n_0}} := \bfq^{-1}(J^n_{lR_{n_0}}).$ The proof of \thmref{cr bound proj} relies on the following key result.

\begin{prop}\label{diff quad decomp}
Let $n\in\bbN$. There exists a $C^r$-diffeomorphism $\chh^n : I^h_0 \to \chI^h_0$ with
$$
\|(\chh^n)^{\pm 1}\|_{C^r} = O(1)
$$
such that for $1\leq l \leq R_n/R_{n_0}$, we have
$$
\phi_0\circ \hH_{lR_{n_0}} \circ \phi_0^{-1}|_{I^n_0} = (\bfq^n_l \circ \chh^n_l) \circ \ldots\circ (\bfq^n_2 \circ \chh^n_2) \circ  (\bfq^{n}_1\circ \chh^{n}_1),
$$
where $\chh^n_l : J^n_{(l-1)R_{n_0}} \to \chJ^n_{lR_{n_0}}$ and $\bfq^n_l : \chJ^n_{lR_{n_0}} \to J^n_{lR_{n_0}}$ are diffeomorphisms given by
\begin{equation}\label{eq.diff quad decomp}
\chh^n_l := \chh^n|_{J^n_{(l-1)R_{n_0}}}
\matsp{and}
\bfq^n_l := \bfq|_{\chJ^n_{lR_{n_0}}}.
\end{equation}
\end{prop}

\begin{lem}
For $n\in\bbN$ and $3 \leq l < R_n/R_{n_0}$, we have
$$
P^{\hm_l}_0 \circ F \circ (\cP^{n_0}_{-1}|_{\chcI^n_l})^{-1} \circ F^{R_{n_0}-1} \circ \phi_0^{-1}|_{J^n_{(l-1)R_{n_0}}} = \bfq^n_l \circ \chh^n_l(x),
$$
where $\chh^n_l$ and $\bfq^n_l$ are as defined in \eqref{eq.diff quad decomp}.
\end{lem}

\begin{proof}
Define $\chgamma^n_0 := \bfq^{-1}(\gamma^n_0),$ where $\gamma^n_0$ is given in \eqref{eq.gammas}. By Lemmas \ref{quad mapping} and \ref{pieces line up}, there exists a $C^r$-diffeomorphism $\psi_{-1, 0}^n : \gamma^n_{-1} \to \chgamma^n_0$ with
$$
\|(\psi_{-1, 0}^n)^{\pm 1}\|_{C^r} = O(1)
$$
such that
$$
P^{\hm_l}_0 \circ F \circ \Phi_{-1}^{-1}\circ G^n_{-1}|_{\chI^n_l} = \bfq \circ \psi_{-1, 0}^n|_{\chI^n_l},
$$
where $G^n_{-1}(x) := (x, g^n_{-1}(x))$. Precomposing with $P_{-1}\circ F^{R_{n_0}-1} \circ \phi_0^{-1}|_{J^n_{(l-1)R_{n_0}}}$ gives the desired result.
\end{proof}

\begin{proof}[Proof of \thmref{cr bound proj}]

For $1\leq l < R_n/R_{n_0}$, let $n_0\leq \hm_l \leq n$ be the largest integer such that 
$$
\{0\} \cup \chJ^n_{lR_{n_0}} \subset \chJ^{\hm_l}_{R_{\hm_l}}.
$$
Denote $\bbL^n_m := \{1\leq l < R_n/R_{n_0} \; | \; \hm_l = m\}$. Then $l \in \bbL^n_m$ if and only if
$$
\chJ^n_{lR_{n_0}} \subset \chJ^m_{R_m}
\matsp{and}
\chJ^n_{lR_{n_0}-1} \cap \chJ^{m+1}_{R_{m+1}} = \varnothing.
$$
Note that
$$
\bigcup_{m=n_0}^n \bbL^n_m = \{1\leq l < R_n/R_{n_0}\}.
$$

Let $U^m_{R_m}$ be the component of $\chJ^m_{R_m} \setminus \chJ^{m+1}_{R_{m+1}}$ contained in $\bbR^-$. Applying \propref{a priori geo} and \lemref{avila lem 1} to $\bfZ\left(\bfq|_{U^m_{R_m}}\right)$, we see that
$$
\sum_{l \in \bbL^n_m} \|\bfZ(\bfq^n_l)-\Id\|_{C^r} = O(\rho^m)
$$
for some uniform constant $\rho \in (0,1)$. The result now follows from \propref{a priori geo}, \propref{diff quad decomp}, and Lemmas \ref{avila lem 1} and \ref{avila lem 2}.
\end{proof}

\begin{thm}\label{cr bound 2d}
For all $n \in \bbN$ sufficiently large, we have $\|\cR^n(F)\|_{C^r} = O(1).$
\end{thm}

\begin{proof}
By \thmref{cr bound proj} and \eqref{eq.proj to original}, we see that
$$
\|\Piod\circ \cR^n(F)\|_{C^r} = O(1).
$$
Since $\cR^n(F)$ is a $\lambda^{(1-\bepsilon)R_n}$-thin H\'enon-like map, the result follows.
\end{proof}


\section{Exponentially Small Pieces}\label{sec.shrink}

Let $F$ be the infinitely regularly H\'enon-like renormalizable map considered in \secref{sec.cr bound}. The goal of this section is to prove Theorem B.

For any integer $l \geq 2$, we have
\begin{equation}\label{eq.decomp iterate}
dR_{n_0} = a_1 R_{n_1} + \ldots + a_k R_{n_k},
\end{equation}
for some
\begin{itemize}
\item $n_0 \leq n_1 < \ldots < n_k$,
\item $1 \leq a_m < r_{n_k}$ for $1 \leq m< k$, and
\item $2 \leq a_k < 2r_{n_k}$.
\end{itemize}
Define
$$
\hcH_{lR_{n_0}} := F^{a_1R_{n_1}}\circ\cP^{n_2}_0\circ F^{a_2R_{n_2}}\circ  \ldots \circ \cP^{n_k}_0\circ F^{a_kR_{n_k}}\circ \cP^{n_k}_0,
$$
where $\cP^n_0 : \cB^n_0 \to \cI^n_0$ for $n \geq n_0$ is the projection map onto $\cI^n_0$ given by
$$
\cP^n_0 := (\Psi^n)^{-1}\circ \Pi_h \circ \Psi^n.
$$
Denote $\hm(lR_{n_0}) := n_1$ and $\hk(lR_{n_0}) := n_k$. Recall the definition of $\hH_i$ given in \secref{sec.a priori}. Then we have
\begin{equation}\label{eq.h commute proj}
\cP^{\hm(lR_{n_0})}_0 \circ \hcH_{lR_{n_0}} = \hH_{lR_{n_0}} \circ \cP^{\hk(lR_{n_0})}_0.
\end{equation}

\begin{lem}\label{H vs F}
Let $i = lR_{n_0}$ for some $l \geq 2$. Then for $n \geq \hk(i)$, we have
$$
\|\hcH_i - F^i|_{\cB^{\hk(i)}_0}\|_{C^0} < \lambda^{(1-\bepsilon)R_{\hm(i)}}.
$$
\end{lem}

\begin{proof}
By \thmref{crit chart} and \propref{c1 bound 2d}, $\|(\Psi^m)^{\pm 1}\|_{C^{r+3}}$ and $\|\tiF_m\|_{C^1}$ are uniformly bounded. Moreover, by \thmref{crit chart} ii), we have
\begin{equation}\label{eq.proj F}
\|\tiF_m - \tiF_m \circ \Pi_h \|_{C^{r+3}} < \lambda^{(1-\bepsilon)R_m}.
\end{equation}

Let $i = lR_{n_0}$ be given by \eqref{eq.decomp iterate} with $\hk(i) = n_k < n$. Note that
$$
F^{R_{n_k}} = (\Psi^{n_k})^{-1}\circ \tiF_{n_k} \circ \Psi^{n_k}
$$
and
$$
\hcH_{R_{n_k}} = F^{R_{n_k}} \circ \cP^{n_k}_0 = (\Psi^{n_k})^{-1}\circ \left(\tiF_{n_k} \circ \Pi_h\right) \circ \Psi^{n_k}.
$$
Hence, we see by \eqref{eq.proj F} and \lemref{vary compose} that
$$
\|\hcH_{R_{n_k}}-F^{R_{n_k}}|_{\cB^{n_k}_0}\|_{C^0} < \lambda^{(1-\bepsilon)R_{n_k}}.
$$
Moreover,
$$
\hcH_{a_kR_{n_k}} = \left((\Psi^{n_k})^{-1}\circ \tiF^{a_k-1}_{n_k} \circ \Psi^{n_k}\right)\circ \hcH_{R_{n_k}}
$$
and
$$
F^{a_kR_{n_k}} = \left((\Psi^{n_k})^{-1}\circ \tiF^{a_k-1}_{n_k} \circ \Psi^{n_k}\right)\circ F^{R_{n_k}}
$$
Thus, another application of \lemref{vary compose} imply
$$
\|\hcH_{a_kR_{n_k}} - F^{a_kR_{n_k}}|_{\cB^{n_k}_0}\|_{C^0} < \lambda^{(1-\bepsilon)R_{n_k}}.
$$

Proceeding by induction, suppose that
$$
\|\hcH_{i_{j+1}} - F^{i_{j+1}}|_{\cB^{n_k}_0}\|_{C^0} < \lambda^{(1-\bepsilon)R_{n_{j+1}}}.
$$
where $1\leq j < k$ and
$$
i_{j+1} := a_{n_{j+1}}R_{n_{j+1}}+ \ldots + a_{n_k} R_{n_k}.
$$
Write
$$
\hcH_{i_j} = (\Psi^{n_j})^{-1} \circ \tiF_{n_j}^{a_{n_j}-1} \circ \left(\tiF_{n_j}\circ \Pi_h \right)\circ \Psi^{n_j}\circ  \hcH_{i_{j+1}}
$$
and
$$
F^{i_j}|_{\cB^{n_k}_0} = (\Psi^{n_j})^{-1} \circ \tiF_{n_j}^{a_{n_j}-1}\circ \tiF_{n_j}\circ \Psi^{n_j}\circ F^{i_{j+1}}|_{\cB^{n_k}_0}.
$$
Applying \lemref{vary compose}, the result follows.
\end{proof}

\begin{lem}\label{H shrinks}
There exists a uniform constant $\rho \in(0,1)$ such that
$$
\sum_{l=2}^{R_n/R_{n_0}}\diam(\hcH_{lR_{n_0}}(\cI^n_0)) =O(\rho^n).
$$
\end{lem}

\begin{proof}
For $2 \leq l \leq R_n/R_{n_0}$, consider the curve $\chcI^n_l \subset \cU_{-1}$ given in \propref{H decomp fine}. By \eqref{eq.h commute proj}, we have
$$
\hcH_{lR_{n_0}}(\cI^n_0) = F \circ \left(\cP^{n_0}_{-1}|_{\chcI^n_l}\right)^{-1} \circ F^{R_{n_0}-1}\circ \hH_{(l-1)R_{n_0}}(\cI^n_0).
$$
Thus, $\{\hcH_{lR_{n_0}}(\cI^n_0)\}_{l=2}^{R_n/R_{n_0}}$ is the image of
$$
\{\cJ^n_{lR_{n_0}} := \hH_{lR_{n_0}}(\cI^n_0)\}_{l=1}^{R_n/R_{n_0}-1}
$$
under
$$
G_n := F \circ \left(\cP^{n_0}_{-1}|_{\Gamma^n_{-1}}\right)^{-1} \circ F^{R_{n_0}-1},
$$
where
$$
\Gamma^n_{-1} := \bigcup_{l=2}^{R_n/R_{n_0}} \chcI^n_l.
$$
Since $\Gamma^n_{-1}$ is uniformly horizontal, $\|G_n\|_{C^r} = O(1)$. The result now follows from \propref{a priori geo}.
\end{proof}

\begin{thm}\label{exp shrink}
There exists a uniform constant $\tirho \in(0,1)$ such that for $n \in \bbN$, we have
$$
\sum_{i=0}^{R_n-1}\diam(F^i(\cB^n_{R_n})) =O(\tirho^n).
$$
\end{thm}

\begin{proof}
Choose $n_0 < m < n$ to be determined later. By \lemref{H vs F}, we see that for $1 \leq l < R_n/R_m$, we have
$$
\diam(F^{lR_m}(\cB^n_{R_n})) < \diam(\hcH_{lR_m}(\cI^n_0)) + \lambda^{(1-\bepsilon)R_{\hm(lR_m)}}.
$$
Thus, by \lemref{H shrinks}, we have
$$
\sum_{l=0}^{R_n/R_m-1}\diam(F^{lR_m}(\cB^n_{R_n})) = O(\rho^n) + \frac{R_n}{R_m}\lambda^{(1-\bepsilon)R_m}.
$$
For $m$ sufficiently large, the expression on the right is bounded by $O(\rho_1^n)$ for some uniform constant $\rho_1 \in (\rho, 1)$.

Let $i = j + a_0R_{n_0} + \ldots + a_{m-1}R_{m-1} + lR_m$ with $0 \leq j < R_{n_0}$; $0 \leq a_k < r_k$ for $n_0 \leq k < m$, and $1 \leq l < R_n/R_m$. We can write
$$
F^{i-lR_m} = F^j\circ (\Psi^{n_0})^{-1}\circ \tiF_{n_0}^{a_0}\circ \Psi^{n_0}\circ \ldots \circ(\Psi^{m-1})^{-1}\circ \tiF_{m-1}^{a_{m-1}}\circ \Psi^{m-1}.
$$
By \thmref{crit chart} and \propref{c1 bound 2d}, we see that
$$
\|F^{i-lR_m}\|_{C^1} < K^m
$$
for some uniform constant $K \geq 1$. Hence,
$$
\sum_{i=0}^{R_n-1}\diam(F^i(\cB^n_{R_n})) = R_mK^m\sum_{l=0}^{R_n/R_m-1}\diam(F^{lR_m}(\cB^n_{R_n})) = O(R_mK^m\rho_1^n).
$$
For $n/m$ sufficiently large, the expression on the right is bounded by $O(\tirho^n)$ for some uniform constant $\tirho \in (\rho_1, 1)$.
\end{proof}

Observe that Theorem B is an immediate consequence of \thmref{exp shrink}.


\section{Regular Unicriticality}\label{sec.unicrit}

Let $F$ be the infinitely regularly H\'enon-like renormalizable map considered in \secref{sec.cr bound}. Recall that the renormalization limit set of $F$ is given by
$$
\Lambda_F := \bigcap_{n=1}^\infty \bigcup_{i = 0}^{R_n-1} \cB^n_{R_n +i}.
$$
By Theorem B, $\Lambda_F$ supports a unique invariant probability measure $\mu$ given by the counting measure:
$$
\mu(\cB^n_i) = 1/R_n
\matsp{for}
n, i \in \bbN.
$$

\begin{prop}[Proposition 8.3 \cite{CLPY1}]\label{lya exp}
With respect to $\mu$, the Lyapunov exponents of $F$ on $\Lambda_F$ are $0$ and $\log\lambda_\mu < 0$ for some $\lambda_\mu \in (0,1)$.
\end{prop}

\begin{prop}[Proposition 4.1 \cite{CLPY1}]\label{homog}
For any $\eta >0$, there exist uniform constants $N_\eta \in \bbN$ and $C_\eta \geq 1$ such that the following holds. Let $p \in \cB^n_k$ and $E_p \in \bbP^2_p$, where $n \geq N_\eta$ and $k \geq 0$. Then for all $i \in \bbN$, we have:
\begin{equation}\label{eq.homog deriv}
C_\eta^{-1}\lambda_\mu^{(1+\eta)i}< \|DF^i|_{E_p}\| < C_\eta\lambda_\mu^{-\eta i}
\end{equation}
and
\begin{equation}\label{eq.homog jac}
C_\eta^{-1}\lambda_\mu^{(1+\eta)i}< \Jac_p(F^i) < C_\eta\lambda_\mu^{(1-\eta)i}.
\end{equation}
\end{prop}

For $p \in \cB^n_0$, define
$$
E^{v,n}_p := D(\Psi^n)^{-1}(E_{\Psi^n(p)}^{gv})
$$
and
$$
E^h_p := D(\Psi^n)^{-1}(E_{\Psi^n(p)}^{gh}) = D(\Phi_0)^{-1}(E_{\Phi_0(p)}^{gh}).
$$

\begin{thm}\label{lya reg henon return}
For any $\delta >0$, there exists $L_\delta \geq 1$ such that for all $n \in \bbN$, the $n$th H\'enon-like return $(F^{R_n}, \Psi^n)$ is $(L_\delta, \delta, \lambda_\mu)$-regular.
\end{thm}

\begin{proof}
Choose $\eta \in (0, \udelta)$. It suffices to show the result for $n \geq \max\{n_0, N_\eta\}$, where $N_\eta$ is given in \propref{homog}. Let $p_0 \in \cB^n_0$. By \lemref{hor deriv out} and \eqref{eq.homog jac}, we see that $p_0$ is $R_n$-times forward $(O(1), \baeta, \lambda_\mu)$-regular horizontally along $E_{p_0}^h$. The required forward $(L_\delta, \delta, \lambda_\mu)$-regularity (vertically) along
$$
E_{p_0}^{v,n} := (D\Psi^n)^{-1}(E_{\Psi^n(p_0)}^{gv})
$$
now follows from \propref{transverse for reg}.

Let $q_0 := p_{R_n}$. We claim that $q_{-1}$ is backward regular (vertically) along
$$
E_{q_{-1}}^v := (D\Phi_{-1})^{-1}(E_{\Phi_{-1}(q_{-1})}^{gv}).
$$
We argue similarly as in the proof of \thmref{pres reg}.

Let $S = lR_{n_0}$ for some $1 \leq l < R_n/R_{n_0}$. Write
$$
S = s_1R_{n_1} + \ldots + s_kR_{n_k},
$$
where $1 \leq s_i < r_{n_i}$ for $1 \leq i \leq k$. Let $1 \leq m \leq k$ be the smallest number such that $\{n_i\}_{i=m}^k$ is tempered. Denote
$$
S' := s_{n_m}R_{n_m} + s_{n_{m+1}}R_{n_{m+1}} + \ldots + s_{n_k}R_{n_k},
$$
and let
$$
\hE_{q_{-1}}:= DF^{S'-1}(E_{q_{-S'}}^{gh}).
$$
Subbing in \eqref{eq.temper in pf} and \eqref{eq.homog jac} (instead of \propref{jac bound}) into \eqref{eq.temper in jac}, we obtain
$$
K^{-(k-m)}C_\eta^{-1}\lambda_\mu^{(1+\eta)S'}< \|DF^{-S'+1}|_{E_{q_{-1}}^v}\| < K^{k-m}C_\eta \lambda_\mu^{(1-\eta)S'}.
$$
Then following the same argument as in the proof of \thmref{pres reg} (but using \eqref{eq.homog deriv} and \eqref{eq.homog jac} instead of \propref{ext deriv bound} and \propref{jac bound} respectively), we conclude that $q_{-1}$ is $(R_n-1)$-times backward $(L_\delta, \delta, \lambda_\mu)$-regular (vertically) along $E_{q_{-1}}^v$.
\end{proof}

Recall that by \thmref{crit rec}, we have
$$
\bigcap_{n=1}^\infty \cB^n_{R_n} = \{v_0\}.
$$

\begin{thm}
The orbit $\{v_m\}_{m\in\bbZ}$ is a regular quadratic critical orbit.
\end{thm}

\begin{proof}
By \thmref{lya reg henon return}, $v_0$ is infinitely forward and backward $(L_\delta, \delta, \lambda_\mu)$-regular along $E^*_{v_0} = E^{ss}_{v_0} = E^c_{v_0}$ for all $\delta > 0$. Thus, $\{v_m\}_{m\in\bbZ}$ is a regular critical orbit. The quadratic tangency of $W^{ss}(v_0)$ and $W^c(v_0)$ at $v_0$ is given in \propref{crit value} iii).
\end{proof}

\subsection{Critical cover}

Let $\delta = \bepsilon$ for some $\epsilon \in (0, 1)$. Choose $\eta \in (0, \uepsilon)$. \propref{homog} and \thmref{lya reg henon return} imply that by replacing $F$ with $\cR^{n_0}(F)$ for some $n_0 \in \bbN$ sufficiently large, we may henceforth assume the following.
\begin{itemize}
\item Conditions \eqref{eq.proper depth 0} and \eqref{eq.proper depth 2} hold with $\lambda = \lambda_\mu$ and $n_0 = 0$.
\item The map $F$ is {\it $\eta$-homogeneous}: for all $p \in \cB$ and $E_p \in \bbP^2_p$, we have
$$
\lambda_\mu^{1+\eta}< \|DF|_{E_p}\| < \lambda_\mu^{-\eta}
\matsp{and}
\lambda_\mu^{1+\eta}< \Jac_p F < \lambda_\mu^{1-\eta}.
$$
\item For $n \in \bbN$, the $n$th H\'enon-like return $(F^{R_n}, \Psi^n)$ is $(1, \eta, \lambda_\mu)$-regular.
\end{itemize}

Denote $\epsilon' := (1+\bepsilon)\epsilon > \epsilon$. For $z = (a,b) \in B^n_0$ and $t \geq 0$, let
$$
V_z(t) := [a - t, a+t] \times I^v_0.
$$
For $p \in \cB^n_0$ and $t \geq 0$, let
$$
\cV_p^n(t) := (\Psi^n)^{-1}(V_{\Psi^n(p)}(t)).
$$
Lastly, for $t > 0$ and $p \in \bbR^2$, denote
$$
\bbD_p(t) := \{q \in \bbR^2 \; | \; \dist(q, p) < t\}.
$$

We now show that $F$ is $(\delta, \epsilon)$-regularly unicritical on $\Lambda_F$ (see \defnref{def unicrit}). First, we need to define a suitable cover of the iterated preimages of critical value $v_0$. For $n \in \bbN$ and $1 \leq i < r_n$, let $\cC^n$ be the connected component of
$$
\cB^n_{R_n} \cap \cV_{v_{-R_n}}^n(\lambda_\mu^{\epsilon' R_n})
$$
containing $v_{-R_n}$. Define
$
\cC^n_i := F^i(\cC^n)
$
for
$
0\leq j <R_n,
$
and
$$
\bfC^N:=\bigcup_{n=1}^{N+1}\bigcup_{i =0}^{R_n-1} \cC^n_i.
$$
Note that $\{v_{-i}\}_{i=1}^{R_{N+1}} \subset \bfC^N.$

\begin{prop}
We have
$
\diam(\cC^n_i) < \lambda_\mu^{\epsilon R_n}.
$
Consequently,
$$
\bfC^N \subset \bigcup_{i=1}^{R_{N+1}} \bbD_{v_{-i}}(\lambda_\mu^{\epsilon i}).
$$
\end{prop}

\begin{proof}
By \thmref{crit chart} ii), $\cB^n_{R_n}$ is a $\lambda_\mu^{(1-\bepsilon)R_n}$-thick strip around the curve $F^{R_n}(\cI^n_0)$, which is vertical quadratic in $\cB^n_0$ with the vertical tangency $\lambda_\mu^{(1-\baeta)R_n}$-close to $v_0$. By \propref{value struct bound}, we have
$$
\cV_{v_{-R_n}}(\lambda_\mu^{\baeta R_n})\cap \cV_{v_0}(\lambda_\mu^{\baeta R_n}) = \varnothing.
$$
By \lemref{quad flat}, the connected component $\Gamma^n$ of the curve 
$$
\cI^n_{R_n} \cap \cV_{v_{-R_n}}(\lambda_\mu^{\baeta R_n})
$$
is $\lambda^{-\baeta R_n}$-horizontal in $\cB^n_0$. Consequently,
$$
\diam(\cC^n) \asymp |\Gamma^n|< \lambda^{-\baeta R_n}\lambda^{\epsilon' R_n}.
$$
Then by $\eta$-homogeneity of $F$, we have
$$
\diam(\cC^n_i) < \lambda^{-\baeta i}\diam(\cC^n)
$$
for $0\leq i < R_n$. The result follows.
\end{proof}

\subsection{Forward regularity away from the critical cover}

For all $p \in \Lambda_F \setminus \{v_0\}$, there exists a unique number $d_p \geq 0$ such that $p \in \cB^{d_p}_0 \setminus \cB^{d_p+1}_0$. Define $\depth(p) := d_p$. If $p = v_0$, define $\depth(p) = \infty$. Let $p_0 \in \Lambda_F$. For $N \in \bbN$, let $0 \leq S \leq N$ be the largest number satisfying
$$
d = \depth(p_S) \geq \depth(p_i)
\matsp{for}
0\leq i \leq N.
$$
Define the {\it valuable moment} and the {\it valuable depth of the $N$-times forward orbit of $p_0$} as
$$
\vm(p_0, N) := S
\matsp{and}
\vd(p_0, N) := d
$$
respectively.

\begin{lem}\label{reg induct}
Let $p_0 \in \Lambda_F$ and $N \in \bbN$. Denote $S:=\vm(p_0, N)$ and $d := \vd(p_0, N)$. Write
$$
S = s_0R_0 + s_1 R_1+ \ldots +s_d R_d,
$$
where $0 \leq s_i < r_i$ for $0\leq i \leq d$. If $p_0 \setminus \bfC^d$, then for $0 \leq n \leq d$ and $0 \leq s \leq s_n$, we have
$$
p_{S_{n-1} + sR_n} \notin \cV^n_{v_0}(\lambda_\mu^{\bepsilon R_n})
\matsp{where}
S_{n-1} := s_0 R_0 + \ldots +s_{n-1} R_{n-1}.
$$
\end{lem}

\begin{proof}
If $q_0 \in \Lambda_F \cap \cV^n_{v_0}(\lambda^{\bepsilon R_n})$, then it follows from \thmref{crit chart} ii) and $\eta$-homogeneity that $q_{-R_{n+1}} \in \cC^{n+1}$. Thus, if $p_{S'} \in \cV^n_{v_0}(\lambda_\mu^{\bepsilon R_n})$, where $S' := S_{n-1} + sR_n$, then $p_{-R_{n+1} + S'} \in \cC^{n+1}$. Therefore,
$$
p_0 \in \cC^{n+1}_{R_{n+1}-S'} \subset \bfC^n \subset \bfC^d.
$$
This is a contradiction.
\end{proof}

\begin{lem}\label{sq away from value}
Denote
$$
\epsilon_i = (1+\bepsilon)^i\bepsilon
\matsp{for}
i \geq 0.
$$
Let $q_0 \in \cB^n_0$ and $E_{q_0} \in \bbP^2_{q_0}$. If
$$
\measuredangle(E_{q_0}, E^{v, n}_{q_0}) > \lambda_\mu^{\epsilon_1 R_n},
$$
then
$$
\|DF^{R_n}|_{E_{q_0}}\| > \lambda_\mu^{\epsilon_2 R_n}.
$$
Moreover, if $q_{R_n} \notin \cV^n_{v_0}(\lambda_\mu^{\epsilon_0 R_n})$, then
$$
\measuredangle(E_{q_{R_n}}, E^{v, n}_{q_{R_n}}) > \lambda_\mu^{\epsilon_1 R_n}.
$$
\end{lem}

\begin{proof}
The estimate on $\|DF^{R_n}|_{E_{q_0}}\|$ follows immediately from the $(1, \eta, \lambda_\mu)$-regularity of the H\'enon-like return $(F^{R_n}, \Psi^n)$. The estimate on $\measuredangle(E_{q_{R_n}}, E^{v, n}_{q_{R_n}})$ follows immediately from \lemref{quad flat}.
\end{proof}

\begin{lem}\label{sq shallow}
For $n, k \in \bbN$, let $q_0 \in \cB^{n+k}_0$ and $E_{q_0} \in \bbP^2_{q_0}$. If
$$
R_n \geq \bepsilon R_{n+k}
\matsp{and}
\measuredangle(E_{q_0}, E^{v, n+k}_{q_0}) > \lambda_\mu^{\bepsilon R_{n+k}},
$$
then
$$
\|DF^{R_n}|_{E_{q_0}}\| > \lambda_\mu^{\bepsilon R_n}
\matsp{and}
\measuredangle(E_{q_{R_n}}, E^{v, n}_{q_{R_n}}) > \lambda_\mu^{\baeta R_n}.
$$
\end{lem}

\begin{proof}
Observe that
$$
\baeta R_n > \baeta\bepsilon R_{n+k} = \bepsilon R_{n+k}.
$$
So
$$
\lambda_\mu^{\baeta R_n} < \lambda_\mu^{\bepsilon R_{n+k}}.
$$
By \thmref{crit chart} i), we have
$$
\measuredangle(E^{v, n+k}_{q_0}, E^{v, n}_{q_0}) < \lambda_\mu^{(1-\baeta) R_n}.
$$
Hence,
$$
\measuredangle(E_{q_0}, E^{v, n}_{q_0}) > \lambda_\mu^{\bepsilon R_{n+k}} - \lambda_\mu^{(1-\baeta) R_n} > \lambda_\mu^{\baeta R_n} - \lambda_\mu^{(1-\baeta) R_n} = \lambda_\mu^{\baeta R_n}.
$$
Since $\depth(q_{R_n}) < n$, we have $q_{R_n} \notin \cV^n_{v_0}(\lambda_\mu^{\baeta R_n})$
by \propref{min domain}. The result then follows from \lemref{quad flat}.
\end{proof}

\begin{thm}\label{for reg away crit}
Let $p_0 \in \Lambda_F$ and $N \in\bbN$. Define
$$
\hE_{p_i} := D(F^i\circ\Phi_0^{-1})(E^{gh}_{p_0})
\matsp{for}
i \geq 0.
$$
If $p_0 \not\in \bfC^d$ with $d := \vd(p_0, N)$, then 
$$
\|DF^N|_{\hE_{p_0}}\| > \lambda_\mu^{\bepsilon N}.
$$
\end{thm}

\begin{proof}
Write
$$
S := \vm(p_0, N) = s_0 R_0 + \ldots + s_{d_{\In}} R_{d_{\In}}
$$
with $0 \leq s_n < r_n$ for $0 \leq n \leq d_{\In}\leq d$. Using Lemmas \ref{reg induct} and \ref{sq away from value}, and arguing inductively, we see that
$$
\|DF^S|_{\hE_{p_0}}\| > \lambda_\mu^{\bepsilon S}\comma
p_S \notin \cV^{d_{\In}}_{v_0}(\lambda_\mu^{\baeta R_{d_{\In}}})
\matsp{and}
\measuredangle(\hE_{p_S}, E_{p_S}^{v, d_{\In}}) > \lambda_\mu^{\baeta R_{d_{\In}}}.
$$

Let
$$
T := N-S = t_0R_0 + \ldots + t_{d_{\out}}R_{d_{\out}}
$$
with $0\leq t_n < r_n$ for $0 \leq n \leq d_{\out} < d$. If $d_{\out} \geq d_{\In}$, then
$$
p_S \notin \cV^{d_{\out}}_{v_0}(\lambda_\mu^{\baeta R_{d_{\out}}}) \subset \cV^{d_{\In}}_{v_0}(\lambda_\mu^{\uepsilon R_{d_{\In}}})
\matsp{and}
\measuredangle(\hE_{p_S}, E_{p_S}^{v, d_{\out}})> \lambda_\mu^{\baeta R_{d_{\out}}}.
$$
Thus, by \lemref{reg induct}, we have
$$
\|DF^{t_{d_{\out}}R_{d_{\out}}}|_{\hE_{p_S}}\| > \lambda_\mu^{\bepsilon t_{d_{\out}}R_{d_{\out}}}.
$$

Denote
$$
T_n :=  t_0R_0 + \ldots + t_nR_n
\matsp{and}
0 \leq n \leq d_{\out}.
$$
Note that $T_n < R_{n+1} \leq \bfb R_n$.

If $d_{\out} < d_{\In}$, let $\chd := d_{\out}$, and denote $t_{d_{\In}} := s_{d_{\In}}$. Otherwise, let $\chd < d_{\out}$ be the largest integer such that $t_{\chd} > 0$. Proceeding by induction, suppose for some $n \leq \chd$ with $t_n > 0$, we have
$$
\|DF^{N-T_n}|_{\hE_{p_0}}\| > \lambda_\mu^{\bepsilon (N-T_n)}
\matsp{and}
\measuredangle(\hE_{p_{N-T_n}}, E_{p_{N-T_n}}^{v, n+k}) > \lambda_\mu^{\baeta R_{n+k}},
$$
where $k > 0$ is the smallest number such that $t_{n+k} > 0$.

If $R_n \geq \bepsilon R_{n+k}$, then \lemref{sq shallow} implies that
$$
\|DF^{t_nR_n}|_{\hE_{p_{N-T_n}}}\| > \lambda_\mu^{\bepsilon t_nR_n}
\matsp{and}
\measuredangle(\hE_{p_{N-T_{n-1}}}, E_{p_{N-T_{n-1}}}^{v, n}) > \lambda_\mu^{\baeta R_n}.
$$
If $R_n < \bepsilon R_{n+k}$, then by $\eta$-homogeneity, we have
$$
\|DF^N|_{\hE_{p_0}}\| > \lambda_\mu^{(1+\eta)T_n}\|DF^{N-T_{n+k}}|_{\hE_{p_0}}\| > \lambda_\mu^{\bepsilon R_{n+k}}\lambda_\mu^{\bepsilon (N-T_{n+k})} > \lambda_\mu^{\bepsilon N}.
$$
\end{proof}


\section{Renormalization Convergence}\label{sec.converge}

\subsection{For unimodal maps}

Let $r \geq 3$ be an integer. Consider a $C^r$-unimodal map $f : I \to I$ with the critical value $v \in I$. For an integer $0 \leq s \leq r$ and a number $t > 0$, the $t$-neighborhood of $f$ with respect to the $C^s$-topology is denoted $\frN^s(f, t)$.

\begin{lem}\label{1d pert ren}
For $K \geq 1$ and $\bfb \geq 2$, there exists a uniform constant $t_0 = t_0(K, \bfb) > 0$ such that the following holds. Let $f \in \frU^r(K)$. Suppose $f$ is non-trivially renormalizable with return time $R \leq \bfb$, and $\cRod(f)$ has $\bfb$-bounded kneading. If $\tif \in \frN^s(f,t) \cap \frU^2$ with $2 \leq s < r$ and $t \in [0, t_0]$, then $\tif$ is valuably renormalizable with $\tau(\tif) = \tau(f)$. Moreover,
$$
\|\cRod(f)-\cRod(\tif)\|_{C^s} < Ct,
$$
where $C \geq 1$ is a uniform constant depending only on $K$, $\bfb$ and $\|f\|_{C^{s+1}}$.
\end{lem}

\begin{proof}
The renormalizability of $\tif$ such that $\tau(\tif) = \tau(f)$ follows immediately from \lemref{1d struct} and \propref{c1 bound geo}.

Denote the critical points of $f$ and $\tif$ by  $c = 0$ and $\tic$ respectively. Define
$$
I^1 := [f(c), f^R(c)]
\matsp{and}
\tiI^1 := [\tif(\tic), \tif^R(\tic)];
$$
and $f_1 := f^R|_{I^1}$ and $\tif_1 := \tif^R|_{\tiI^1}$. Let $S$ and $\tiS$ be the unique orientation-preserving affine maps on $\bbR$ such that $S\circ f_1 \circ S^{-1}, \tiS\circ\tif_1 \circ \tiS^{-1} \in \frU^s$.

By \lemref{vary compose}, we see that
$$
\|f_1-\tif_1\|_{C^s} < Ct.
$$
This implies immediately that $\|S - \tiS\| < C t$. The result follows.
\end{proof}

Consider the full renormalization attractor $\frA$ contained in the space $\frU^\omega$ of analytic unimodal maps. For an integer $\bfb \geq 2$, the compact invariant subset of $\frA$ consisting of all infinitely renormalizable unimodal maps with combinatorics of $\bfb$-bounded type is denoted $\frA_{\bfb}$.

The following is a consequence of the fact that $\frA_{\bfb}$ is a hyperbolic attractor for the renormalization operator $\cRod$ acting on $\frU^3$.

\begin{lem}\label{1d ren hyp}
Let $r \geq 3$ and $N \in \bbN$ be integers, and let $K \geq 1$ be a number. Suppose  $f \in \frU^r$ is $N$-times valuably renormalizable. Then for any $f^* \in \frA_{\bfb}$ with $\tau_N(f) = \tau_N(f^*)$, we have:
$$
\|\cRod^n(f) - \cRod^n(f^*)\|_{C^r} = C\rho^n\|f - f^*\|_{C^r}
\matsp{for}
1\leq n < N/2,
$$
where $\rho = \rho(\bfb) \in (0,1)$ is a universal constant and $C \geq 1$ is a uniform constant depending only on $\bfb$ and $\|f\|_{C^r}$.
\end{lem}

\subsection{For H\'enon-like maps}\label{subsec.ren conv henon}

Let $F$ be the infinitely regularly H\'enon-like renormalizable $C^{r+4}$-map considered in \secref{sec.cr bound}. For $n\geq n_0$, denote $F_n := \cR^n(F)$ and $f_n := \Piod(F_n)$. By \thmref{rescale}, we have $F_n \in \frHL^{r+3}_{\lambda_n}(\bfK)$, where $\bfK \geq 1$ is a uniform constant, and $\lambda_n := \lambda^{(1-\bepsilon)R_n}$. Moreover, by \thmref{cr bound 2d}, $\|F_n\|_{C^r}$ is uniformly bounded.

\begin{prop}[Shadowing Lemma]\label{shadow}
For $N \in \bbN$, there exists $n_1 = n_1(N) \in \bbN$ such that for all $n \geq n_1$, the map $f_n$ is $N$-times valuably renormalizable with
$
\tau_N(f_n) = \tau_N(F_n).
$
Moreover, we have
$$
\|f_{n+k} - \cRod^k(f_n)\|_{C^{r-1}} < C^k\lambda^{(1-\bepsilon)R_n}
\matsp{for}
1\leq k \leq N
$$
for some uniform constant $C \geq 1$.
\end{prop}

\begin{proof}
First, consider the case when $N=1$. The renormalizability of $f_n$ so that $\tau(f_n) = \tau(F_n)$ follows immediately from \lemref{1d struct}, and Propositions \ref{2d close 1d} and \ref{exp scale}.

Note that
$$
\|F_n - F_n\circ \Pi_h\|_{C^r} < \lambda^{(1-\bepsilon)R_n}.
$$
Since $\|F_n\|_{C^r}$ is uniformly bounded, \lemref{vary compose} implies that
$$
\|F_n^{r_n} - (F_n\circ \Pi_h)^{r_n}\|_{C^{r-1}} < C\lambda^{(1-\bepsilon)R_n}
$$
for some uniform constant $C \geq 1$. Thus,
$$
\|\Piod(F_n^{r_n}) - f_n^{r_n}\|_{C^{r-1}} < C \lambda^{(1-\bepsilon)R_n}.
$$
It follows that if $S$ and $\tiS$ are the unique orientation-preserving affine maps on $\bbR$ such that $S\circ \Piod(F_n^{r_n}) \circ S^{-1}, \tiS\circ f_n^{r_n} \circ \tiS^{-1} \in \frU^s$, then
$$
\|S - \tiS\| < C\lambda^{(1-\bepsilon)R_n}.
$$
Thus,
$$
\|f_{n+1} - \cRod(f_n)\|_{C^{r-1}} < C\lambda^{(1-\bepsilon)R_n}.
$$

Proceeding inductively, suppose that the result is true for all $1 \leq k < N$. In particular, we have
$$
\|f_{n+N-1} - \cRod^{N-1}(f_n)\|_{C^{r-1}} < C^{N-1}\lambda^{(1-\bepsilon)R_n}.
$$
By the above argument, $f_{n+N-1}$ is valuably renormalizable so that
$$
\tau(f_{n+N-1}) = \tau(F_{n+N-1}).
$$
If $n_1$ is sufficiently large, it follows from \lemref{1d pert ren} that $\cRod^{N-1}(f_n)$ is also valuably renormalizable, and
$$
\tau(\cRod^{N-1}(f_n)) = \tau(f_{n+N-1}).
$$

For $m \in \bbN$, we have
$$
\|f_{n+m} - \cRod(f_{n+m-1})\|_{C^{r-1}} < \lambda^{(1-\bepsilon)R_{n+m}}.
$$
Applying \lemref{1d pert ren} $0 \leq k < N$ times, we obtain
$$
\|\cRod^k(f_{n+m}) - \cRod^{k+1}(f_{n+m-1})\|_{C^{r-1}} < C^k\lambda^{(1-\bepsilon)R_{n+m}}.
$$
Thus,
\begin{align*}
\|f_{n+N} - \cRod^N(f_n)\|_{C^{r-1}} &\leq \sum_{k=0}^{N-1} \|\cRod^k(f_{n+N-k}) - \cRod^{k+1}(f_{n+N-(k+1)})\|_{C^{r-1}}\\
&< \sum_{k=0}^{N-1} C^k\lambda^{(1-\bepsilon) R_{n+N-k}}\\
&< O(C^N\lambda^{(1-\bepsilon) R_n}).
\end{align*}
\end{proof}

\begin{proof}[Proof of Theorem C]
Statements i), ii) and iii) are given by \thmref{crit chart}. Statement iv) is given by \thmref{cr bound 2d}. Hence, it remains to prove Statement v).

Suppose $r \geq 4$. Let $f^* \in \frA_{\bfb}$ so that
$$
\tau_\infty(f^*) = \tau_\infty(F) := [\tau(f_0), \tau(f_1), \ldots].
$$
Denote $f^*_n := \cRod^n(f^*)$ for $n \geq 0$.

Consider the constants $C \geq 1$ and $\rho \in (0,1)$ given in \lemref{1d ren hyp}. Choose $N \in \bbN$ sufficiently large so that $C\rho^N < \tirho < 1$. Let $n_1 = n_1(2N) \in \bbN$ be the number given in \propref{shadow}. Then for all $n \geq n_1$, we have
\begin{align*}
\|f_{n+N} - f^*_{n+N}\|_{C^{r-1}} &\leq \|f_{n+N} - \cRod^N(f_n)\|_{C^{r-1}} + \|\cRod^N(f_n) - \cRod^N(f_n^*)\|_{C^{r-1}}\\
&\leq O(\lambda^{(1-\bepsilon)R_n}) +\tirho\|f_n - f_n^*\|_{C^{r-1}}\\
&< \tirho'\|f_n - f_n^*\|_{C^{r-1}},
\end{align*}
for some uniform constant $\tirho' \in(0,1)$.
\end{proof}


\appendix


\section{Quantitative Estimates on Invariant Manifolds}\label{sec.pesin}

In this section, we summarize the results in \cite{CLPY2}. Let $r \geq 2$ be an integer, and consider a $C^{r+1}$-diffeomorphism $F : \cB \to F(\cB) \Subset \cB$, where $\cB \subset \bbR^2$ is a bounded domain. Let $\lambda, \epsilon \in (0,1)$ with $\bepsilon < 1$.

Let $p_0 \in \cB$ and $E^v_{p_0} \in \bbP^2_{p_0}$. For $m \in \bbZ$, decompose the tangent space at $p_m$ as
$$
\bbP^2_{p_m} = (E^v_{p_m})^\perp \oplus E^v_{p_m}.
$$
In this decomposition, we have
$$
D_{p_m}F =: \begin{bmatrix}
\alpha_m & 0\\
\zeta_m & \beta_m
\end{bmatrix},
$$
where $\alpha_m, \beta_m >0$ and $\zeta_m \in \bbR$.

For some $M, N \in \bbN \cup \{0, \infty\}$ and $L \geq 1$, suppose for $s \in \{0, 1\}$, we have
$$
L^{-1}\lambda^{(1+\epsilon)n} \leq \frac{\beta_{0} \ldots \beta_{{n-1}}}{(\alpha_0 \ldots \alpha_{n-1})^s}\leq L\lambda^{(1-\epsilon)n}
\matsp{for}
1 \leq n \leq N,
$$
and
$$
L^{-1}\lambda^{(1+\epsilon)n}\leq \frac{\beta_{{-n}} \ldots \beta_{{-1}}}{(\alpha_{{-n}} \ldots \alpha_{{-1}})^s}
\leq L\lambda^{(1-\epsilon)n}
\matsp{for}
1 \leq n \leq M.
$$
Then we say that $p_0$ is {\it $(M, N)$-times $(L, \epsilon, \lambda)$-regular along $E^v_{p_0}$}.

\begin{prop}[Growth in irregularity]\cite[Proposition 5.5]{CLPY2}\label{grow irreg}
For $-M \leq m \leq N$, let $\cL_{p_m} \geq 1$ be the minimum value such that $p_m$ is $(M+m, N-m)$-times $(\cL_{p_m}, \epsilon, \lambda)$-regular along $E^v_{p_m}$. Then
$$
\cL_{p_m} < \bL\lambda^{-\bepsilon |m|}.
$$
\end{prop}

\subsection{Quasi-Linearization}

For $w, l > 0$, denote
$$
\bbB(w, l) := (-w, w)\times (-l, l) \subset \bbR^2
\matsp{and}
\bbB(l) := \bbB(l, l).
$$

\begin{thm}[Regular charts]\cite[Theorem 6.1]{CLPY2}\label{reg chart}
There exists a uniform constant
$$
C = C(\|DF\|_{C^r}, \lambda^{-\epsilon}) \geq 1
$$
such that the following holds. For $-M \leq m \leq N$, let
$$
\omega := \frac{\lambda^{1-\epsilon}}{1-\lambda^{1-\epsilon}}\cdot \|DF^{-1}\|\cdot\|DF\|
\matsp{and}
\cK_{p_m} := \bL(1+\omega)^5\|DF^{-1}\|\lambda^{1-\bepsilon |m|}.
$$
Define
$$
U_{p_m} := \bbB(l_{p_m})
\matsp{where}
l_{p_m} := \lambda^{1+\bepsilon}(C\cK_{p_m})^{-1}.
$$
Then there exists a $C^r$-chart $\Phi_{p_m} : (\cU_{p_m}, p_m) \to (U_{p_m}, 0)$ such that $D\Phi_{p_m}(E^v_{p_m}) = E^{gv}_0$, 
$$
\|D\Phi_{p_m}^{-1}\|_{C^{r-1}} < C(1+\omega)
\comma
\|D\Phi_{p_m}\|_{C^s} < C\cK_{p_m}^{s+1}
\matsp{for}
0 \leq s <r,
$$
and the map $\Phi_{p_{m+1}} \circ F|_{\cU_{p_m}} \circ \Phi_{p_m}^{-1}$ extends to a globally defined $C^r$-diffeomorphism $F_{p_m} : (\bbR^2, 0) \to (\bbR^2, 0)$ satisfying the following properties:
\begin{enumerate}[i)]
\item $\displaystyle \|DF_{p_m}\|_{C^{r-1}} \leq \|DF\|_{C^r}$;
\item we have
$$
D_0F_{p_m} =\begin{bmatrix}
a_m & 0\\
0 & b_m
\end{bmatrix},
\matsp{where}
\lambda^{\bepsilon} < a_m < \lambda^{-\bepsilon}
\matsp{and}
\lambda^{1+\bepsilon} < b_m < \lambda^{1-\bepsilon}.
$$
\item $\|D_z F_{p_m} - D_0F_{p_m}\|_{C^0} < \lambda^{1+\bepsilon}$ for $z \in \bbR^2$;
\item we have
$$
F_{p_m}(x,y) = (f_{p_m}(x), e_{p_m}(x,y))
\matsp{for}
(x,y) \in \bbR^2,
$$
where $f_{p_m}:(\bbR, 0) \to (\bbR, 0)$ is a $C^r$-diffeomorphism, and $e_{p_m} : \bbR^2 \to \bbR$ is a $C^r$-map such that for all $0 \leq s \leq r$, we have
$$
\partial_x^s e_{p_m}(\cdot, y) \leq \|DF\|_{C^r} |y|
\matsp{for}
y \in \bbR.
$$
\end{enumerate}
\end{thm}

The construction in \thmref{reg chart} is referred to as {\it a Q-linearization of $F$ along the $(M,N)$-orbit of $p_0$ with vertical direction $E^v_{p_0}$}. For $-M \leq m \leq N$, we refer to $l_{p_m}$, $\cU_{p_m}$, $\Phi_{p_m}$ and $F_{p_m}$ as a {\it regular radius}, a {\it regular neighborhood}, a {\it regular chart} and a {\it Q-linearized map at $p_m$} respectively.

For $p \in \bbR^2$ and $t > 0$, let
$$
\bbD_p(t) := \{\|q - p\| < t\}.
$$

\begin{lem}\cite[Lemma 6.2]{CLPY2}\label{size reg nbh}
For $-M \leq m \leq N$, we have
$$
\cU_{p_m} \supset \bbD_{p_m}\left(\frac{\lambda^{1+\bepsilon}}{C^2\cK_{p_m}^2}\right),
$$
where $C, \cK_{p_m} \geq 1$ are given in \thmref{reg chart}.
\end{lem}

\subsection{$C^1$-estimates}

\begin{prop}[Jacobian bounds]\cite[Proposition 6.14]{CLPY2}\label{jac bound}
We have
$$
\bL^{-1}\lambda^{(1+\bepsilon)n} \leq \Jac_{p_0}F^n \leq \bL\lambda^{(1-\bepsilon)n}
\matsp{for}
1 \leq n \leq N,
$$
and
$$
\bL^{-1}\lambda^{-(1-\bepsilon)n}\leq \Jac_{p_0}F^{-n} \leq \bL\lambda^{-(1+\bepsilon)n}
\matsp{for}
1 \leq n \leq M.
$$
\end{prop}

\begin{prop}[Derivative bounds]\cite[Proposition 6.15]{CLPY2}\label{ext deriv bound}
Let $C \geq 1$ and $\omega >0$ be the uniform constants given in \thmref{reg chart}. For $E_{p_0} \in \bbP^2_{p_0}$, we have
$$
\frac{\lambda^{(1+\bepsilon)n}}{C\bL(1+\omega)^2} \leq \|DF^n|_{E_{p_0}}\| \leq C(1+\omega)^2\lambda^{-\bepsilon n}
\matsp{for}
1 \leq n \leq N,
$$
and
$$
\frac{\lambda^{\bepsilon n}}{{C\bL}(1+\omega)^2}\leq \|DF^{-n}|_{E_{p_0}}\| \leq C(1+\omega)^2\lambda^{-(1+\bepsilon)n}
\matsp{for}
1 \leq n \leq M.
$$
\end{prop}

Consider the sequence of Q-linearized maps $\{F_{p_m}\}_{-M}^N$ given in \thmref{reg chart}. For $1 \leq n \leq N-m$, we denote
\begin{equation}\label{eq.compose}
F_{p_m}^n = (f_{p_m}^n, e_{p_m}^n):= F_{p_{m +n-1}} \circ \ldots \circ F_{p_{m+1}}\circ F_{p_m}.
\end{equation}

\begin{prop}\cite[Proposition 6.4]{CLPY2}\label{lin comp}
For $-M \leq m \leq N$ and $0 \leq n \leq N-m$, consider the $C^r$-diffeomorphism $F_{p_m}^n$ given in \eqref{eq.compose}. Let $z = (x, y) \in U_{p_m}$, and suppose that
$$
z_i = (x_i, y_i) := F_{p_m}^i(z) \in U_{p_{m+i}}
\matsp{for}
0 \leq i \leq n.
$$
Denote
$$
D_zF_{p_m}^n =: \begin{bmatrix}
a_m^n(z) & 0\\
c_m^n(z) & b_m^n(z)
\end{bmatrix}.
$$
Define
$$
\chl_h := \sup_n n\lambda^{\bepsilon n} < \infty
\matsp{and}
\chl_v := (1-\lambda^{1-\bepsilon})^{-1};
$$
and
$$
\chi_h := \exp\left(\frac{\chl_h\|F\|_{C^3}}{\lambda^{\bepsilon}}\right)
\matsp{and}
\chi_v := \exp\left(\frac{(\chl_h+\chl_v)\|F\|_{C^3}}{\lambda^{\bepsilon}}\right).
$$
Then
$$
\frac{1}{\chi_h}\leq \frac{a_m^n(z)}{a_m^n(0)} \leq \chi_h
\comma
\frac{1}{\chi_v}\leq \frac{b_m^n(z)}{b_m^n(0)} \leq \chi_v
\matsp{and}
\|\gamma_m^n\| < \lambda^{(1-\bepsilon)n}.
$$
\end{prop}

For $-M \leq m \leq N$ and $q \in \cU_{p_m}$, write $z := \Phi_{p_m}(q)$. The {\it vertical/horizontal direction at $q$ in $\cU_{p_m}$} is defined as $E^{v/h}_q := D\Phi_{p_m}^{-1}(E^{gv/gh}_z)$. By the construction of regular charts in \thmref{reg chart}, vertical directions are invariant under $F$ (i.e. $DF(E_q^v) = E_{F(q)}^v$ for $q \in \cU_{p_m}$). Note that the same is not true for horizontal directions.

\begin{prop}\cite[Proposition 6.5]{CLPY2}\label{chart cons}
For $-M \leq m \leq N$ and $q \in \cU_{p_m}$, we have
$$
\frac{1}{\sqrt{2}} \leq \frac{\|D\Phi_{p_m}|_{E_z^{v/h}}\|}{\|D\Phi_{p_m}|_{E_{p_m}^{v/h}}\|} \leq \sqrt{2}.
$$
\end{prop}

\begin{cor}\cite[Corollary 6.6]{CLPY2}\label{ver hor cons}
For some $-M \leq m_0 \leq N$, let $q_{m_0} \in \cU_{p_{m_0}}$. Suppose for $m_0 \leq m \leq m_1 \leq N$, we have $q_m \in \cU_{p_m}$. Let
$$
\hE^h_{q_m} := DF^{m-m_0}(E^h_{q_{m_0}}).
$$
Then for $m_0 \leq m'\leq m_1$, we have
$$
\frac{1}{2\chi_h}\leq \frac{\|DF^{m'-m}|_{\hE^h_{q_m}}\|}{\|DF^{m'-m}|_{E^h_{p_m}}\|} \leq 2\chi_h
\matsp{and}
\frac{1}{2\chi_v}\leq \frac{\|DF^{m'-m}|_{E^v_{q_m}}\|}{\|DF^{m'-m}|_{E^v_{p_m}}\|} \leq 2\chi_v,
$$
where $\chi_h$ and $\chi_v$ are constants given in \propref{lin comp}.
\end{cor}

\begin{prop}[Vertical alignment of forward contracting directions]\cite[Proposition 6.8]{CLPY2}\label{vert angle shrink}
Let $q_0 \in \cU_{p_0}$ and $\tiE_{q_0}^v \in \bbP^2_{q_0}$. Suppose $q_i \in \cU_{p_ i}$ for $0 \leq i \leq n \leq N$, and that
$$
\nu := \|DF^n|_{\tiE_{q_0}^v}\| < \frac{\lambda^{\bepsilon n}}{\chi_h(2+\omega)^3 \bC},
$$
where $C, \omega, \chi_h \geq 1$ are uniform constants given in \thmref{reg chart} and \propref{lin comp}. Denote $z_0 := \Phi_{p_0}(q_0)$ and $\tiE_{z_0}^v := D\Phi_{p_0}(\tiE^v_{q_0})$. Then
$$
\measuredangle(\tiE^v_{z_0}, E_{z_0}^{gv}) < \chi_h (1+\omega)\bC\lambda^{-\bepsilon n} \nu.
$$
\end{prop}

\begin{prop}[Horizontal alignment of backward neutral directions]\cite[Proposition 6.9]{CLPY2}\label{hor angle shrink}
Let $q_0 \in \cU_{p_0}$ and $\tiE_{q_0}^h \in \bbP^2_{q_0}$. Suppose  $q_{-i} \in \cU_{p_ {-i}}$ for $0 \leq i \leq n \leq M$, and that
$$
\mu := \|DF^{-n}|_{\tiE^h_{q_0}}\| < \frac{1}{\chi_v(2+\omega)^3\bC \lambda^{(1-\bepsilon)n}}.
$$
Denote
$$
z_0 := \Phi_{p_0}(q_0)
\comma
\tiE_{z_0}^h := D\Phi_{p_0}(\tiE^h_{q_0})
\matsp{and}
\hE_{z_0}^h := D\Phi_{p_0} \circ F^n(E^h_{q_{-n}}).
$$
Then
$$
\measuredangle(\tiE_{q_0}^h, \hE_{q_0}^h) < \chi_v(1+\omega)\bC \lambda^{(1-\bepsilon)n}\cdot \mu.
$$
\end{prop}

The {\it $n$-times truncated regular neighborhood of $p_0$} is defined as
$$
\cU_{p_m}^n := \Phi_{p_m}^{-1}\left(U_{p_m}^n\right)\subset \cU_{p_m},
\matsp{where}
U_{p_0}^n := \bbB\left(\lambda^{\bepsilon n}l_{p_m}, l_{p_m}\right).
$$
The purpose of truncating a regular neighborhood is to ensure that its iterated images stay inside regular neighborhoods.

\begin{lem}\cite[Lemma 6.10]{CLPY2}\label{trunc neigh fit}
Let $-M \leq m \leq N$ and $0 \leq n \leq N-m$. We have $F^i(\cU_{p_m}^n) \subset \cU_{p_{m+i}}$ for $0 \leq i \leq n$.
\end{lem}

\begin{prop}\cite[Propositions B.5 and B.6]{CLPY2}\label{cr lin comp}
There exists a uniform constant $K = K(\|DF\|_{C^r}, \lambda, \epsilon, r) \geq 1$ such that the following result holds. For $-M \leq m \leq N$ and $0 \leq n \leq N-m$, consider the $C^r$-maps $f_{p_m}^n$ and $e_{p_m}^n$ given in \eqref{eq.compose}. Then we have
$$
\|Df_{p_m}^n\|_{C^{r-1}} < K\lambda^{-\bepsilon n}
\matsp{and}
\|De_{p_m}^n\|_{C^{r-1}} < K\lambda^{(1-\bepsilon)n}.
$$
\end{prop}

\subsection{$C^r$-estimates}

Let $g : \bbR \to \bbR$ be a $C^r$-function. The curve
$$
\Gamma_g := \{(x, g(x)) \; x \in \bbR\}
$$
is the {\it horizontal graph of $g$}. Let $H : \bbR^2 \to \bbR^2$ be a $C^r$-diffeomorphism. Suppose that there exists a $C^r$-function $H_*(g) : \bbR \to \bbR$ such that $H(\Gamma_g) = \Gamma_{H_*(g)}$. Then $H_*(g)$ and $\Gamma_{H_*(g)}$ are referred to as the {\it horizontal graph transform of $g$} and $\Gamma_g$ {\it by $H$} respectively.

\begin{prop}[$C^r$-convergence of horizontal graphs]\cite[Proposition 4.5]{CLPY2}\label{back dt}\label{for gt}
Let $g : \bbR \to \bbR$ be a $C^r$-map with $\|g'\|_{C^{r-1}} < \infty$. For $-M \leq m \leq N$ and $1 \leq n \leq N-m$, consider the graph transform $\tig := (F_{p_m}^n)_*(g)$. Then
$$
\|\tig'\|_{C^{r-1}} < C\lambda^{(1-\bepsilon)n}(1+\|g'\|_{C^{r-1}})^r
$$
where $C = C(\bfC, \lambda,\epsilon, r) \geq 1$ is a uniform constant.
\end{prop}

For $p \in \bbR^2$ and $u \in \bbR$, let $E_p^u \in \bbP_p^2$ be the tangent direction at $p$ given by
$$
E_p^u := \{r(u, 1) \; | \; r \in \bbR\}.
$$
Let $\xi : \bbR^2 \to \bbR$ be a $C^{r-1}$-map. The direction field
$$
\cE_\xi := \{E_p^{\xi(p)} \; | \; p \in \bbR^2\}
$$
is the {\it vertical direction field of $\xi$}. Let $H : \bbR^2 \to \bbR^2$ be a $C^r$-diffeomorphism. Suppose that there exists a $C^{r-1}$-map $H^*(\xi) : \bbR^2 \to \bbR$ such that $DH^{-1}(\cE_\xi) =\cE_{H^*(\xi)}$. Then $H^*(\xi)$ and $\cE_{H^*(\xi)}$ are referred to as the {\it vertical direction field transform of $\xi$} and $\cE_\xi$ {\it by $H$} respectively.

\begin{prop}[Backward vertical direction field transform]\cite[Proposition 4.6]{CLPY2}\label{back dt}
There exist uniform constants $C, \tiC \geq 1$ depending only on $\bfC, \lambda, \epsilon, r$  such that the following holds. Let $\xi : \bbR^2 \to \bbR$ be a $C^{r-1}$-map with $\|\xi\|_{C^{r-1}} < \infty$. For $-M \leq m < N$ and $0 \leq n \leq M+m$, consider the vertical direction transform 
$$
\tixi := (F_{p_m}^n)^*(\xi)|_{\bbR \times (-1, 1)}.
$$
Suppose
$$
C\lambda^{(1-\bepsilon)n}(1+\|\xi\|_{C^{r-1}}) < 1.
$$
Then
$$
\|\tixi\|_{C^{r-1}} < \tiC\lambda^{(1-\bepsilon)n}\|\xi\|_{C^{r-1}}.
$$
\end{prop}

\subsection{Stable and center manifolds}

For $-M \leq m \leq N$, define the {\it local vertical} and {\it horizontal manifold at $p_m$} as
$$
W^v_{\loc}(p_m) := \Phi_{p_m}^{-1}(\{(0, y) \in U_{p_m}\})
\matsp{and}
W^h_{\loc}(p_m) := \Phi_{p_m}^{-1}(\{(x,0) \in U_{p_m}\})
$$
respectively.

If $N = \infty$, then \propref{vert angle shrink} implies that $E_{p_0}^v$ is the unique direction along which $p_0$ is infinitely forward regular. In this case, we denote $E_{p_0}^{ss} := E_{p_0}^v$, and refer to this direction as the {\it strong stable direction at $p_0$}. Additionally, we define the {\it strong stable manifold of $p_0$} as
$$
W^{ss}(p_0) := \left\{q_0 \in \Omega \; | \; \limsup_{n \to \infty}\frac{1}{n}\log\|q_n - p_n\| < (1-\epsilon)\log\lambda\right\}.
$$

\begin{thm}[Canonical strong stable manifold]\cite[Theorem 6.13]{CLPY2}\label{stable}
If $N = \infty$, then
$$
W^{ss}(p_0) := \bigcup_{n=0}^\infty F^{-n}(W^v_{\loc}(p_n)).
$$
Consequently, $W^{ss}(p_0)$ is a $C^{r+1}$-smooth manifold.
\end{thm}

If $M = \infty$, then \propref{hor angle shrink} implies that $E_{p_0}^h$ is the unique direction along which $p_0$ is infinitely backward regular. In this case, we denote $E_{p_0}^c := E_{p_0}^h$, and refer to this direction as the {\it center direction at $p_0$}. Moreover, we define the {\it (local) center manifold at $p_0$} as
$$
W^c(p_0) := \Phi_{p_0}^{-1}(\{(x, 0) \in U_{p_0}\}).
$$
Unlike strong stable manifolds, center manifolds are not canonically defined. However, the following result states that it still has a canonical jet.

\begin{thm}[Canonical jets of center manifolds]\cite[Theorem 6.16]{CLPY2}\label{center jet}
Suppose $M = \infty$. Let $\Gamma_0: (-t, t) \to \cU_{p_0}$ be a $C^{r+1}$-curve parameterized by its arclength such that $\Gamma_0(0) = p_0$, and for all $n \in \bbN$, we have
$$
\|DF^{-n}|_{\Gamma_0'(t)}\| < \lambda^{-\frac{(1-\bepsilon)n}{r+1}}
\matsp{for}
|t| < \lambda^{\epsilon n}.
$$
Then $\Gamma_0$ has a degree $r+1$ tangency with $W^c(p_0)$ at $p_0$.
\end{thm}

\subsection{Horizontal regularity}

We say that $p \in \cB$ is {\it $N$-times forward horizontally $(L, \epsilon, \lambda)$-regular along $E_p^{h,+} \in \bbP_p^2$} if, for $s\in \{1, 2\}$, we have
\begin{equation}\label{eq:hor for reg}
L^{-1}\lambda^{(1+\epsilon)n} \leq \frac{\Jac_p F^n}{\|D_pF^n|_{E_p^{h,+}}\|^s} \leq L\lambda^{(1-\epsilon)n}
\matsp{for}
1 \leq n \leq N.
\end{equation}
Similarly, we say that $p$ is {\it $M$-times backward horizontally $(L, \epsilon, \lambda)$-regular along $E_p^{h,-} \in \bbP_p^2$} if, for $s \in \{1, 2\}$, we have
\begin{equation}\label{eq:hor back reg}
L^{-1}\lambda^{(1+\epsilon)n}  \leq  \frac{\|D_pF^{-n}|_{E_p^{h,-}}\|^s}{\Jac_p F^{-n}} \leq L\lambda^{(1-\epsilon)n}
\matsp{for}
1 \leq n \leq M.
\end{equation}
If \eqref{eq:hor for reg} and \eqref{eq:hor back reg} hold with $E_p^h := E_p^{h,+} = E_p^{h,-}$, then $p$ is {\it $(M, N)$-times horizontally $(L, \epsilon, \lambda)$-regular along $E_p^h$}.

\begin{prop}[Horizontal vs vertical forward regularity]\cite[Proposition 5.2]{CLPY2}\label{transverse for reg}
If $p$ is $N$-times forward horizontally $(L, \epsilon, \lambda)$-regular along $E_p^h \in \bbP_p^2$, then there exists $E_p^v \in \bbP_p^2$ such that $p$ is $N$-times forward $(\bL, \bepsilon, \lambda)$-regular along $E_p^v$.
\end{prop}

\begin{prop}[Horizontal vs vertical backward regularity]\cite[Proposition 5.3]{CLPY2}\label{transverse back reg}
Suppose $p$ is $M$-times backward horizontally $(L, \epsilon, \lambda)$-regular along $E_p^h \in \bbP_p^2$. Let $E_p^v \in \bbP_p^2 \setminus \{E_p^h\}$. If $\measuredangle (E_p^h, E_p^v) > \theta$, then the point $p$ is $M$-times backward $(\bL/\theta^2, \bepsilon, \lambda)$-regular along $E_p^v$.
\end{prop}


\section{Classification of Fixed Points}\label{sec.fixed}

Let $F : \cB \to F(\cB) \Subset \cB$ be a dissipative diffeomorphism defined on a Jordan domain $\cB \subset \bbR^2$. Suppose that $q_0 \in \cB$ is an isolated fixed point for $F$, and that $\lambda_-, \lambda_+ \in \bbR$ with $0 < |\lambda_-| < |\lambda_+|$ are the eigenvalues of $D_{q_0}F$. If $|\lambda_+| \geq 1$, then $q_0$ has a well-defined invariant local center manifold $W^c_{\loc}(q_0)$. In this case, we classify $q_0$ as:
\begin{itemize}
\item a {\it saddle with reflection} if the branches of $W^c_{\loc}(q_0) \setminus \{q_0\}$ alternate  and are repelling;
\item a {\it saddle with no reflection} if both branches of $W^c_{\loc}(q_0) \setminus \{q_0\}$ are fixed and repelling;
\item a {\it saddle-node} if both branches of $W^c_{\loc}(q_0) \setminus \{q_0\}$ are fixed and one is repelling while the other is attracting.
\end{itemize}

The {\it index of $q_0$}, denoted $\Index(q_0)$, is defined as the winding number of the vector field
$$
\Delta_{q_0} F(p) := F(p-x_0)-(p-q_0),
$$
and can be determined based on the type of $q_0$ as follows:
\begin{itemize}
\item $\Index(q_0) = 1$ if $q_0$ is a sink or a saddle with reflection;
\item $\Index(q_0) = 0$ if $q_0$ is a saddle-node; or
\item $\Index(q_0) = -1$ if $q_0$ is a saddle with no reflection.
\end{itemize}

\begin{prop}\label{exist saddle}
Let $F : \cB \to F(\cB) \Subset \cB$ be a dissipative diffeomorphism defined on a Jordan domain $\cB \subset \bbR^2$. Suppose that there exists an $R$-periodic Jordan subdomain $\cB^1 \Subset \cB$ for some integer $R \geq 2$. Then there exists a $r$-periodic saddle point in $\cB$ for some integer $r$ that divides $R$.
\end{prop}

\begin{proof}
If $F^R$ has a non-isolated fixed point $q_0$, then $q_0$ must have an indifferent eigenvalue, and hence is a type of saddle.

Suppose that all fixed points of $F^R$ are isolated. By the classical Lefschetz formula, when the number of fixed points is finite, the sum of the index of the fixed points in the disc is equal to $1$. Observe that $F$ has at least one fixed point and one $R$-periodic orbit in $\cB$. Hence, not all fixed points of $F^R$ can be sinks.
\end{proof}


\section{Elementary $C^r$-Estimates}

\begin{lem}\label{vary compose}
Let $d \in\bbN$. Consider maps $H_0, \tiH_0 : U \to V$ and $H_1, \tiH_1 : V \to \bbR^d$ defined on domains $U, V \subset \bbR^d$. Suppose $H_0, \tiH_0, \tiH_1$ are $C^{r-1}$ and $H_1$ is $C^r$; and
$$
\|H_i-\tiH_i\|_{C^{r-1}} < \delta
\matsp{for}
i \in \{0, 1\}.
$$
Then we have
$$
\|H_1 \circ H_0 - \tiH_1\circ \tiH_0\|_{C^{r-1}} < \delta P(\|H_1\|_{C^r}, \|\tiH_0\|_{C^{r-1}}),
$$
where $P$ is a two-variable homogeneous polynomial of degree $r$.
\end{lem}

\begin{proof}
Let $d_i := H_i-\tiH_i$. A straightforward computation shows that
\begin{align*}
H_1 \circ H_0 &= H_1 \circ (\tiH_0 + d_0)\\
&= H_1\circ \tiH_0 +O(\|DH_1\circ \tiH_0\|\|d_0\|)\\
&= \tiH_1 \circ \tiH_0 + d_1\circ \tiH_0 + O(\|DH_1\circ \tiH_0\|\|d_0\|).
\end{align*}
The result follows.
\end{proof}

\begin{lem}\cite[Lemma C.3]{CLPY3}\label{factor}
For $r \geq 4$, let $f : I \to f(I)$ be a $C^r$-map defined on an interval $0 \in I \subset \bbR$ such that $f(0) = 0 = f'(0)$ and $f''(0) = \kappa > 0$. Then there exists a $C^r$-diffeomorphism $\psi_f : I \to \psi_f(I)$ such that $f(x) = \kappa \cdot (\psi_f(x))^2$, and $\|\psi_f^{\pm 1}\|_{C^{r-3}} < C$ for some uniform constant $C = C(\|f\|_{C^r},\kappa, r) >0$.
\end{lem}

Let $g : I \to J$ be a $C^1$-diffeomorphism between two intervals $I, J \subset \bbR^2$. Define the {\it zoom-in operator} $\bfZ$ by
$$
\bfZ(g)(t) := |J|^{-1}\cdot g(|I| t).
$$
Note that $\bfZ(g) : [0, 1] \to [0,1]$.

\begin{lem}\cite[Lemma 5]{AvdMMa}\label{avila lem 1}
Let $\phi : U \to \phi(U)$ be a $C^r$-diffeomorphism defined on a domain $U \subset \bbR$. Then there exists a uniform constant
$$
K = K(\|\phi\|_{C^r}, \|\phi''/\phi'\|_{C^0}) \geq 1
$$
such that for any interval $I \subset U$, we have
$$
\|\bfZ(\phi|_{I})-\Id\|_{C^r} \leq K|I|.
$$
\end{lem}

\begin{lem}\cite[Lemma 6]{AvdMMa}\label{avila lem 2}
For $1\leq i \leq n$, let $\phi_i : [0, 1] \to [0, 1]$ be a $C^r$-diffeomorphism such that
$$
\sum_{i=1}^n\|\phi_i - \Id\|_{C^r} = O(1).
$$
Then
$$
\|\phi_n\circ\ldots\circ\phi_1\|_{C^r} = O(1).
$$
\end{lem}


\end{document}